\documentclass[]{amsart}

\usepackage{amssymb}
\usepackage{amsfonts}
\usepackage{amsmath}
\usepackage{amsthm}
\usepackage{mathtools}
\usepackage{graphicx}
\usepackage{mathrsfs}
\usepackage{dsfont}
\usepackage{amscd}
\usepackage{multirow}
\usepackage[all]{xy}
\normalfont
\usepackage[T1]{fontenc}
\usepackage{calligra}
\usepackage{verbatim}
\usepackage[usenames,dvipsnames]{color}
\usepackage[colorlinks=true,linkcolor=Blue,citecolor=Violet]{hyperref}
\usepackage{enumerate}
\usepackage{slashed}
\usepackage{nicematrix}
\usepackage{microtype}

\newcommand{\N}{{\mathds{N}}}
\newcommand{\Z}{{\mathds{Z}}}
\newcommand{\Q}{{\mathds{Q}}}
\newcommand{\R}{{\mathds{R}}}
\newcommand{\C}{{\mathds{C}}}
\newcommand{\T}{{\mathds{T}}}
\newcommand{\U}{{\mathds{U}}}
\newcommand{\D}{{\mathfrak{D}}}
\newcommand{\A}{{\mathfrak{A}}}
\newcommand{\B}{{\mathfrak{B}}}

\newcommand{\bigslant}[2]{{\raisebox{.2em}{$#1$}\left/\raisebox{-.2em}{$#2$}\right.}}
\newcommand{\Nbar}{\overline{\N}}
\newcommand{\Lip}{{\mathsf{L}}}
\newcommand{\TLip}{{\mathsf{T}}}

\newcommand{\Hilbert}{{\mathscr{H}}}

\newcommand{\dpropinquity}[1]{{\mathsf{\Lambda}^\ast_{#1}}}

\newcommand{\dmodpropinquity}[1]{{\mathsf{\Lambda}^{\ast\mathsf{mod}}_{#1}}}
\newcommand{\dmetpropinquity}[1]{{\mathsf{\Lambda}^{\ast\mathsf{met}}_{#1}}}

\newcommand{\spectralpropinquity}{{\mathsf{\Lambda}^{\mathsf{spec}}}}

\newcommand{\Kantorovich}[1]{{\mathsf{mk}_{#1}}}

\newcommand{\Haus}[1]{{\mathsf{Haus}_{#1}}}

\newcommand{\StateSpace}{{\mathscr{S}}}

\newcommand{\MongeKant}{{Mon\-ge-Kan\-to\-ro\-vich metric}}

\newcommand{\gQVB}{metrized quantum vector bundle}

\newcommand{\mcc}[3]{{\mathrm{metCor}\left({#1},{#2},{#3}\right)}}

\newcommand{\qcms}{quantum compact metric space}

\newcommand{\unit}{1}

\newcommand{\sa}[1]{{\mathfrak{sa}\left({#1}\right)}}

\newcommand{\inner}[3]{{\left<{#1},{#2}\right>_{#3}}}

\newcommand{\dom}[1]{{\operatorname*{dom}\left({#1}\right)}}

\newcommand{\norm}[2]{\left\|{#1}\right\|_{#2}}

\newcommand{\range}[1]{{\mathrm{ran}\left({#1}\right)}}
\newcommand{\grad}[2]{{\nabla_{#1}{#2}}} 

\newcommand{\CDN}{{\mathsf{DN}}}
\newcommand{\TDN}{{\mathsf{TN}}}
\newcommand{\cocycle}[1]{{\mathrm{\varsigma}_{#1}}}
\newcommand{\worknote}[1]{}
\newcommand{\opnorm}[3]{{\left|\mkern-1.5mu\left|\mkern-1.5mu\left| {#1} \right|\mkern-1.5mu\right|\mkern-1.5mu\right|_{#3}^{#2}}}

\newcommand{\tunnelmagnitude}[2]{{\mu\left({#1}\middle\vert{#2}\right)}}

\newcommand{\tunnelextent}[1]{{\chi\left({#1}\right)}}

\newcommand{\alg}[1]{{\mathfrak{#1}}}

\newcommand{\module}[1]{{\mathscr{#1}}}
\newcommand{\qt}[2]{{C^\ast\left( \Z^d_{#1}, {#2} \right)}}
\newcommand{\qtd}[3]{{C^\ast\left( \Z^{#1}_{#2},{#3} \right)}}

\newcommand{\dil}[1]{{\mathrm{dil}\left({#1}\right)}}

\newcommand{\length}[1]{{\mathsf{len}\left({#1}\right)}}

\theoremstyle{plain}
\newtheorem{theorem}{Theorem}[section]

\newtheorem{corollary}[theorem]{Corollary}

\newtheorem{lemma}[theorem]{Lemma}

\newtheorem{theorem-definition}[theorem]{Theorem-Definition}

\theoremstyle{definition}
\newtheorem{definition}[theorem]{Definition}

\newtheorem{example}[theorem]{Example}

\newtheorem{convention}[theorem]{Convention}
\newtheorem{hypothesis}[theorem]{Hypothesis}

\theoremstyle{remark}

\newtheorem{remark}[theorem]{Remark}

\newtheorem{notation}[theorem]{Notation}

\renewcommand{\geq}{\geqslant}
\renewcommand{\leq}{\leqslant}

\newcommand{\conv}[1]{*_{#1}\,}

\newcommand{\Dirac}{{\slashed{D}}}

\numberwithin{equation}{section}

\allowdisplaybreaks[4]

\begin{document}

\title[]{Convergence of Spectral Triples on Fuzzy Tori to Spectral Triples on Quantum Tori}
\author{Fr\'{e}d\'{e}ric Latr\'{e}moli\`{e}re}
\email{frederic@math.du.edu}
\urladdr{http://www.math.du.edu/\symbol{126}frederic}
\address{Department of Mathematics \\ University of Denver \\ Denver CO 80208}

\date{\today}

\begin{abstract}
  Fuzzy tori are finite dimensional C*-algebras endowed with an appropriate notion of noncommutative geometry inherited from an ergodic action of a finite closed subgroup of the torus, which are meant as finite dimensional approximations of tori and more generally, quantum tori. A mean to specify the geometry of a noncommutative space is by constructing over it a spectral triple. We prove in this paper that we can construct spectral triples on fuzzy tori which, as the dimension grows to infinity and under other natural conditions, converge to a natural spectral triple on quantum tori, in the sense of the spectral propinquity. This provides a formal assertion that indeed, fuzzy tori approximate quantum tori, not only as quantum metric spaces, but as noncommutative differentiable manifolds --- including convergence of the state spaces as metric spaces and of the quantum dynamics generated by the Dirac operators of the spectral triples, in an appropriate sense.
\end{abstract}
\maketitle

\let\oldtocsection=\tocsection
\let\oldtocsubsection=\tocsubsection
\let\oldtocsubsubsection=\tocsubsubsection
 
\renewcommand{\tocsection}[2]{\hspace{0em}\oldtocsection{#1}{#2}}
\renewcommand{\tocsubsection}[2]{\hspace{2em}\oldtocsubsection{#1}{#2}}
\renewcommand{\tocsubsubsection}[2]{\hspace{2.5em}\oldtocsubsubsection{#1}{#2}}

\tableofcontents


\section{Introduction}

\subsection{The Problem}

Matrix models for quantum field theories and string theories (in particular, involving the geometry of so-called compactified dimensions, or the geometry on $D$-branes) have become an interesting tool for the study of such fundamental questions as the search for a quantum theory of gravitation, e.g. \cite{Kimura01,Schreivogl13,Barrett15,Connes97,Seiberg99}. The asymptotic behaviors of such models, as the dimension of the matrix algebras involved grows to infinity, is of central interest. We have led a research program where the study of these asymptotic behaviors uses the formalism of convergence for certain distance functions on quantum spaces, starting with a distance on the class of {\qcms s} \cite{Connes89,Rieffel98a,Rieffel99,Rieffel00,Rieffel10b,Latremoliere12b,Latremoliere13,Latremoliere13b,Latremoliere15} which is a noncommutative analogue of the \emph{Gromov-Hausdorff distance}. In particular, we have recently introduced in \cite{Latremoliere18g} a generalization of the Gromov-Hausdorff distance on the class of metric \emph{spectral triples}, which are the structures introduced by Connes \cite{Connes89,Connes} to describe the noncommutative analogues of Riemannian manifolds. Our distance on metric spectral triples is called the \emph{spectral propinquity} \cite{Latremoliere18g}. In this paper, we construct metric spectral triples on \emph{fuzzy tori}, which are a particularly relevant family of matrix models used in physics \cite{Kimura01,Schreivogl13,Barrett15,Connes97}, and we prove their convergences to metric spectral triples on \emph{quantum tori}, in the sense of the spectral propinquity. Our construction is inspired by discussions in the literature in mathematical physics, cited above, of what a spectral triple on a fuzzy torus should be, in analogy with the construction of certain natural Dirac operators on classical tori.

\bigskip

Fuzzy tori are twisted convolution C*-algebras of finite products of finite cyclic groups. In general, a fuzzy $d$-torus associated to a real $d\times d$ antisymmetric matrix $(\theta_{js})_{1\leq j,s \leq d}$ and natural numbers $k_1,\ldots,k_d$, with the condition that $\gcd(k_j,k_s)\theta_{js} \in \Z$ for all $j,s \in \{1,\ldots,d\}$, is the universal C*-algebra generated by $d$ unitaries $U_1,\ldots,U_d$ such that
\begin{equation}\label{fuzzy-torus-eq}
  \forall j,s \in \{1,\ldots,d\} \quad U_j U_s = \exp(2i\pi\theta_{js}) U_s U_j
\end{equation}
and
\begin{equation*}
  \forall j \in \{1,\ldots,d\} \quad U_j^{k_j} = 1 \text.
\end{equation*}
Fuzzy tori, or, at least, families of unitaries in finite dimension with the above commutation relation, can for instance be found in the discussion of quantum mechanics in finite dimension in \cite[pp. 272--280]{Weyl}.

An example of a fuzzy torus is given by the C*-algebra generated by the so-called clock ($C_n$) and shift ($S_n$) matrices, defined by:
\begin{equation}\label{Clock-Shift-eq}
  S_n =
  \begin{pNiceMatrix}
    0      & 1 & 0      & \Cdots &  0 \\
    \Vdots & \Ddots & \Ddots & \Ddots  & \Vdots \\
           &        &        &         & 0\\
    0      & \Cdots &        &         0 & 1 \\
    1      & 0      & \Cdots &           & 0
  \end{pNiceMatrix} \text{ and }C_n =
  \begin{pNiceMatrix}
    1 & & & \\ 
    & \exp\left(\frac{2i\pi}{n}\right) & & \\
    & & \Ddots & \\
    & & & \exp\left(\frac{2i(n-1)\pi}{n}\right)
  \end{pNiceMatrix} \text{,}
\end{equation}
for any $n \in \N\setminus\{0\}$. These matrices appear in many physically-motivated work (see, e.g., \cite{Santhanam78},\cite{Schreivogl13},\cite{Kimura01},\cite{Vourdas04}), as well as in such work as t'Hooft work \cite{tHooft02} on using an underlying dynamics to provide a model for quantum physics, where the clock and shift matrices are of course associated with the C*-crossed-product of the rotation on the cyclic group of $n$-roots of unity by itself via translation (again, see \cite{Weyl}).

\bigskip

Fuzzy tori are often seen as finite dimensional approximations of quantum tori --- including approximations of the classical tori. Quantum tori are the twisted convolution C*-algebras for the groups $\Z^d$. Equivalently, the quantum $d$-torus associated with some real, $d\times d$, antisymmetric matrix $\theta=(\theta_{js})_{1\leq j,s \leq d}$, is the universal C*-algebra generated by $d$ unitaries $U_1$,\ldots,$U_d$ such that
\begin{equation}\label{quantum-torus-eq}
  \forall j,s \in \{1,\ldots,d\} \quad U_j U_s = \exp(2i\pi \theta_{js}) U_s U_j \text.
\end{equation}
Classical tori are particular cases of quantum tori, when $\theta$ is the zero matrix. The analogy between Equation (\ref{fuzzy-torus-eq}) and Equation (\ref{quantum-torus-eq}) is obvious. However, fuzzy tori are finite dimensional (their unitary generators have finite orders) while quantum tori are always infinite dimensional (and are not approximately finite dimensional C*-algebras either \cite{Glimm60},\cite{Bratteli72}). As a finite dimensional C*-algebra is seen as a sort of quantum analogue of a finite set, the idea of approximating quantum tori with fuzzy tori has been a common heuristics. The question for us is: how do we make these heuristics formal?

\bigskip

Fuzzy tori, as finite dimensional C*-algebras, are finite products of full matrix algebras. In particular, $C^\ast(S_n,C_n)$ is simply the C*-algebra of $n\times n$ matrices. Thus, one may ask what makes the algebra of $n\times n$ matrices a finite dimensional quantum analogue of a torus rather than, say, the finite dimensional quantum analogue of a sphere \cite{Rieffel01},\cite{Rieffel09},\cite{Rieffel10b},\cite{Rieffel15}. The answer is, informally, given by introducing some noncommutative, or quantum, geometry on these algebras. Quantum tori, in particular, have a long history as prototypes for noncommutative Riemannian manifolds, starting with Connes' first proposal for an operator-algebra based form of noncommutative geometry in \cite{Connes80}. Since fuzzy tori should be geometric approximations of quantum tori, we thus have some guidance on what the geometry of a fuzzy torus should be. A core proposal of noncommutative geometry is that the analogue of a Riemannian geometry of a quantum space is encoded in a structure called a \emph{spectral triple} \cite{Connes80},\cite{Connes}, introduced by Connes as early as 1985 in his inauguration lecture at the Coll{\`e}ge de France.

A \emph{spectral triple} is a far--reaching generalization of the Dirac operator acting on the spinor bundle of a compact connected spin manifold \cite{Connes89}, given by the following data:
\begin{definition}
  A \emph{spectral triple} $(\A,\Hilbert,\Dirac)$ over a unital C*-algebra $\A$ is given by a Hilbert space $\Hilbert$ and, on a dense subspace $\dom{\Dirac}$ of $\Hilbert$, a self--adjoint operator $\Dirac$ with compact resolvent, such that:
  \begin{enumerate}
  \item there exists a *-representation $\pi$ of $\A$ on $\Hilbert$,
  \item there exists a dense *-subalgebra $A \subseteq \A$ such that
    \begin{equation*}
      \forall a \in A \quad \pi(a) \dom{\Dirac} \subseteq \dom{\Dirac}
    \end{equation*}
    and
    \begin{equation*}
      \forall a \in A \quad [\Dirac,\pi(a)] \text{ is a bounded operator on $\dom{\Dirac}$ \text.}
    \end{equation*}
  \end{enumerate}
  The operator $\Dirac$ is referred to as the \emph{Dirac} operator of the spectral triple $(\A,\Hilbert,\Dirac)$.
\end{definition}

Various spectral triples have been constructed on quantum tori (e.g. \cite{Connes80},\cite{Rieffel98a},\cite{Rieffel01}, \cite{Sitarz13},\cite{Sitarz15},\cite{Gabriel16}). Most of these constructions employ, as a starting point, the dual action of the tori on quantum tori, and the Lie group structure of the tori. Our own presentation will follow this path as well.

The formalism of spectral triples is flexible enough that it is perfectly reasonable to define spectral triples on finite dimensional C*-algebras --- i.e. quantum the analogues of finite sets --- thus providing a unified framework for differential structures and their discrete analogues. In fact, any self-adjoint operator acting on the Hilbert space of some finite dimensional *-representation of a finite dimensional C*-algebra automatically gives us a spectral triple! Various spectral triples have been discussed on fuzzy tori \cite{Kimura01},\cite{Schreivogl13},\cite{Barrett15}, usually constructed in analogy with a classical Dirac operator on a torus. Once more, the heuristic behind these constructions is that the spectral triples on fuzzy tori should converge to some spectral triple on a torus (these papers only consider a commutative limit, though we actually want to allow the limits to be any quantum torus), but no formalism of what such a statement could mean is provided. It is, indeed, not a trivial matter.

\bigskip

The main justifications for our construction of spectral triples on fuzzy tori below are that, at once, they are analogous to natural spectral triples on classical and quantum tori (especially, quite similar to suggestions found in \cite{Schreivogl13},\cite{Barrett15}), and that indeed, our scheme will give convergent sequences of spectral triples. While the complete description of our spectral triples on arbitrary fuzzy and quantum tori requires that we first lay down some notations regarding these C*-algebras, we now give a general idea of the form of our spectral triples on fuzzy and quantum tori (up to a unitary equivalence), to provide an informal introduction to our main result. We begin with an important special case.

The C*-algebra $C^\ast(C_n,S_n)$ is a fuzzy torus, where $S_n C_n = z_n C_n S_n$ and $z_n = \exp\left(\frac{2i\pi}{n}\right)$. Heuristically, and in fact, formally for certain quantum metrics \cite{Latremoliere05},\cite{Latremoliere13c}, $C^\ast(C_n,S_n)$ approximates $C(\T^2)$.

The fuzzy torus carries a natural action of $\left(\bigslant{\Z}{n\Z}\right)^2$, by setting, for all $z \in \bigslant{\Z}{n\Z}$:
\begin{equation*}
  \alpha_n^{(z,1)}(a) = S_n^z a S_n^{-z} \text{ and }\alpha_n^{(1,z)}(a) = C_n^{-z} a C_n^{z}
\end{equation*}
noting that $C_n^z$ is defined as $C_n^{m}$ for any $m$ whose class in $\bigslant{\Z}{n\Z}$ is $z$, since $C_n^n = 1$; moreover we also note that $C_n^{-z} = (C_n^\ast)^z$. The same comment applies to $S_n$. Since $\alpha_{n}^{z,1}$ and $\alpha_{n}^{1,w}$ commute as *-automorphisms of $C^\ast(C_n,S_n)$, for all $z,w \in \bigslant{\Z}{n\Z}$, we define an action of $\left(\bigslant{\Z}{n\Z}\right)^2$ on $C^\ast(C_n,S_n)$, by setting $(z,w)\in\left(\bigslant{\Z}{n\Z}\right)^2\mapsto \alpha_n^{z,w} =\alpha_n^{z,1}\alpha_n^{1,w}$. The action $\alpha_n$ is called the dual action on $C^\ast(C_n,S_n)$. The dual action is what gives $C^\ast(C_n,S_n)$ its quantum geometry.

The dual action of $\left(\bigslant{\Z}{n\Z}\right)^2$ on $C^\ast(C_n,S_n)$ converges, in a way which can be formalized (see \cite{Latremoliere18b},\cite{Latremoliere18c}), to the action by translation of $\T^2$ on the C*-algebra $C(\T^2)$ of $\C$-valued continuous functions over the $2$-torus $\T^2$. This action is again called the dual action of $\T^2$ on $C(\T^2)$. The C*-algebra $C(\T^2)$ is the universal C*-algebra generated by two commuting unitaries $U$ and $V$, and the dual action is uniquely characterized as the action of $\T^2$ on $C(\T^2)$ by *-automorphisms, such that $(z,w) \in \T^2$ is sent to the unique *-automorphism $\alpha_\infty^{(z,w)}$ of $C(U,V)$ such that:
\begin{equation*}
  \alpha_\infty^{(z,w)}(U) = z U \text{ and }\alpha_\infty^{(z,w)}(V) = w V \text.
\end{equation*}
Since $\T^2$ is a Lie group, a general consequence of the existence of this dual action $\alpha_\infty$ implies the existence of a dense *-subalgebra $C^1(\T^2)$ of $C^\ast(U,V)$ which carries an action of the Lie algebra of $\T^2$ in a natural way; for our purpose, we focus on the two canonical derivations on $C^1(\T^2)$, defined, for all $a\in C^1(\T^2)$, by:
\begin{equation*}
  \partial_{\infty,1}(a) = \lim_{t\rightarrow 0} \frac{\alpha_\infty^{(\exp(it),1)}(a) - a}{t} \text{ and }\partial_{\infty,2}(a) = \lim_{t\rightarrow 0} \frac{\alpha_\infty^{(1,\exp(it))}(a) - a}{t} \text.
\end{equation*}
These two derivations are, of course, the usual vector fields on $\T^2$ used as the canonical moving frame for $\T^2$ as a differential manifold.

\medskip

It is thus natural, at first glance, to consider, for $n\in\N$, that the maps
\begin{equation*}
  \partial_{n,1} : a \in C^\ast(C_n,S_n) \mapsto \frac{S_n a S_n^\ast - a}{\frac{2\pi}{n}}\text{ and }\partial_{n,2} : a \in C^\ast(C_n,S_n) \mapsto \frac{C_n^\ast a C_n - a}{\frac{2\pi}{n}}
\end{equation*}
are analogous to the derivations $\partial_{\infty,1}$ and $\partial_{\infty,2}$ of the torus --- the normalization by $\frac{2\pi}{n}$ chosen to give the desired asymptotic behavior, and it could be replaced by any other sequence with the same asymptotic behavior.

Unfortunately, thus defined, $\partial_{n,1}$ and $\partial_{n,2}$ are not derivations of $C^\ast(C_n,S_n)$. In fact, all derivations of $C^\ast(C_n,S_n)$ are given as commutators. Nonetheless, we then observe that
\begin{multline*}
  \frac{n}{2\pi}[S_n,a] = \frac{\alpha_n^{(z,1)}(a) - a}{\frac{2\pi}{n}} S_n = \partial_{n,1}(a) S_n  \\ \text{ and } \frac{n}{2\pi}[C_n^\ast,a] = \frac{\alpha_n^{(1,z)}(a) - a}{\frac{2\pi}{n}} C_n^\ast = \partial_{n,2}(a) C_n^\ast \text,
\end{multline*}
thus connecting our initial guess and a more formally appropriate approach to a discrete quantized calculus for $C^\ast(C_n,S_n)$. Similar computations hold for $[C_n,\cdot]$ and $[S_n^\ast,\cdot]$.

The C*-algebra of $n\times n$ matrices $C^\ast(C_n,S_n)$ is naturally a Hilbert space for the inner product $a,b \in C^\ast(C_n,S_n) \mapsto \mathrm{tr}(a^\ast b)$, with $\mathrm{tr}$ the normalized trace on $n\times n$ matrices. Let $\Hilbert_n$ denote this Hilbert space. Thus, $C^\ast(C_n,S_n)$ naturally acts by left and by right multiplication on $\Hilbert_n$. If $a\in C^\ast(C_n,S_n)$ and $\xi \in \Hilbert_n$, then $[a,\xi] = a\xi - \xi\cdot a$.

This leads us to propose the following self-adjoint operator as a Dirac operator for the fuzzy torus $C^\ast(C_n,S_n)$ --- the self-adjointness is why we use the real part and imaginary part of $C_n$ and $S_n$:
\begin{multline*}
  \Dirac_n = \frac{n}{2 i \pi} \bigg( \left[\frac{S_n+S_n^\ast}{2},\cdot\right]\otimes \gamma_1 + \left[\frac{S_n-S_n^\ast}{2i},\cdot\right]\otimes \gamma_2 \\ +\left[\frac{C_n+C_n^\ast}{2},\cdot\right]\otimes \gamma_3 + \left[\frac{C_n^\ast - C_n}{2i},\cdot\right]\otimes \gamma_4 \bigg)\text,
\end{multline*}
where $\gamma_1,\gamma_2,\gamma_3,\gamma_4$ are $4\times 4$ skew-adjoint matrices such that $\gamma_j\gamma_s+\gamma_s\gamma_j = 0$ if $s\neq j$ and $\gamma_j^2=-1$, for all $j,s\in\{1,\ldots,4\}$. Indeed, we easily compute that
\begin{multline*}
  [\Dirac_n,a] = \frac{n}{2 i \pi} \bigg( \left[\frac{S_n+S_n^\ast}{2},a\right]\otimes \gamma_1 + \left[\frac{S_n-S_n^\ast}{2i},a\right]\otimes \gamma_2 \\ +\left[\frac{C_n+C_n^\ast}{2},a\right]\otimes \gamma_3 +\left[\frac{C_n^\ast - C_n}{2i},a\right]\otimes \gamma_4 \bigg)\text.
\end{multline*}

With this in mind, our heuristics suggests that the limit Dirac operator on $C(\T^2)$ should be given by the following operator acting on a dense subspace of $L^2(\T^2)\otimes\C^4$,
\begin{equation*}
  \Dirac_\infty = \frac{V+V^\ast}{2}\partial_U \otimes \gamma_1 + \frac{V-V^\ast}{2i}\partial_U \otimes \gamma_2 + \frac{U+U^\ast}{2}\partial_V \otimes \gamma_3 + \frac{U^\ast-U}{2i}\partial_V \otimes \gamma_4 \text,
\end{equation*}
where $\partial_U$ and $\partial_V$ are the Sobolev derivatives on $L^2(\T^2)$ which are the closures, respectively, of the operators $\partial_{\infty,1}$ and $\partial_{\infty,2}$, seeing $C(\T^2)$ canonically as a dense subspace of $L^2(\T^2)$.

\medskip

The operator $\Dirac_\infty$, as we shall prove, is indeed a densely defined self-adjoint operator with compact resolvent on $L^2(\T^2,\C^4)$, and $(C(\T^2),L^2(\T^2,\C^4),\Dirac_\infty)$ is indeed a spectral triple on $C(\T^2)$. \emph{In this paper, we will prove, in particular, that $(C(C_n,S_n),\C^n\otimes\C^4,\Dirac_n)$ converges, in the sense of the spectral propinquity, to $(C(\T^2),L^2(\T^2,\C^4),\Dirac_\infty)$.} This Dirac operator is natural, though it involves a sort of rotating frame of spinors, which is the cost of using commutators in defining our spectral triples on fuzzy tori.

\bigskip

Our work goes well beyond the special case discussed above. We allow for any reasonable approximation of any quantum torus by fuzzy torus (and even by mixtures of quantum and fuzzy tori) in our work. To this end, the scheme explained above needs some adjustment. For instance, $C(\T^2)$ can also be seen as the limit of $C(\U^2_n)$, with $\U_n = \left\{ z \in \C : z^n = 1 \right\}$, yet the only derivation of $C(\U^2_n)$ is the zero function. On the other extreme, if $\A$ is a simple quantum torus, which is basically the generic case, then we can not find any non-central unitaries, and thus, our approximation scheme would again not be possible as is. Both these situations illustrate that, in general, we can not find, as in the case of $C^\ast(C_n,S_n)$, all the basic ingredients to carry our scheme within the algebras of interest. But this can be remedied.

The idea is that  we can always embed a fuzzy torus or a quantum torus in a C*-crossed product which contains the unitaries needed to perform our task above. For instance, if we embed $C(\U^2_n)$ in $C(\U^2_n)\rtimes_{\alpha_n} \U_n^2$, where $\alpha_n$ is the action by translation (again, the canonical dual action), then there are by constructions two unitaries $U_{n,3}$ and $U_{n,4}$ in $C(\U^2_n)\rtimes_{\alpha_n} \U_n^2$ such that $U_{n,3} \, a U_{n,3}^\ast = \alpha_n^{(z,1)}(a)$ and $U_{n,4} \, a U_{n,4}^\ast = \alpha_n^{(1,z)}(a)$ for all $a\in C(\U_n^2)$. We are then able to follow a path analogous to our work on $C^\ast(C_n,S_n)$, constructing a spectral triple on $C(\U_n^2)$ which converges, for the propinquity, to a spectral triple on $C(\T^2)$ --- constructed using some elements from $C(\T^2)\rtimes\Z^2$ for the trivial action.

\bigskip

Thus, in this paper, we will start with a sequence $(\A_n)_{n\in\N}$ of fuzzy or quantum tori, and a quantum torus $\A_\infty$, subject to the following natural condition. For all $n\in \N\cup\{\infty\}$, let $\theta_n = (\theta_n^{js})_{1\leq j,s\leq d}$ be some $d\times d$ antisymmetric, real, matrix chosen so that the canonical, generating, unitaries $U_{n,1}$,\ldots,$U_{n,d}$ of $\A_n$ satisfy $U_{n,j} U_{n,s} = \exp(2i\pi\theta_n^{js}) U_{n,s} U_{n,j}$ for all $j,s \in \{1,\ldots,d\}$. We assume that $(\theta_n)_{n\in\N}$ converges to $\theta_\infty$.

For each $n\in\N\cup\{\infty\}$, we will embed $\A_n$ in a fuzzy or quantum $d'$-torus $\B_n$ with $d'\geq d$, in order to carry out the scheme described above for the clock and shift matrices fuzzy tori --- we will explain this embedding later in the paper, and we note that in general, several choices are possible. We denote by $\Hilbert_n$ the Hilbert space of the Gelfand-Naimark-Segal representation of $\B_n$ for the canonical trace.

We fix some gamma matrices $\gamma_1$,\ldots,$\gamma_{d+d'}$ acting on some finite dimensional spaces ---i.e. we fix some *-representation of a Clifford algebra on a finite dimensional space $\mathscr{C}$.  We will prove that natural spectral triples on fuzzy tori, whose Dirac operators are the form
\begin{equation*}
  \Dirac_n = \sum_{j=1}^d [F_{n,j},\cdot]\otimes \gamma_j +  \sum_{j=1}^d [R_{n,j},\cdot]\otimes \gamma_{d+j} + \mathrm{term} \text,
\end{equation*}
acting on some dense subspace of a Hilbert space $\Hilbert_n\otimes \mathscr{C}$, converge to spectral triples on tori of the form
\begin{equation*}
  \Dirac_\infty = \sum_{j=1}^d R_{\infty,j} \partial_j \otimes\gamma_j +  \sum_{j=1}^d F_{\infty,j}\partial_j \otimes\gamma_{d+j} + \mathrm{term} \text,
\end{equation*}
acting on a dense subspace of $\Hilbert_\infty\otimes\mathscr{C}$, where, for all $n\in\N\cup\{\infty\}$, and for all $j\in\{1,\ldots,d\}$, the operators $R_{n,j}$ and $F_{n,j}$ are appropriately scaled version of, respectively, the real part and the imaginary part of some of the generating unitaries in $\B_n$, with the additional condition, when $n=\infty$, that $R_{\infty,j}$ and $F_{\infty,j}$ commute with $\A_\infty$ (in fact, $R_{\infty,j}$ and $F_{\infty,j}$ are the real parts and the imaginary parts of some central canonical unitary generators of $\B_\infty$, up to a $\pm 1$ factor). The additional ``term'', which we leave for the main formal description, is needed in order to ensure that $\Dirac_\infty$ has a compact resolvent --- indeed, the operator $\Dirac_\infty$ is defined on $\Hilbert_\infty\otimes \mathscr{C}$ where $\Hilbert_\infty$ carries a *-representation of $\B_\infty$, not just $\A_\infty$. In turn, this term in $\Dirac_\infty$ must have a discrete form in $\Dirac_n$ ($n\in\N$), so that the desired convergence occurs. The formal description of these spectral triples will be given once we have the needed notation about quantum tori.

\medskip

Our spectral triples on $C^\ast(C_n,S_n)$ can also be found in \cite{Schreivogl13}, so our construction extends the construction in \cite{Schreivogl13} in various, far reaching ways (as they can be defined for any fuzzy or quantum torus). A related construction of a different, yet similar, spectral triple over $C^\ast(C_n,S_n)$ is found in \cite{Barrett15}, though, again, our techniques in this paper are applicable to a much wider family of examples. Using the commutators with the clock and shift matrices as discrete versions of the canonical moving frame on the torus is also discussed, for instance, in \cite{Kimura01}. Our spectral triples also share some commonalities with the spectral triples in \cite{Sitarz13},\cite{Sitarz15}, since they involve modifying the usual, ``flat'' spectral triple on quantum tori by elements which commute with the quantum torus. However, our approach requires, in general, an additional term to compensate for the introduction of these elements, which in turn, is caused by the form of the finite dimensional spectral triples on fuzzy tori --- a matter foreign to the discussion in \cite{Sitarz13},\cite{Sitarz15}.

While our spectral triples have nice properties --- the symbol of their square give the familiar symbol of the Laplacian, and the dual actions act by Lipschitz functions, and in fact, isometries in many cases (we can always choose, with our scheme, that the dual action should act by quantum isometries) --- their primary value is in the fact that they form convergent sequence for the spectral propinquity.

\subsection{The Spectral Propinquity}
    
We now must of course explain what the convergence of metric spectral triples actually means, since it is the very reason for the present work. Such a notion is certainly involved: for instance, convergence of the spectra of the Dirac operators of spectral triples is vastly insufficient, since the spectrum of the Dirac operator of a spectral triple can not distinguish even between non-homeomorphic spaces.  Our notion of convergence will now take us to the realm of noncommutative metric geometry.

\bigskip

Our idea for convergence for spectral triples begins, as a first step, by exploiting the crucial observation by Connes that a spectral triple induces an extended metric on the state space of its underlying C*-algebra. Indeed, it may help to motivate our idea with a simple observation. Informally, it is natural to look upon the sets $\U_n = \{ z \in \C : z^n = 1\}$ as approximations of the circle $\T = \{z\in\C : |z|=1\}$. However, $\U_n$ is, topologically, just a set of $n$ elements, and thus it is homeomorphic to $\mathds{S}_n = \{\frac{1}{n},\ldots,\frac{n}{n} =1\}$. Yet the latter set is, informally, an approximation of $[0,1]$. Underlying our intuition here is not topology, but metric geometry. Indeed, as $n$ tends to $\infty$, the sequence $(\U_n)_{n\in\N}$, where $\U_n$ is endowed with its usual metric as a subspace of $\C$, does converges to $\T$ for the \emph{Hausdorff distance} \cite{Hausdorff} induced by the Euclidean metric on $\C$, whereas $(\mathds{S}_n)_{n\in\N}$ converges, for the same metric, to $[0,1]$.

We seek a similar framework to formalize that some sequences of ``fuzzy tori'' converge to quantum tori: at a minimum, we want to endow these spaces with a quantum analogue of a metric, and prove the convergence of the resulting structure for an analogue of the Hausdorff distance, adapted to our noncommutative geometric setting. Since many reasonable spectral triples do give rise to metrics on state spaces \cite{Connes89}, we have the starting point for our definition of a convergence of spectral triples: metrics they induce should converge in a generalized Hausdorff distance. This certainly seems a reasonably physical concept. The generalized Hausdorff distance we shall use is the Gromov-Hausdorff propinquity on the class of {\qcms s}, which we introduced in \cite{Latremoliere13},\cite{Latremoliere13b},\cite{Latremoliere14}, and we now describe. As we shall see shortly, convergence of spectral triples involve more than the convergence of the underlying metric, but this is the first step.

\bigskip

The idea of a quantum metric begins with the  well-known and profound (contravariant) equivalence of category between the category of compact Hausdorff spaces and the category of unital \emph{Abelian} C*-algebras, established by Gelfand and Naimark. In general, a category of quantum spaces consists of algebras which generalizes (and include) some algebras of functions over certain types of spaces, with morphisms the reversed arrows from the natural morphisms between these algebras. For instance, a quantum compact Hausdorff space is described by a unital C*-algebra, which is no longer assumed to be commutative, and seen as an object in the dual category of C*-algebras.

Quantum compact metric spaces, in turn, are described by generalizations of the algebras of \emph{Lipschitz functions} over a compact metric space. If $(X,d)$ is a compact metric space, then the \emph{Lipschitz constant} $\Lip_d(f)$ of an $\R$-valued function $f:X\rightarrow \R$ is given by
  \begin{equation}\label{Lipschitz-eq}
    \Lip_d(f) = \sup\left\{ \frac{|f(x)-f(y)|}{d(x,y)} : x,y \in X, x\neq y \right\} \text,
  \end{equation} allowing $\Lip_d(f) = \infty$.
  
  A function from $X$ to $\C$ with a finite Lipschitz constant is called a Lipschitz function over $(X,d)$ --- and it is always an element of $C(X)$. This definition gives a seminorm $\Lip_d$ defined on a dense subalgebra of $C(X)$.

  \begin{convention}
    If $L$ is a seminorm defined on a dense subspace $\dom{L}$ of a normed vector space $E$, then we set $L(x) = \infty$ whenever $x\notin \dom{L}$, so that $\dom{L} = \{x \in E : L(x) < \infty\}$. Within this convention, we will use $0\infty = 0$, $t\leq\infty$, and $\infty + t = t + \infty = \infty$ for all $t\in\R$.
  \end{convention}

  Kantorovich, motivated by his work on Monge's transportation problem, defined, in \cite{Kantorovich40},\cite{Kantorovich58}, a distance $\Kantorovich{\Lip_d}$ on the state space $\StateSpace(C(X))$ of $C(X)$, i.e. the set of integrals with respect to Radon probability measures over $X$, by setting, for all $\mu,\nu \in \StateSpace(C(X))$:
  \begin{equation*}
    \Kantorovich{\Lip_d} (\mu,\nu) = \sup\left\{ \left| \int_X f\, d\mu - \int_X f\, d\nu\right| : f \in C(X), f=\overline{f}, \Lip_d(f) \leq 1\right\}\text.
  \end{equation*}
  This metric, known as the {\MongeKant}, induces the weak* topology on $\StateSpace(C(X))$, making it a fundamental tool not only in optimal transport theory, but also probability, statistics, and even fractal theory. It has been often renamed (for instance, Dobrushin \cite{Dobrushin70} called this metric the Wasserstein metric, following works in probability by Wasserstein \cite{Wasserstein69}). Moreover, the map which sends points of $X$ to their corresponding Dirac point measures over $X$ becomes an isometry from $(X,d)$ to $(\StateSpace(C(X)),\Kantorovich{\Lip_d})$; therefore, $\Lip_d$ allows us to recover the metric over $X$.

  Rieffel \cite{Rieffel98a},\cite{Rieffel99}, motivated by Connes' original proposal about quantum metric spaces \cite{Connes89}, proposed that quantum metrics could be metrics on state spaces induced by duality from seminorms on (no longer necessarily commutative) unital C*-algebras which share basic properties with the Lipschitz seminorms. The exact list of which properties of the Lipschitz seminorms should be kept has evolved in time. We settle on the definition which we used in \cite{Latremoliere13},\cite{Latremoliere15}.

  \begin{notation}
    If $\A$ is a normed vector space, the norm of $\A$ is denoted by $\norm{\cdot}{\A}$, unless otherwise specified.
  \end{notation}
  
\begin{definition}\label{qcms-def}  
  A \emph{\qcms} $(\A,\Lip)$ is an ordered pair of a unital C*-algebra and a seminorm $\Lip$, defined on a dense subspace $\dom{\Lip}\subseteq \sa{\A}$ of the space $\sa{\A} = \{ a \in \A : a^\ast = a \}$ of self-adjoint elements in $\A$, such that:
  \begin{enumerate}
  \item $\forall a \in \dom{\Lip} \quad \Lip(a) = 0 \iff a \in \R \unit_\A$, where $\unit_\A$ is the unit of $\A$,
  \item the \emph{\MongeKant} $\Kantorovich{\Lip}$, defined between any two states $\varphi$, $\psi$ of $\A$ by
    \begin{equation*}
      \Kantorovich{\Lip}(\varphi,\psi) = \sup\left\{ |\varphi(a) - \psi(a)| : a\in\dom{\Lip}\text{ and }\Lip(a) \leq 1 \right\}\text,
    \end{equation*}
    is a metric on the state space $\StateSpace(\A)$ of $\A$ which induces the weak* topology on $\StateSpace(\A)$,
  \item for all $a,b \in \dom{\Lip}$, the Jordan product $\frac{ab + ba}{2}$ of $a$ with $b$, and the Lie product $\frac{a b - b a}{2i}$ of $a$ with $b$, are both elements of $\dom{\Lip}$, and moreover:
    \begin{equation*}
      \max\left\{ \Lip\left(\frac{ab+ba}{2}\right),\Lip\left(\frac{ab-ba}{2i}\right)\right\} \leq \norm{a}{\A} \Lip(b) + \Lip(a)\norm{b}{\A} \text,
    \end{equation*}
  \item $\{a \in \dom{\Lip} : \Lip(a) \leq 1 \}$ is closed in $\A$.
  \end{enumerate}
\end{definition}

It can indeed be checked that if $(X,d)$ is a compact metric space, then $(C(X),\Lip_d)$ is a {\qcms}. Noncommutative examples include C*-algebras endowed with a strongly continuous ergodic action of a compact group equipped with a continuous length function \cite{Rieffel98a} --- which includes quantum tori, fuzzy tori \cite{Rieffel98a}, and noncommutative solenoids \cite{Latremoliere16}; C*-algebras of word hyperbolic group \cite{Ozawa05}, C*-algebras of nilpotent discrete groups \cite{Rieffel12}, Podles spheres \cite{Kaad18}, certain C*-crossed-products \cite{Hawkins13}, certain quantum groups \cite{Voigt14}, and more (e.g., \cite{Rieffel02},\cite{Li09},\cite{Li04},\cite{Lapidus08}). It is possible to relax the Leibniz relation (i.e. Condition (3) in Definition (\ref{qcms-def})) to obtain even more examples, such as AF algebras \cite{Latremoliere15d}. A generalization of the theory of {\qcms s} to locally compact quantum metric spaces can be found in \cite{Latremoliere05b},\cite{Latremoliere12b}.

\medskip

The relationship between spectral triples and {\qcms s} is captured in the following definition, which owes to the original observation by Connes \cite{Connes89} that a spectral triple induces an extended metric on the state space of its underlying C*-algebras.

\begin{notation}
If $T : E\rightarrow F$ is a linear map between two normed spaces, we write its norm as $\opnorm{T}{E}{F}$ (allowing for $\infty$). When $E=F$, we simply write $\opnorm{T}{}{F}$.
\end{notation}

\begin{definition}
  A spectral triple $(\A,\Hilbert,\Dirac)$, with the *-representation of $\A$ on $\Hilbert$ denoted by $\pi$, is a \emph{metric spectral triple} when, if we set, for all $a\in\sa{\A}$ such that $\pi(a) \dom{\Dirac}\subseteq \dom{\Dirac}$,
  \begin{equation*}
    \Lip_\Dirac(a) = \opnorm{[\Dirac,\pi(a)]}{}{\Hilbert}\text,
  \end{equation*}
  then $(\A,\Lip_\Dirac)$ is a {\qcms}.
\end{definition}  

\begin{remark}
  An immediate consequence of the definition of a metric spectral triple $(\A,\Hilbert,\pi)$ is that the *-representation of $\A$ on $\Hilbert$ used to define such a spectral triple must be faithful; we thus can identify $\A$ with its image by this *-representation, and will do so whenever it introduces no risk of confusion.
\end{remark}

An example of metric spectral triple is given on quantum tori, and on certain quantum homogeneous spaces, by Rieffel in \cite{Rieffel98a}. Another family of examples is given by the curved quantum tori \cite{Sitarz13},\cite{Sitarz15}, as proven in \cite{Latremoliere15c}. Other examples include spectral triples from length functions on certain groups \cite{Ozawa05},\cite{Rieffel02},\cite{Rieffel15b}, spectral triples over Podles spheres \cite{Kaad18}, or spectral triples over certain fractals \cite{Lapidus08},\cite{Latremoliere20a}, for instance.

In this paper, we will prove that our proposed spectral triples on fuzzy and quantum tori are all metric spectral triples. Therefore, it becomes possible to apply to them our theory of {\qcms s} and their convergence, which we now turn to.

\medskip

The class of {\qcms s} form a category for an appropriate notion of Lipschitz morphisms \cite{Latremoliere16b}. We focus here on the subcategory whose objects are {\qcms s} and morphisms are quantum isometries, in the following sense (note that strictly speaking, quantum isometries should be given by the opposite arrows, but it will be clearer to keep our arrows in the same direction as the actual morphisms). The idea behind the following notion of quantum isometries is due to Rieffel, and is motivated by McShane's theorem \cite{McShane34} for the extension of \emph{real-valued} Lipschitz functions.

\begin{definition}[{\cite{Rieffel99},\cite{Rieffel00},\cite{Latremoliere13}}]
  A \emph{quantum isometry} $\pi : (\A,\Lip_\A) \rightarrow (\B,\Lip_\B)$ is a surjective *-morphism $\pi : \A \twoheadrightarrow \B$ such that
  \begin{equation*}
    \forall b \in \sa{\B} \quad \Lip_\B(b) = \inf\left\{ \Lip_\A(a) : \pi(a) = b \right\} \text.
  \end{equation*}
\end{definition}

In particular, we will consider two {\qcms s} to be isomorphic when they are fully quantum isometric, in the following sense.
\begin{definition}[{\cite{Latremoliere13}}]
  A \emph{full quantum isometry} $\pi :(\A,\Lip_\A) \rightarrow (\B,\Lip_\B)$ is a *-isomorphism $\pi : \A \rightarrow \B$ such that $\Lip_\B\circ\pi = \Lip_\A$.
\end{definition}

\medskip

Our principal interest is the construction, on the class of {\qcms s}, of a distance function, called the \emph{Gromov-Hausdorff propinquity}, which generalizes the Gromov-Hausdorff distance \cite{Gromov},\cite{Gromov81},\cite{Edwards75}, itself a generalization of the Hausdorff distance \cite{Hausdorff}.

\begin{notation}
  The \emph{Hausdorff distance} \cite{Hausdorff} induced on the class of closed subsets of a compact metric space $(X,d)$ is denoted by $\Haus{d}$. We recall that if $A_1,A_2\subseteq X$ are two closed subsets of $(X,d)$, then
  \begin{equation*}
    \Haus{d}(A_1,A_2) = \max_{\{j,s\}=\{1,2\}} \sup_{x\in A_j} \inf_{y \in A_s} d(x,y) \text.
  \end{equation*}
\end{notation}

The construction of the propinquity, upon which all of our work is built, goes as follows. Motivated by Edwards \cite{Edwards75} and Gromov \cite{Gromov81} constructions, we define the propinquity by first considering any quantum analogues of an isometric embedding of two {\qcms s} into a third one.

\begin{figure}[t]
  \begin{equation*}\xymatrixcolsep{0pc}
   \xymatrix{
      & (\D, \TLip)  \ar_{\pi_\A}[dl] \ar^{\pi_\B}[dr]& \\
      (\A,\Lip_\A) &   & (\B,\Lip_\B) 
      }
    \end{equation*}
    \caption{A tunnel}
\end{figure}
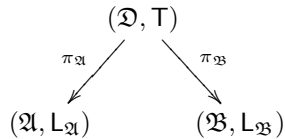

\begin{definition}\label{tunnel-def}
  A \emph{tunnel} $\tau = (\D,\Lip_\D,\pi_1,\pi_2)$ between two {\qcms s} $(\A_1,\Lip_1)$ and $(\A_2,\Lip_2)$ is given by a {\qcms} $(\D,\Lip_\D)$ and, for each $j\in\{1,2\}$, a quantum isometry $\pi_j$ from $(\D,\Lip_\D)$ onto $(\A_j,\Lip_j)$.
\end{definition}

\begin{notation}
  If $\pi : \D \rightarrow \A$ is a *-epimorphism between two unital C*-algebras $\D$ and $\A$, then we denote the transpose linear map $\varphi \in \A^\ast \mapsto \varphi\circ\pi \in \D^\ast$ by $\pi^\ast$. In particular, $\pi^\ast$ is an injective map from $\StateSpace(\A)$ into $\StateSpace(\D)$. When $\pi$ is actually a quantum isometry from $(\D,\Lip_\D)$ onto $(\A,\Lip_\A)$, for some L-seminorms $\Lip_\D$ and $\Lip_\A$, then $\pi^\ast$ is an isometry from $(\StateSpace(\A),\Kantorovich{\Lip_\A})$ to $(\StateSpace(\D),\Kantorovich{\Lip_\D})$ \cite{Rieffel99},\cite{Rieffel00}.
\end{notation}

Our metric is then defined as follows.
\begin{definition}
  Let $(\A,\Lip_\A)$ and $(\B,\Lip_\B)$ be two {\qcms s}. If $\tau = (\D,\Lip_\D,\pi_1,\pi_2)$ is a tunnel from $(\A,\Lip_\A)$ to $(\B,\Lip_\B)$, then the \emph{extent} $\tunnelextent{\tau}$ of $\tau$ is the non-negative real number
  \begin{equation*}
    \tunnelextent{\tau} = \max\left\{ \Haus{\Kantorovich{\Lip_\D}}(\StateSpace(\D),\pi_1^\ast\StateSpace(\A)), \Haus{\Kantorovich{\Lip_\D}}(\StateSpace(\D),\pi_2^\ast\StateSpace(\B)) \right\} \text.
  \end{equation*}
  
  The \emph{propinquity} $\dpropinquity{}((\A,\Lip_\A),(\B,\Lip_\B))$ is the non-negative number
  \begin{equation*}
    \dpropinquity{}((\A,\Lip_\A),(\B,\Lip_\B)) = \inf\left\{ \tunnelextent{\tau} : \text{ $\tau$ is a tunnel from $(\A,\Lip_\A)$ to $(\B,\Lip_\B)$} \right\}\text.
  \end{equation*}
\end{definition}

We were able to prove that the dual propinquity satisfies the following set of desirable properties, which, in particular, extend the properties of the Gromov-Hausdorff distance from the class of compact metric spaces to the class of {\qcms s}:
\begin{theorem}[{\cite{Latremoliere13},\cite{Latremoliere13b},\cite{Latremoliere14}}]
  The propinquity $\dpropinquity{}$ is a complete distance (class) function over the class of {\qcms s}, up to full quantum isometry. Moreover, the (class) function which sends any compact metric space $(X,d)$ to the {\qcms} $(C(X),\Lip_d)$, with $\Lip_d$ defined by Expression (\ref{Lipschitz-eq}), is a homeomorphism from the class of compact metric spaces, endowed with the topology of the Gromov-Hausdorff distance, to its range, with the topology induced by the Gromov-Hausdorff propinquity.
\end{theorem}
Since, in particular, a sequence $(X_n,d_n)_{n\in\N}$ of compact metric spaces converges to a compact metric space $(X,d)$ for the Gromov-Hausdorff distance if, and only if, the sequence $(C(X_n),\Lip_{d_n})_{n\in\N}$ converges to $(C(X),\Lip_d)$ for the propinquity $\dpropinquity{}$, the propinquity allows us to extend the topology of the Gromov-Hausdorff distance to noncommutative geometry. The propinquity allows us, for instance, to prove an analogue of Gromov's compactness theorem in noncommutative geometry \cite{Latremoliere15}.

Our main application for the propinquity, so far, is to obtain continuous families of {\qcms s}, and in particular, to construct rigorous finite dimensional approximations of {\qcms s}, often motivated by mathematical physics. The physical relevance of our metric is that it indeed implies the convergence of the state space of a C*-algebra. Examples of convergence of {\qcms s} for the propinquity include any convergence in the sense of the Gromov-Hausdorff distance for classical compact metric spaces (such as convergence of some fractals for their geodesic distances \cite{Latremoliere20a}), finite dimensional approximation of AF algebras and continuous families of AF algebras (such as UHF algebras and Effros-Shen algebras) \cite{Latremoliere15d}, convergence of matrix algebras to the sphere \cite{Rieffel15}, finite dimensional approximation of {\qcms s} built over nuclear quasi-diagonal C*-algebras \cite{Latremoliere15}, and, notably for our purpose, the convergence of sequences of fuzzy tori to quantum tori \cite{Latremoliere05},\cite{Latremoliere13c}, for the quantum metrics introduced by Rieffel in \cite{Rieffel98a}.

\medskip

In this work, we will prove, as a first step toward the convergence of spectral triples, that \emph{the quantum metrics induced by our spectral triples on fuzzy tori converge to quantum metrics on quantum tori under natural conditions}. The quantum metrics induced by the spectral triples in this paper are not the same as the quantum metrics used in \cite{Latremoliere13c} (which were not constructed using a spectral triple), so this result is new to the present work; however, our strategy consists in establishing properties for the quantum metrics of our spectral triples which then enables us to apply the techniques of \cite{Latremoliere13c}, which we will not repeat here. We will, however, record the main ingredients of the proof of \cite[Theorem 5.2.5]{Latremoliere13c} which we need to further study the convergence of spectral triples, beyond the metric aspects.

\bigskip

Now, a spectral triple contains more information than the metric it induces. A proper notion of convergence for spectral triples should also account for this extra information. Our idea, which we developed in \cite{Latremoliere18g}, begins with the observation that if $(\A,\Hilbert,\Dirac)$ is a spectral triple, then $t \in \R \mapsto S^t = \exp(i t \Dirac)$ is a continuous group action by unitaries of $\R$ on $\Hilbert$, i.e. it is a \emph{quantum dynamics} (since $\Dirac$ is self-adjoint). This quantum dynamics is, of course, tightly related to $\Dirac$, which is its infinitesimal generator. Moreover, for any operator $a$ in a dense *-subalgebra of $\A$, we also note that $[\Dirac,a] = \lim_{t\rightarrow 0} \frac{S^t a S^{-t} - a}{i t}$. Thus, we propose to define the convergence of a spectral triple by adding, to the convergence of the underlying quantum metrics, the convergence of the associated quantum dynamics, seen as the natural physical object associated with a spectral triple. To this end, we proceed in two steps. First, since the quantum dynamics for two different spectral triples typically act on different Hilbert spaces (in the present work, spectral triples for fuzzy tori act on finite dimensional Hilbert spaces, while spectral triples on quantum tori act on infinite dimensional Hilbert spaces), and since spectral triples involve *-representations as well, we need a mean to compare these Hilbert spaces, as modules over {\qcms s}. Second, we will indeed use a covariant version of the propinquity generalized to modules over {\qcms s}.

\medskip

The construction of the spectral propinquity, which we now describe, is thus based on various new distances, based upon the propinquity, but defined for various classes of quantum objects: we shall use the ideas of \cite{Latremoliere16c},\cite{Latremoliere17a},\cite{Latremoliere18a},\cite{Latremoliere18d}, where we define the modular propinquity --- a metric between appropriately defined modules over {\qcms s}, as defined below. We also use the ideas of \cite{Latremoliere17c},\cite{Latremoliere18b},\cite{Latremoliere18c},\cite{Latremoliere18g}, where we define a covariant version of the propinquity and the modular propinquity, which enables us to discuss convergences of group actions. We present a synthesis of the tools from these works which are needed to define the spectral propinquity below.

\medskip

Let $(\A,\Hilbert,\Dirac)$ be a metric spectral triple. By \cite{Latremoliere18g}, if we set:
\begin{equation*}
  \forall \xi \in \dom{\Dirac} \quad \CDN(\xi) = \norm{\xi}{\Hilbert} + \norm{\Dirac \xi}{\Hilbert}
\end{equation*}
and
\begin{equation*}
  \forall a \in \{ b\in\sa{\A}:b\,\dom{\Dirac}\subseteq \dom{\Dirac} \} \quad \Lip_{\Dirac}(a) = \opnorm{[\Dirac,a]}{}{\Hilbert}
\end{equation*}
allowing for the value $\infty$, the tuple
\begin{equation}\label{mcc-eq}
  \mcc{\A}{\Hilbert}{\Dirac} = (\Hilbert,\CDN,\C,0,\A,\Lip_\Dirac)
\end{equation}
is an example of a \emph{metrical $C^\ast$-correspondence} in the following sense:

\begin{definition}[{\cite{Latremoliere16c,Latremoliere18d}}]\label{mcc-def}
  A \emph{$K$-metrical $C^\ast$-correspondence}
  \begin{equation*}
    (\module{M},\CDN,\B,\Lip_\B,\A,\Lip_\A)\text,
  \end{equation*}
  for $K\geq 1$, is given by:
  \begin{itemize}
  \item two {\qcms s} $(\A,\Lip_\A)$ and $(\B,\Lip_\B)$,
  \item a right Hilbert C*-module  $\module{M}$ over $\B$, which also carries a left action of $\A$ (by adjoinable operators), i.e. $\module{M}$ is an $\A$-$\B$ C*-correspondence,
  \item a norm $\CDN$ defined on a dense $\A$-submodule $\dom{\CDN}$ of $\module{M}$
  \end{itemize}
  such that, denoting the $\B$-inner product on $\module{M}$ by $\inner{\cdot}{\cdot}{\module{M}}$ and the associated norm by $\norm{\cdot}{\module{M}}$, the following properties hold:
  \begin{enumerate}
  \item $\forall \omega \in \dom{\CDN} \quad \norm{\omega}{\module{M}} \leq \CDN(\omega)$,
  \item $\{ \omega\in\dom{\CDN} : \CDN(\omega) \leq 1 \}$ is compact in $\norm{\cdot}{\module{M}}$,
  \item for all $\omega,\eta\in\module{M}$, denoting $b = \inner{\omega}{\eta}{\module{M}} \in \B$, we have:
    \begin{equation*}
      \max\left\{ \Lip_\B\left(\frac{b+b^\ast}{2}\right), \Lip_\B\left(\frac{b-b^\ast}{2i}\right)\right\} \leq 2 \CDN(\omega)\CDN(\eta)\text{,}
    \end{equation*}
    (we refer to this inequality as the inner Leibniz inequality),
  \item for all $a\in\dom{\Lip_\A}$ and $\omega\in \dom{\CDN}$ we have:
    \begin{equation*}
      \CDN(a\omega) \leq K \left(\norm{a}{\A} + \Lip_\A(a)\right)\CDN(\omega)\text{.}
    \end{equation*}
    (we refer to this inequality as the modular Leibniz inequality).
  \end{enumerate}
  The norm $\CDN$ is called a \emph{D-norm}.
  
  When $(\module{M},\CDN,\B,\Lip_\B,\C,0)$ is a metrical $C^\ast$-correspondence, $(\module{M},\CDN,\B,\Lip_\B)$ is called a \emph{\gQVB}. 
\end{definition}

\begin{remark}
  We called metrical C*-correspondence by the different name metrical quantum vector bundles in \cite{Latremoliere18g}; however we prefer to use a terminology more in line with common usage in C*-algebra theory.
\end{remark}

Examples of {\gQVB s} include Hermitian bundles over Riemannian manifolds, with the D-norm constructed from any metric connection \cite{Latremoliere16c}, and Heisenberg modules over quantum tori \cite{Latremoliere17a}, where the D-norms also arise naturally from a noncommutative analogue of a metric connection. For our purpose, the main example of metrical C*-correspondence is given by metric spectral triples \cite{Latremoliere18g}, using Expression (\ref{mcc-def}). We thus will apply this construction to the spectral triples on fuzzy and quantum tori we introduce in this paper.

\bigskip

Thus, the question is: how do we define the convergence of metrical C*-correspondences? The idea, actually, is rather natural. For each $j \in \{1,2\}$, let:
\begin{equation*}
  \left( \module{M}_j, \CDN_j, \B_j, \Lip_j, \A_j, \Lip'_j \right)
\end{equation*}
be a metrical $C^\ast$-correspondence. Our first task, of course, is to determine how close $(\A_1,\Lip'_1)$ and $(\A_2,\Lip'_2)$ are, and how close $(\B_1,\Lip_1)$ and $(\B_2,\Lip_2)$ are. This means that we will need two tunnels, $\tau_\A$ and $\tau_\B$ (see Definition (\ref{tunnel-def}) above).  To deal with the modules $\module{M}_1$ and $\module{M}_2$, we then mimic our idea from the propinquity. Remarkably, once the proper notion of tunnel between metrical C*-correspondence is introduced, the tools introduced for the propinquity can be applied directly. Formally, it is convenient to proceed in two steps: first, we work with the ``modular quantum vector bundle'' part.

\begin{definition}
  Let $\A$,$\B$ be two unital C*-algebras. A \emph{module morphism} $(\Pi,\pi)$ from a right $\A$-module $\module{M}$ to a right $\B$-module $\module{N}$ is given by the following data:
  \begin{itemize}
  \item a unital *-morphism $\pi : \A\rightarrow\B$,
  \item a linear map $\Pi : \module{M}\rightarrow\module{N}$ such that
    \begin{equation*}
      \forall a \in \A, \quad \forall \omega \in \module{M}, \quad \Pi(\omega a) = \Pi(\omega)\pi(a) \text{,}
    \end{equation*}
  \end{itemize}
  with the analogue definition for a module morphism between left modules.
  
  The module morphism $(\Pi,\pi)$ is said to be \emph{surjective} when both $\Pi$ and $\pi$ are surjective maps, and it is said to be an \emph{isomorphism} when both $\Pi$ and $\pi$ are bijections.
  
  A Hilbert module morphism $(\Pi,\pi)$ from a right $\A$ C*-Hilbert module $\module{M}$ to a right $\B$ C*-Hilbert module $\module{N}$ is a module morphism when
  \begin{equation*}
    \forall \omega,\xi \in \module{M}, \quad \inner{\Pi(\omega)}{\Pi(\xi)}{\module{N}} = \pi\left(\inner{\omega}{\xi}{\module{M}}\right) \text{.}
  \end{equation*}
\end{definition}

We refer to \cite{Latremoliere16c} for examples of such a structure and for its motivations.

\bigskip

We defined the modular propinquity in \cite{Latremoliere16c,Latremoliere18d} by extending the notion of a tunnel between {\qcms s} to the notion of a tunnel between {\gQVB s}.

\begin{definition}[{\cite{Latremoliere18d}}]
  Let $(\module{M}_j,\CDN^j,\A_j,\Lip_j)$ be a {\gQVB}, for $j\in\{1,2\}$. A \emph{modular tunnel} $(\mathds{D},(\Pi_1,\pi_1),(\Pi_2,\pi_2))$ is given by
  \begin{enumerate}
  \item a {\gQVB} $\mathds{D} = (\module{P}, \CDN', \D, \Lip_\D)$,
  \item a tunnel $(\D,\Lip_\D,\pi_1,\pi_2)$ from $(\A_1,\Lip_1)$ to $(\A_2,\Lip_2)$,
  \item for each $j\in\{1,2\}$, $(\Pi_j,\pi_j)$ is a surjective Hilbert module morphism from $\module{P}$ (over $\D$) to $\module{M}_j$ such that
    \begin{equation*}
      \forall \omega \in \module{M}_j, \quad \CDN^j(\omega) = \inf\left\{ \CDN'(\eta) : \Pi_j(\eta) = \omega \right\} \text{.}
    \end{equation*}
  \end{enumerate}
\end{definition}

The extent of a modular tunnel is computed just like the extent of the underlying tunnel between {\qcms s}, as we see in the following definition.

\begin{definition}[{\cite{Latremoliere18d}}]
  The \emph{extent}, $\tunnelextent{\tau}$, of a modular tunnel
  \begin{equation*}
    \tau = (\mathds{D},(\Pi_1,\pi_1),(\Pi_2,\pi_2))\text,
  \end{equation*}
  with $\mathds{D} = (\module{P},\CDN',\D,\Lip_\D)$, is the extent of the tunnel $(\D,\Lip_\D,\pi_1,\pi_2)$.
\end{definition}

The modular propinquity is then defined along the same lines as the propinquity.

\begin{definition}[{\cite{Latremoliere16c,Latremoliere18d}}]
  The \emph{modular propinquity}, $\dmodpropinquity{}(\mathds{A},\mathds{B})$, between two {\gQVB s} $\mathds{A}$ and $\mathds{B}$ is the nonnegative number given by
  \begin{equation*}
    \dmodpropinquity{}(\mathds{A},\mathds{B}) = \inf\left\{\tunnelextent{\tau} : \text{ $\tau$ is a modular tunnel from $\mathds{A}$ to $\mathds{B}$} \right\}\text{.}
  \end{equation*}
\end{definition}

We now record a few fundamental properties of the modular propinquity.

\begin{theorem}[{\cite{Latremoliere16c,Latremoliere18d}}]
  The modular propinquity is a complete metric, up to full isometric isomorphism, where two {\gQVB s} $(\module{M},\CDN,\A,\Lip_\A)$ and $(\module{N},\CDN',\B,\Lip_\B)$ are fully isometrically isomorphic if and only if there exists a Hilbert module isomorphism $(\Pi,\pi)$ from $\module{M}$ to $\module{N}$ such that
  \begin{itemize}
  \item $\Lip_\B\circ\pi = \Lip_\A$,
  \item $\CDN'\circ\Pi = \CDN$,
  \item $\pi$ is a full quantum isometry from $(\A,\Lip_\A)$ to $(\B,\Lip_\B)$.
  \end{itemize}
\end{theorem}

We extend a tunnel between {\gQVB s} to a tunnel between metrical C*-correspondences as follows. The idea behind the following definition is just to bring the ideas described up to now together in a single package. We thus obtain the following construction, and basic properties, for a metrical propinquity.

\begin{definition}[{\cite{Latremoliere18d}}]\label{metrical-tunnel-def}
  Let $\mathds{A}^j = (\module{M}_j,\CDN_{\module{M}_j},\B_j,\Lip_{\B_j},\A_j,\Lip_{\A_j})$ be a metrical C*-correspondence, for each $j\in\{1,2\}$.

  A \emph{$K$-metrical tunnel} $(\tau,\tau')$ from $\mathds{A}^1$ to $\mathds{A}^2$, for some $K\geq 1$, is given by the following data:
   \begin{enumerate}
   \item a modular tunnel $\tau = (\mathds{D},(\theta_1,\Theta_1),(\theta_2,\Theta_2))$ from $(\module{M}_1,\CDN_{\module{M}_1},\B_1,\Lip_{\B_1})$ to $(\module{M}_2,\CDN_{\module{M}_2},\B_2,\Lip_{\B_2})$, where we write $\mathds{D} = (\module{P},\CDN,\D,\Lip_\D)$,
     \item a tunnel $\tau' = (\D',\Lip',\pi^1,\pi^2)$ from $(\A_1,\Lip_{\A_1})$ to $(\A_2,\Lip_{\A_2})$,
     \item there exists a *-morphism from $\D$' to the C*-algebra of adjoinable,
       $\D$-linear operators on $\module{P}$; we will write this *-representation as a left action on $\module{P}$,
     \item $\forall \omega \in \module{P}, \forall d \in \D',\quad \CDN(d\omega)\leq K (\Lip'(d)+\norm{d}{\D'})\CDN(\omega)$,
     \item for all $j\in\{1,2\}$, the pair $(\pi^j, \Theta^j)$ is a left module morphism from the left $\D'$-module $\module{P}$ to the left $\A_j$-module $\module{M}_j$.
   \end{enumerate}
\end{definition}

\begin{definition}[{\cite{Latremoliere18d}}]
  The \emph{extent}, $\tunnelextent{\tau,\tau'}$, of a metrical tunnel $(\tau,\tau')$ is given by
    \begin{equation*}
      \tunnelextent{\tau,\tau'} = \max\left\{\tunnelextent{\tau},\tunnelextent{\tau'}\right\} \text{.}
    \end{equation*}
\end{definition}

\begin{definition}[{\cite{Latremoliere18d}}]\label{metrical-prop-def}
  Fix $K\geq 1$. The \emph{metrical $K$-propinquity}, $\dmetpropinquity{K}(\mathds{A},\mathds{B})$, between two {\gQVB s} $\mathds{A}$ and $\mathds{B}$ is the nonnegative number given by
  \begin{equation*}
    \dmetpropinquity{K}(\mathds{A},\mathds{B}) = \inf\left\{\tunnelextent{\tau} : \text{ $\tau$ is a metrical $K$-tunnel from $\mathds{A}$ to $\mathds{B}$} \right\}\text{.}
  \end{equation*}
\end{definition}

\begin{theorem}[{\cite{Latremoliere18d}}]
  Fix $K \geq 1$. The metrical $K$-propinquity is a complete metric, up to full isometry, on the class of $K$-metrical C*-correspondence, where two metrical C*-correspondences
\begin{equation*}
  \left( \module{M}_j, \CDN_j, \B_j, \Lip_j, \A_j, \Lip'_j \right)
\end{equation*}
  (where $j\in\{1,2\}$) are fully isometric when:
  \begin{itemize}
  \item the underlying {\gQVB s} $(\module{M}_1,\CDN_1,\B_1,\Lip_1)$ and $(\module{M}_2,\allowbreak \CDN_2,\B_2,\Lip_2)$ are isometrically isomorphic,
  \item the {\qcms s} $(\A_1,\Lip'_1)$ and $(\A_2,\Lip'_2)$ are fully quantum isometric,
  \item we can choose the above quantum isometries so that the actions of $\A_1$ on $\module{M}_1$ and $\A_2$ on $\module{M}_2$ are intertwined by these morphisms.
  \end{itemize}
\end{theorem}

We should address the constant $K$ which is a part of the definition of the metrical propinquity. In general, the properties of our various propinquity metrics depend on the choice of some uniform Leibniz properties, common to all the spaces under considerations. There is some useful flexibility as to what this Leibniz property needs to be. For instance, in this work, we will see that $K = 4$, or anything larger, is an appropriate choice to accommodate our examples. The value of $K$ is not important, as long as we can find some value which works for the entire sequence we wish to prove converge in the sense of the propinquity. Moreover, there is an obvious relation between metrical propinquities for different values of $K$ --- the propinquity for a smaller $K$ dominates all the ones with a larger choice of $K$, making the picture easy to understand.

\medskip

We now can describe the last step in the definition of the convergence of spectral triples, i.e., the convergence of the associated quantum dynamics, which occur on different metrical C*-correspondences. Let us assume that we are given two metric spectral triples $(\A_1,\Hilbert_1,\Dirac_1)$ and $(\A_2,\Hilbert_2,\Dirac_2)$. In Equation (\ref{mcc-eq}), we have defined the metrical C*-correspondences $\mcc{\A_1}{\Hilbert_1}{\Dirac_1}$ and $\mcc{\A_2}{\Hilbert_2}{\Dirac_2}$, associated respectively with the metric spectral triples $(\A_1,\Hilbert_1,\Dirac_1)$ and $(\A_2,\Hilbert_2,\Dirac_2)$. Let $(\tau,\tau')$ be a metrical tunnel between $\mcc{\A_1}{\Hilbert_1}{\Dirac_1}$ and $\mcc{\A_2}{\Hilbert_2}{\Dirac_2}$.

    In particular, let us write $\tau = (\mathds{P},(\Phi_1,\phi_1),(\Phi_2,\phi_2))$ --- where $\mathds{P}$ is a {\gQVB} --- and note that by Definition \ref{metrical-tunnel-def} for metrical tunnels, $\tau$ is a modular tunnel  between $(\Hilbert_1,\CDN_1,\C,0)$ and $(\Hilbert_2,\CDN_2,\C,0)$, where $\CDN_1$ and $\CDN_2$ are, respectively, the graph norms of $\Dirac_1$ and $\Dirac_2$. Furthermore, let us write $\mathds{P} = (\mathscr{P},\CDN_\D,\D,\Lip_\D)$.

    Since $\Dirac_1$ and $\Dirac_2$ are self-adjoint, we can define two strongly continuous actions of $\R$ by unitaries on $\Hilbert_1$ and $\Hilbert_2$, by letting
    \begin{equation*}
      \forall j \in \{1,2\}, \quad \forall t \in \R, \quad T_j^t = \exp(i t \Dirac_j) \text{.}
    \end{equation*}

    The \emph{spectral propinquity} is defined by extending the metrical propinquity of Definition \ref{metrical-prop-def} to include the actions $T_1$ and $T_2$ of $\R$.
    
    With this in mind, let $\varepsilon > 0$ and assume that $\tunnelextent{\tau} \leq \varepsilon$. Let us call a pair $(\varsigma_1,\varsigma_2)$ of maps from $\R$ to $\R$ an $\varepsilon$-\emph{iso-iso}, for some $\varepsilon > 0$, whenever
    \begin{multline*}
      \forall \{j,k\} = \{1,2\}, \quad \forall x,y,z \in \left[-\frac{1}{\varepsilon},\frac{1}{\varepsilon}\right], \\ \big| |\left(\varsigma_j(x)+\varsigma_j(y)\right)-z| - |(x+y) - \varsigma_k(z)| \big| \leq \varepsilon \text{,}
    \end{multline*}
    and $\varsigma_1(0) = \varsigma_2(0) = 0$.
    
    As discussed in \cite{Latremoliere18b}, such maps can be used to define a distance on the class of proper monoids, but for our purpose, as we only work with the proper group $\R$, the definition simplifies somewhat. In fact, we only recall the definition so that we may properly define the spectral propinquity below: for our purposes, the only iso-iso map we will work with is simply the identity of $\R$.

    Thus, suppose that we are given an $\varepsilon$-iso-iso $(\varsigma_1,\varsigma_2)$ from $\R$ to $\R$, as above. We call $(\tau,\tau',\varsigma_1,\varsigma_2)$ an $\varepsilon$-covariant metrical tunnel.

The \emph{$\varepsilon$-covariant reach} of $(\tau,\varsigma_1,\varsigma_2)$ is then defined as follows:
\begin{multline*}
  \max_{\{j,s\} = \{1,2\}} \sup_{\substack{\xi \in \Hilbert_j\\ \CDN_j(\xi)\leq 1}}\inf_{\substack{\xi'\in\Hilbert_s \\ \CDN_s(\xi')\leq 1}}\sup_{|t| \leq \frac{1}{\varepsilon}} \\
  \sup_{\substack{\omega\in\module{P}\\\CDN_\D(\omega)\leq 1}} \left|\inner{T_j^t\xi}{\Pi_j(\omega)}{\Hilbert_j} - \inner{T_s^{\varsigma_s(t)}\xi'}{\Pi_s(\omega)}{\Hilbert_s}\right| \text{.}
\end{multline*}

We define the \emph{$\varepsilon$-magnitude} $\tunnelmagnitude{\tau,\tau',\varsigma_1,\varsigma_2}{\varepsilon}$ of an $\varepsilon$-covariant tunnel $(\tau,\tau',\varsigma_1,\varsigma_2)$ to be the maximum of the extent of $\tau$, the extent of $\tau'$, and the $\varepsilon$-covariant reach of $(\tau,\varsigma_1,\varsigma_2)$. 

The \emph{spectral propinquity}
\begin{equation*}
  \spectralpropinquity{}((\A_1,\Hilbert_1,\Dirac_1),(\A_2,\Hilbert_2,D_2))
\end{equation*}
is the nonnegative number
\begin{multline*}
  \min\bigg\{\frac{\sqrt{2}}{2},\inf\{\varepsilon>0 : \exists \text{ $\varepsilon$-covariant metrical tunnel $\tau$, }
  \\ \text{$\tau$ from $\mcc{\A_1}{\Hilbert_1}{\Dirac_1}$ to $\mcc{\A_2}{\Hilbert_2}{\Dirac_2}$ such that} \\ \tunnelmagnitude{\tau}{\varepsilon} \leq \varepsilon \} \bigg\} \text{.}
\end{multline*}

We refer to \cite{Latremoliere18b,Latremoliere18d} for a discussion of the fundamental properties of the covariant propinquity, including a discussion of sufficient conditions for completeness (on certain classes, including when the group is Abelian, such as in the present case). We record here the following property of the spectral propinquity.

\begin{theorem}[{\cite{Latremoliere18g}}]
  The spectral propinquity $\spectralpropinquity{}$ is a metric on the class of metric spectral triples, up to the following coincidence property: for any metric spectral triples $(\A_1,\Hilbert_1,\Dirac_1)$ and $(\A_2,\Hilbert_2,\Dirac_2)$,
  \begin{equation*}
    \spectralpropinquity{}((\A_1,\Hilbert_1,\Dirac_1),(\A_2,\Hilbert_2,\Dirac_2)) = 0
  \end{equation*}
  if and only if there exists a unitary map $V : \Hilbert_1\rightarrow \Hilbert_2$ such that $V\dom{\Dirac_1} = \dom{\Dirac_2}$,
  \begin{equation*}
    V \Dirac_1 V^\ast = \Dirac_2
  \end{equation*}
  and
  \begin{equation*}
    \mathrm{Ad}_V = V (\cdot) V^\ast \text{ is a *-isomorphism from $\A_1$ onto $\A_2$},
  \end{equation*}
  where (as is customary) we identify $\A_1$ and $\A_2$ with their images by their representations on $\Hilbert_1$ and $\Hilbert_2$, respectively. In particular, the *-isomorphism from $\A_1$ onto $\A_2$ implemented by the adjoint action of $V$ is a full quantum isometry from $(\A_1,\opnorm{[\Dirac_1,\cdot]}{}{\Hilbert_1})$ onto $(\A_2,\opnorm{[\Dirac_2,\cdot]}{}{\Hilbert_2})$.
\end{theorem}

An interesting example of convergence of spectral triples for the spectral propinquity is given by approximations of spectral triples on certain fractals by spectral triples on their natural approximating, finite, graphs \cite{Latremoliere20a}.

In this paper, \emph{we will prove that sequences of spectral triples on fuzzy tori converge, under very natural assumptions, to spectral triples on quantum tori, in the sense of the spectral propinquity.} Physically, this implies that the state spaces of the physical models described by fuzzy tori converge, as metric spaces, to the state spaces of the quantum tori, for the metric induced by the spectral triples, and that the quantum dynamics associated with these spectral triples converge as well. We refer to \cite{Latremoliere13,Latremoliere13b,Latremoliere16c,Latremoliere18b,Latremoliere18g} for a detailed discussion of this metric, but an informal understanding can be glanced with the following construction from these papers; we simply state that, informally, if two spectral triples are close, then, in particular, there exists a compact-set valued map from the Hilbert space of one of these spectral triples to the other, which behaves as a set-valued version of a unitary, as well as almost intertwining the actions by unitaries induced by the Dirac operators of the spectral triples.

\bigskip

The structure of the paper follows the strategy described above. We first set our notation by describing the formal background on quantum and fuzzy tori. We then introduce our spectral triples --- which include the proof that indeed, our proposed Dirac operators are self-adjoint with compact resolvent. We then establish core technical results about the metric properties of these spectral triples, including an analogue of the mean value theorem. This allows us to prove that our spectral triples are indeed, metric spectral triples. We can then prove the convergence of the quantum metrics induced by these spectral triples on fuzzy tori, for the propinquity. We then prove that the underlying Hilbert spaces of our spectral triples, seen as modules over fuzzy and quantum tori, converge as metrical C*-correspondence for the metrical propinquity. We conclude by proving the convergence of the quantum dynamics associated with our spectral triples.

\section{The Geometry of Quantum and Fuzzy Tori}

We begin with a description of fuzzy and quantum tori, and then introduce the spectral triples on fuzzy tori which we will consider, as well as the spectral triples on quantum tori which will be their limit for the spectral propinquity.

\subsection{Background: the quantum and fuzzy tori}

A torus is the Gelfand spectrum of the convolution C*-algebra of a finite product of infinite cyclic groups. A \emph{quantum torus} is a deformation of a classical torus, defined as a \emph{twisted} convolution algebra of a finite product of \emph{infinite} cyclic groups, or equivalently, as a deformation-quantization of the torus for certain Poisson structures \cite{Rieffel81,Rieffel90}. We adopt the former description. A \emph{fuzzy torus} is then a finite dimensional version of a quantum torus, namely, a twisted convolution C*-algebra for a finite product of \emph{finite} cyclic groups. Since our work in this paper focuses on spectral triples, we will give a presentation of quantum and fuzzy tori which stresses a particular *-representation for these C*-algebras.

\begin{notation}
  Let $\N_\ast = \N\setminus\{0,1\}$. Let $\Nbar_\ast = \N_\ast \cup\{ \infty \}$ be the one point compactification of $\N_\ast$.

  Let $d \in \N_\ast$. For all $k = (k(1),\ldots,k(d)) \in \Nbar_\ast^d$, we set:
  \begin{equation*}
    k\Z^d = \prod_{j=1}^d k(j) \Z \;\text{ and }\;  \Z^d_k = \bigslant{\Z^d}{k\Z^d} = \prod_{j=1}^d \bigslant{\Z}{k(j)\Z} \text{,}
  \end{equation*}
  with the convention that $\infty\Z = \{ 0 \}$, so that $\Z^d_{(\infty,\ldots,\infty)} = \Z^d$. We also write $\infty^d$ for $(\infty,\ldots,\infty) \in \Nbar_\ast^d$.
\end{notation}

A $2$-cocycle of a discrete Abelian group $G$, with values in the group $\T = \{ z \in \C : |z|=1 \}$ \cite{Mackey58}, is a map $\sigma : G \times G \mapsto \T$ such that
  \begin{equation}\label{cocycle-id}
    \forall x,y,z \in G \quad \sigma(x,y)\sigma(x+y,z) = \sigma(x,y+z)\sigma(y,z)\text,
  \end{equation}
  and $\sigma(e,x)=\sigma(x,e) =1$ for all $x\in G$, with $e\in G$ the unit of $G$. Two $2$-cocycles $\sigma$ and $\sigma'$ are cohomologous when there exists a function $f : G\rightarrow \T$ such that
  \begin{equation*}
    \forall x,y \in G \quad \sigma(x,y) = f(x)f(y)\overline{f(x+y)} \sigma'(x,y) \text.
  \end{equation*}
  In particular, a $2$-cocycle is trivial when it is cohomologous to the constant function $1$. Now, any $2$-cocycle of $G$ is cohomologous to a \emph{normalized} $2$-cocycle of $G$ \cite{Kleppner74}, where $\sigma$ is a normalized $2$-cocycle when
  \begin{equation*}
    \forall x \in G \quad \sigma(x,-x) = 1 \text.
  \end{equation*}

  Since two cohomologous $2$-cocycles of a discrete group $G$ give rise to *-isomorphic twisted convolution C*-algebras of $G$, we will only work with normalized $2$-cocycles of the groups $\Z^d_k$.

  We record that, if $\sigma$ is a normalized $2$-cocycle of $G$, then $\overline{\sigma(x,y)} = \sigma(-y,-x)$ for all $x,y \in G$ \cite{Kleppner74}.

We will follow a helpful convention when working with $2$-cocycles of quotient groups.

\begin{convention}\label{quotient-cocycle}
  Assume $H$ is a quotient of a discrete Abelian group $G$, and let $q : G \rightarrow H$ be the canonical surjection. If $\sigma$ is a normalized $2$-cocycle on $H$, then $(g,g') \in G^2 \mapsto \sigma(q(g),q(g'))$ is a normalized $2$-cocycle over $G$, which we will henceforth still denote by $\sigma$. Thus, we identify the space of all $2$-cocycle over $H$ with a subspace of $2$-cocycle over $G$.
\end{convention}

\begin{notation}\label{cocycle-notation}
  Let $d\in\N_\ast$. Let $\mathcal{C}^d_{\infty^d}$ be the space of all normalized $2$-cocycles $\sigma$ over $\Z^d$, endowed with the topology of pointwise convergence (which makes $\mathcal{C}^d_{\infty^d}$ compact, by Tychonoff theorem). For each $k\in \Nbar_\ast^d$, we let $\mathcal{C}^d_k$ be the space of of all normalized $2$-cocycles of $\Z^d_k$, seen as a subspace of $\mathcal{C}^d_{\infty^d}$ via Convention (\ref{quotient-cocycle}).

  We then let $\Xi^d = \left\{ (k,\sigma) \in \Nbar_\ast^d\times\mathcal{C}^d_{\infty^d} : \sigma\in\mathcal{C}^d_k \right\}$, endowed with the product topology. The space $\Xi^d$ is actually compact.
\end{notation}

We note that, for any $d \in \N_\ast$, using \cite[Theorem 7.1]{Kleppner65}, any $2$-cocycle of $\Z^d$ is cohomologous to a skew bicharacter, i.e. a $2$-cocycle of the form
\begin{equation}\label{sigma-eq}
  \cocycle{\theta} : (x,y) \in \Z^d \times \Z^d \longmapsto \exp\left( i \pi \inner{\theta x}{y}{} \right)
\end{equation}
for some real $d\times d$ antisymmetric matrix $\theta \in \alg{M}_d(\R)$, with the convention that we write elements in $\Z^d$ as $d\times 1$ matrices, and with $\inner{\cdot}{\cdot}{}$ the usual inner product of $\R^d$. Skew bicharacters are, in particular, normalized. However, we may not obtain such a convenient formula for all normalized $2$-cocycles of $\Z^d_k$ (up to cocycle equivalence) unless all components of $k$ are odds (since every element of $\Z^d_k$ can then be halved) --- otherwise, we could use a generic expression as found \cite{Kleppner74} which depends on a measurable choice of a square root function over $\C$, but we will not need it. This, however, brings up to the following remark.

\begin{remark}
  We will rely on the proofs in \cite{Latremoliere13c} later in this paper, and we choose our notations to match this reference. However, there is a small correction which we need to point out.

  in \cite[Notation 3.4]{Latremoliere13c}, the definition of $\Xi^d$ is stated improperly: it should be defined as we did in Notation (\ref{cocycle-notation}), and \emph{not} be restricted to pairs of the form $(k,\sigma)$ with $\sigma$ a skew bicharacter of $\Z^d_k$. Indeed, the entire work in \cite{Latremoliere13c} only uses the $2$-cocycle property and the fact that it is normalized.

  This unfortunate misprint of ours has no impact on \cite{Latremoliere13c}, with the single exception of one computation in the proof of \cite[Theorem 3.8]{Latremoliere13c}, which we reprove below as Lemma (\ref{projective-rep-lemma}) --- and \cite[Theorem 3.8]{Latremoliere13c} is true and well-known anyway. Thus, past this unfortunate terminology error of ours in \cite{Latremoliere13c}, we will use \cite{Latremoliere13c} unchanged.
\end{remark}

\bigskip

Quantum and fuzzy tori are twisted convolution C*-algebras of the groups $\Z_k^d$, for any $d\in\N_\ast$, for any $k \in \Nbar_\ast^d$ and for any $2$-cocycle $\sigma$ of $\Z^d_k$ --- and indeed, up to a *-isomorphism, all fuzzy and quantum tori are obtained using the $2$-cocycle $\sigma$ with $(k,\sigma)\in\Xi^d$. We begin our presentation of these C*-algebras by introducing certain natural projective unitary representations of $\Z^d_k$ on the Hilbert space $\ell^2(\Z^d_k)$, where we use the following notation.

\begin{notation}
For any (nonempty) set $E$ and any $p \in [1,\infty)$, the set $\ell^p(E)$ is the set of all absolutely $p$-summable (for the counting measure) complex valued functions over $E$, endowed with the norm:
\begin{equation*}
\|\xi\|_{\ell^p(E)} = \left(\sum_{x \in E} |\xi(x)|^p\right)^{\frac{1}{p}}
\end{equation*}
for all $\xi \in \ell^p(E)$.

For all $x \in E$, we define $\delta_x$ by
\begin{equation*}
  \delta_x : y \in E \mapsto \begin{cases}
    1 \text{ if $y=x$,} \\
    0 \text{ otherwise;}
  \end{cases}
\end{equation*}
of course $\delta_x \in \ell^p(\Z^d_k)$.  

Moreover, if $p = 2$ then $(\ell^2(E),\|\cdot\|_{\ell^2(E)})$ is a Hilbert space, where the inner product is given by
\begin{equation*}
  \forall \xi,\eta\in\ell^2\left(E\right) \quad\quad \inner{\xi}{\eta}{\ell^2(E)} = \sum_{x\in E} \overline{\xi(x)}\eta(x)
\end{equation*}
\end{notation}

\begin{lemma}\label{projective-rep-lemma}
  Let $d\in\N_\ast$ and $(k,\sigma) \in \Xi^d$. For all $m \in \Z^d_k$, the operator $W_{k,\sigma}^m$ defined, for all $\xi \in \ell^2(\Z^d_k)$, by:
  \begin{equation*}
    W_{k,\sigma}^m \xi : n \in \Z^d_k \mapsto \sigma(m,n-m) \xi(n-m) \text,
  \end{equation*}
  is a unitary operator on $\ell^2(\Z^d_k)$, and moreover, for all $m,n \in \Z^d_k$,
  \begin{equation}\label{w-commutation-eq}
    W_{k,\sigma}^m \, W_{k,\sigma}^n = \sigma(m,n) W_{k,\sigma}^{m+n} \text.
  \end{equation}
\end{lemma}

\begin{proof}
  By Expression (\ref{cocycle-id}), if $m,n,p \in \Z^d_k$, and if $\xi \in \ell^2(\Z^d_k)$, then
  \begin{align*}
    \left(W_{k,\sigma}^m \, W_{k,\sigma}^n \xi\right) (p)
    &= \sigma(m,p-m) \left(W_{k,\sigma}^n\xi\right)(p-m)\\
    &= \sigma(\underbracket[1pt]{m}_{\small =x},\underbracket[1pt]{p-m}_{\small = y+z}) \sigma(\underbracket[1pt]{n}_{\small =y},\underbracket[1pt]{p-m-n}_{\small =z}) \xi(p-m-n) \\
    &= \sigma(\underbracket[1pt]{m}_{\small x},\underbracket[1pt]{n}_{\small y}) \sigma(\underbracket[1pt]{n+m}_{\small =x+y},\underbracket[1pt]{p-(m+n)}_{\small z}) \xi(p-(m+n)) \\
    &= \left(\sigma(m,n) W_{k,\sigma}^{m+n} \xi\right)(p) \text.
  \end{align*}
  Hence, Expression (\ref{w-commutation-eq}) holds. An easy computation shows that $W_{k,\sigma}^0$ is the identity. Last, let $\xi,\eta\in\ell^2(\Z^d_k)$ and $m \in \Z^d_k$. We compute:
  \begin{align*}
    \inner{W_{k,\sigma}^m\xi}{\eta}{\ell^2(\Z^d_k)}
    &= \sum_{p \in \Z^d_k} \overline{\sigma(m,p-m)\xi(p-m)}\eta(p) \\
    &= \sum_{q \in \Z^d_k} \overline{\sigma(m,q)} \, \overline{\xi(q)}\eta(q+m) \\
    &= \sum_{q\in\Z^d_k} \sigma(-q,-m) \overline{\xi(q)} \eta(q+m) \\
    &= \sum_{q\in\Z^d_k} \overline{\xi(q)} \, \sigma(\underbracket[1pt]{-q}_{\small =x},\underbracket[1pt]{-m}_{\small =y}) \overbracket{\sigma(\underbracket[1pt]{-q-m}_{\small x+y},\underbracket[1pt]{q+m}_{\small =z})}^{\small =1} \eta(q+m) \\
    &= \sum_{q\in\Z^d_k} \overline{\xi(q)} \, \overbracket[1pt]{\sigma(\underbracket[1pt]{-q}_{\small x},\underbracket[1pt]{q}_{\small y+z})}^{\small =1} \sigma(\underbracket[1pt]{-m}_{\small y},\underbracket[1pt]{q+m}_{\small z}) \eta(q+m) \\
    &= \inner{\xi}{W_{k,\sigma}^{-m}\eta}{\ell^2(\Z^d_k)} \text.
  \end{align*}
  Therefore, $\left(W_{k,\sigma}^{m}\right)^\ast = W_{k,\sigma}^{-m}$; our lemma is thus proven.
\end{proof}

The quantum and fuzzy tori are the C*-algebras generated by the projective *-representations of $\Z^d_k$ defined in Lemma (\ref{projective-rep-lemma}).

\begin{definition}
Let $d\in\N_\ast$ and let $(k,\sigma) \in \Xi^d$. The C*-algebra $\qt{k}{\sigma}$ is the completion of the *-algebra generated by $\{ W_{k,\sigma}^m : m \in \Z^d_k \}$ (using the notation of Lemma (\ref{projective-rep-lemma})), for the operator norm $\opnorm{\cdot}{}{\ell^2(\Z^d_k)}$. 

When $k = \infty^d$, the C*-algebra $\qt{\infty^d}{\sigma} = C^\ast(\Z^d,\sigma)$ is called a \emph{quantum torus}. When $k\in\N_\ast^d$, the C*-algebra $\qt{k}{\sigma}$ is called a \emph{fuzzy torus}.
\end{definition}

\begin{remark}
  We are not aware of a common term to name the C*-algebras $\qt{k}{\sigma}$ when $k \in \Nbar_\ast^d$, with $k\notin\N_\ast^d$ and $k\neq \infty^d$, i.e. when $k$ contains some finite and some infinite values. Our work in this paper applies to these mixed objects just as well as to quantum tori and fuzzy tori.
\end{remark}

The C*-algebra $\qt{k}{\sigma}$, for any $(k,\sigma)\in\Xi^d$, is *-isomorphic to a C*-algebra constructed as a completion of $\ell^1(\Z^d_k)$, for an appropriate product, adjoint, and norm. This presentation of quantum and fuzzy tori will also be helpful.

\begin{lemma}\label{conv-lemma}
  Let $d\in\N_\ast$ and $(k,\sigma)\in \Xi^d$. We use the notation of Lemma (\ref{projective-rep-lemma}). For all $f \in \ell^1(\Z^d_k)$, the operator
  \begin{equation*}
    \pi_{k,\sigma}(f) = \sum_{m\in \Z^d_k} f(m) W_{k,\sigma}^m
  \end{equation*}
  is bounded on $\ell^2(\Z^d_k)$, with $\opnorm{\pi_{k,\sigma}(f)}{}{\ell^2(\Z^d_k)} \leq \norm{f}{\ell^1(\Z^d_k)}$, and moreover, $\pi_{k,\sigma}(f) \in \qt{k}{\sigma}$. Moreover,
  \begin{enumerate}
  \item $\pi_{k,\sigma}$ is injective;
  \item for all $f,g \in \ell^1(\Z^d_k)$,
    \begin{equation*}
      \pi_{k,\sigma}(f) \pi_{k,\sigma}(g) = \pi_{k,\sigma}(f\conv{k,\sigma} g)
    \end{equation*}
    where
    \begin{equation}\label{conv-eq}
      f \conv{k,\sigma} g : n \in \Z^d_k \mapsto \sum_{m \in \Z^d_k} f(m) g(n-m) \sigma(m,n-m) \text{;}
    \end{equation}
  \item for all $f \in \ell^1(\Z^d_k)$,
    \begin{equation*}
      \pi_{k,\sigma}(f)^\ast = \pi_{k,\sigma}(f^\ast)
    \end{equation*}
    where
    $f^\ast : m \in \Z^d_k \mapsto \overline{f(-m)}$.
  \end{enumerate}
\end{lemma}

\begin{proof}
  Since $W_{k,\sigma}^m$ is unitary for all $m\in \Z^d_k$, it is immediate that $\opnorm{\pi_{k,\sigma}(f)}{}{\ell^2(\Z^d_k)} \leq \norm{f}{\ell^1(\Z^d_k)}$; moreover it then follows immediately that $\pi_{k,\sigma}(f) \in \qt{k}{\sigma}$.
  If $\pi_{k,\sigma}(f) = 0$, then
  \begin{equation*}
    0 = \pi_{k,\sigma}(f)\delta_0 = \sum_{m\in\Z^d_k} f(m)\underbracket[1pt]{\sigma(m,0)}_{\small =1} \delta_{m}
  \end{equation*}
  so $\sum_{m\in\Z^d_k} |f(m)|^2 = \norm{\pi_{k,\sigma}(f)\delta_0}{\ell^2(\Z^d_k)}^2 = 0$, and thus $f = 0$. So $\pi_{k,\sigma}$ is injective. The other assertions follow from direct computations.
\end{proof}

\begin{remark}
  We note that, for any $d\in \N_\ast$ and $k\in\Nbar_\ast^d$, the adjoint operation of $(\ell^1(\Z^d_k),\conv{k,\sigma},\cdot^\ast)$ does \emph{not} depend on $\sigma$. This will be a very helpful property for us, and it follows from our choice to work with normalized $2$-cocycles.
\end{remark}

We now recall:
\begin{lemma}\label{pi-k-sigma-lemma}\cite{Zeller-Meier68}
  Let $d\in\N_\ast$, and let $(k,\sigma)\in\Xi^d$. Using the notation of Lemma (\ref{conv-lemma}), the triple $(\ell^1(\Z^d_k),\conv{k,\sigma},\cdot^\ast)$ is a Banach *-algebra. Moreover, the function
  \begin{equation*}
    f \in \ell^1\left(\Z^d_k\right) \mapsto \opnorm{\pi_{k,\sigma}(f)}{}{\ell^2\left(\Z^d_k\right)}
  \end{equation*}
  is a C*-norm on $(\ell^1(\Z^d_k),\conv{k,\sigma},\cdot^\ast)$; the completion of $(\ell^1(\Z^d_k),\conv{k,\sigma},\cdot^\ast)$ for this norm is a C*-algebra, and the unique extension of $\pi_{k,\sigma}$ to this completion is a *-isomorphism onto $\qt{k}{\sigma}$.
\end{lemma}

For all $d\in\N_\ast$, and for all $(k,\sigma)\in\Xi^d$, the C*-algebra $\qt{k}{\sigma}$ is finitely generated by a subset of the unitaries $\left\{ W_{k,\sigma}^m : m \in \Z^d_k \right\}$, subject to a natural commutation relation.

\begin{notation}\label{e-j-notation}
  Let $d \in \N_\ast$. For all $j \in \{1,\ldots,d\}$, we set
  \begin{equation*}
    e_j = \begin{pNiceMatrix}[last-col] 0 & \\ \vdots &  \\ 0 &  \\ 1 & \leftarrow j^{\mathrm{th}} \text{ row} \\ 0 & \\ \vdots & \\ 0 & \end{pNiceMatrix} \in \Z^d \text,
  \end{equation*}
 and we \emph{identify} $e_j$ with its class in $\Z^d_k$ for \emph{any} $k\in\Nbar_\ast^d$.
\end{notation}

Let $d \in \N_\ast$ and $(k,\sigma) \in \Xi^d$. For each $j\in\{1,\ldots,d\}$, note that $\pi_{k,\sigma}(\delta_{e_j}) = W_{k,\sigma}^{e_j}$. By Expression (\ref{w-commutation-eq}), we also note that for all $j,s\in\{1,\ldots,d\}$:
\begin{equation}\label{u-commutation-eq}
  W_{k,\sigma}^{e_j} W_{k,\sigma}^{e_s} = \sigma(e_j,e_s)\overline{\sigma(e_s,e_j)} W_{k,\sigma}^{e_s} W_{k,\sigma}^{e_j}  \text.
\end{equation}
By Expression (\ref{w-commutation-eq}), we conclude that $\qt{k}{\sigma}$ is the closure of the *-algebra generated by $\{W_{k,\sigma}^{e_1},\ldots,W_{k,\sigma}^{e_d}\}$. We thus refer to the unitaries $W_{k,\sigma}^{e_1}$,\ldots,$W_{k,\sigma}^{e_d}$ as the \emph{canonical unitaries} of $\qt{k}{\sigma}$.

\medskip

We now look at a few specific examples, and constructions, of quantum and fuzzy tori.

\begin{example}
 For all $n\in\Nbar_\ast$, if $k=\left(\underbrace{n,\ldots,n}_{d \text{ times}}\right)$, then the C*-algebra $\qt{k}{1}$ is the C*-algebra of $\C$-valued (continuous) functions over the finite group $\{ (z_1,\ldots,z_d)\in \C : z_j^n = 1 \}$. Moreover, $\qt{\infty^d}{1}$ is the C*-algebra $C(\T^d)$ of $\C$-valued, continuous functions over the $d$-torus $\T^d$.
\end{example}

On the other hand, quantum and fuzzy tori can be simple.

\begin{example}
  If $\theta\in\R\setminus\Q$ and if
  \begin{equation*}
    \sigma : m,m' \in \Z^2 \longmapsto \exp\left(i \pi \inner{\begin{pmatrix} 0 & \theta \\ -\theta & 0 \end{pmatrix}m}{m'}{}\right)
  \end{equation*}
  then $C^\ast(\Z^2,\sigma)$ is a simple C*-algebra \cite{Davidson}.
\end{example}

\begin{example}\label{Clock-Shift-odd-ex}
  Let $n\in\N_\ast$ be an \emph{odd} natural number, and write $n = 2 p + 1$ for $p\in\N$. The C*-algebra $\A_n = \qtd{2}{(n,n)}{\sigma_n}$, with
  \begin{equation*}
    \sigma_n : m,m' \mapsto \exp\left( 2 i \pi \inner{ \begin{pmatrix} 0 & -\frac{p+1}{n} \\ \frac{p+1}{n} & 0 \end{pmatrix} m}{m'}{} \right)
  \end{equation*}
  is the C*-algebra $\A_n$ generated by the canonical unitaries $W_{(n,n),\sigma_n}^{e_1}$, $W_{(n,n),\sigma_n}^{e_2}$ subject to the commutation relation:
  \begin{equation*}
    W_{(n,n),\sigma_n}^{e_1} W_{(n,n),\sigma_n}^{e_2} = \exp\left(\frac{2i\pi}{n}\right) W_{(n,n),\sigma_n}^{e_2} W_{(n,n),\sigma_n}^{e_1} \text,
  \end{equation*}
  since $2p \equiv 1 \mod n$, and $\left(W_{(n,n),\sigma_n}^{e_j}\right)^n = 1$ for each $j\in\{1,2\}$. So there exists a *-isomorphism from $\A_n$ to the C*-algebra generated by the clock and shift matrices $S_n$ and $C_n$ (see Expression (\ref{Clock-Shift-eq})), sending $W_{(n,n),\sigma_n}^{e_1}$ to $S_n$, and $W_{(n,n),\sigma_n}^{e_2}$ to $C_n$.

  We note, as will prove helpful later, that $(\sigma_n)_{n\in\N}$, thus defined, converges to $((m_1,m_2),(m'_1,m'_2))\in\Z^2\times\Z^2 \mapsto(-1)^{m_1 m'_2 - m'_1 m_2}$ in $\mathcal{C}_{\infty^2}^2$, which, in turn, is cohomologous to the trivial cocycle $1$.
\end{example}

In this work, we will construct approximations of quantum tori by fuzzy tori. The heuristics behind our approximations, very informally stated, is that we approximate a quantum torus $\A$ by fuzzy tori, whose canonical unitaries satisfy ``almost'' the commutation relation between the canonical unitaries defining $\A$. The first, and important, step toward our result is the following result.

\begin{theorem}\label{Cocycle-approx-thm}
  Let $d\in \N_\ast$, and let $k\in \Nbar_\ast^d$. If $\theta$ is an antisymmetric $d\times d$ real matrix such that $\inner{\theta m}{m'}{} \in \Z$ for all $m,m' \in k\Z^d$, then there exists a normalized $2$-cocycle $\sigma$ of $\Z^d_k$ such that, for all $m,n \in \Z^d_k$
  \begin{equation*}
    W_{k,\sigma}^m W_{k,\sigma}^{n} = \exp\left(2i\pi\inner{\theta m}{n}{}\right) W_{k,\sigma}^n W_{k,\sigma}^m \text.
  \end{equation*}
  Moreover, if $\theta$ is an antisymmetric $d\times d$ matrix, if $(k_n)_{n\in\N}$ is any sequence in $\N_\ast^d$ such that $\lim_{n\rightarrow\infty} k_n = \infty^d$, and if $(\theta_n)_{n\in\N}$ is a sequence of $d\times d$ antisymmetric matrices such that $\lim_{n\rightarrow\infty} \theta_n = \theta$, and
  \begin{equation*}
    \forall n\in\N \quad \left\{ \inner{\theta_n m}{m'}{} : m,m' \in k_n \Z^d \right\} \subseteq \Z\text,
  \end{equation*}
  then there exists a sequence $(k_n,\sigma_n)_{n\in\N}$ in $\Xi^d$, and a normalized $2$-cocycle $\sigma'$ of $\Z^d$ cohomologous to:
  \begin{equation*}
    \cocycle{\theta}: (m,m')\in\Z^d\times\Z^d \mapsto \exp\left(i\pi \inner{\theta m}{m'}{}\right)\text,
  \end{equation*}
  such that, for all $n\in\N$:
  \begin{equation}\label{main-approx-comm-eq}
    W_{k_n,\sigma_n}^m W_{k_n,\sigma_n}^{m'} = \exp\left(2i\pi\inner{\theta_n m}{m'}{}\right) W_{k_n,\sigma_n}^{m'} W_{k_n,\sigma_n}^{m}
  \end{equation}
  and $\lim_{n\rightarrow\infty} (k_n,\sigma_n) = (\infty^d,\sigma')$.
\end{theorem}

\begin{proof}
  Let $\omega$ be the unique upper triangular matrix, with zero diagonal, such that $\theta=\omega-\omega^\intercal$, where $\omega^\intercal$ is the transpose matrix of $\omega$.

  Let $\varsigma$ be the bicharacter of $\Z^d_k$ induced by the bicharacter of $\Z^d$:
  \begin{equation*}
    (m,m')\in\Z^d\times\Z^d \mapsto \exp\left(2i\pi\inner{\omega m}{m'}{}\right)
  \end{equation*}
  using Convention (\ref{quotient-cocycle}). Now, let $\beta \in \R$ such that $\beta\notin \Z + \{ \inner{\theta m}{m'}{} : m,m' \in \Z^d \}$. Let $\mathrm{sqrt}$ be the function which maps $r\exp(2i\pi t)$, with $r > 0$ and $t \in (\beta,\beta+1)$, to $\sqrt{r} \exp(i\pi t)$. We then let
  \begin{equation*}
    f : m \in \Z^d_k \mapsto \overline{\mathrm{sqrt}(\varsigma(m,-m))} \text.
  \end{equation*}
  We then define
  \begin{equation*}
    \sigma:(m,m')\in\Z^d_k\times\Z^d_k \mapsto f(m)f(m')\overline{f(m+m')} \varsigma(m,m') \text.
  \end{equation*}
  By construction, $\sigma$ is a normalized $2$-cocycle of $\Z^d_k$, cohomologous to $\varsigma$. Moreover, we compute that:
  \begin{align*}
    W_{k,\sigma}^m W_{k,\sigma}^{m'}
    &= \sigma(m,m')\overline{\sigma(m',m)} W_{k,\sigma}^{m'} W_{k,\sigma}^m \\
    &= f(m)f(m')\overline{f(m+m')} \varsigma(m,m') \\
    &\quad \quad \overline{f(m) f(m')} f(m+m') \overline{\varsigma(m',m)} W_{k,\sigma}^{m'} W_{k,\sigma}^m \\
    &= \varsigma(m,m')\overline{\varsigma(m',m)} W_{k,\sigma}^{m'} W_{k,\sigma}^m \\
    &= \exp\left(2i\pi\left( \inner{\omega m}{m'}{} - \inner{\omega m'}{m}{} \right) \right) W_{k,\sigma}^{m'} W_{k,\sigma}^m \\
    &= \exp\left(2i\pi\left( \inner{(\omega-\omega^\intercal) m}{m'}{} \right) \right) W_{k,\sigma}^{m'} W_{k,\sigma}^m \\
    &= \exp\left(2i\pi\left( \inner{\theta m}{m'}{} \right) \right) W_{k,\sigma}^{m'} W_{k,\sigma}^m \text,
  \end{align*}
  as desired.

  Now, assume $k = \infty^d$ --- so $\theta$ is any antisymmetric $d\times d$ matrix. Fix:
  \begin{equation*}
    \beta \in \R \setminus \left(\Q + \{ \inner{\theta m}{m'}{} : m,m' \in \Z^d\}  \right)\text.
  \end{equation*}

  Let $(k_n)_{n\in\N}$ and $(\theta_n)_{n\in\N}$ be given, as in the hypothesis of our theorem. Note that, in particular, $\inner{\theta_n m}{m'}{} \in \Q$ for all $m,m' \in \Z^d$ and for all $n\in\N$. For each $n\in\N$, let $\sigma_n$ be the normalized $2$-cocycle of $\Z^d_{k_n}$, constructed from the matrix $\theta_n$, as in the first part of the proof, for the same fixed $\beta$ we have chosen. In particular, the commutation relation (\ref{main-approx-comm-eq}) holds for each $n\in\N$.

  Last, let:
  \begin{multline*}
    \sigma' :(m,m') \in \Z^d\times\Z^d \mapsto \overline{\mathrm{sqrt}(\varsigma(m,-m))\mathrm{sqrt}(\varsigma(m',-m'))}\\ \mathrm{sqrt}(\varsigma(m+m',-(m+m'))) \varsigma(m,m') \text.
  \end{multline*}

  Since $\mathrm{sqrt}$ is continuous on $\C\setminus\{ r\exp(2i\pi\beta) : r\geq 0\}$ by construction, we then note that, for any $m,m' \in \Z^d$, since $\lim_{n\rightarrow\infty}\exp\left(2i\pi\inner{\theta_{k_n} m}{m'}{}\right) = \exp(2i\pi\inner{\theta m}{m'}{})$, we also have $\lim_{n\rightarrow\infty}\sigma_{k_n}(m,m') = \sigma'(m,m')$. It remains to show that $\sigma'$ is cohomologous to $\cocycle{\theta}$.

  For each $m\in\Z^d$, let $s(m)\in\{-1,1\}$ be defined by
  \begin{equation*}
    \mathrm{sqrt}(\exp(2i\pi \inner{\omega m}{-m}{})) = s(m) \exp(i\pi \inner{\omega m}{-m}{}) \text.
  \end{equation*}
  We thus compute, for all $m,m' \in \Z^d$:
  \begin{multline*}
    s(m)s(m')s(m+m')\sigma'(m,m') \\
    \begin{split}
      &= \overline{s(m)\mathrm{sqrt}(\varsigma(m,-m))}\overline{s(m')\mathrm{sqrt}(\varsigma(m',-m'))} \\
      &\quad\quad s(m+m')\mathrm{sqrt}(\varsigma(m+m',-m-m')) \varsigma(m,m') \\
      &=\exp(i\pi\inner{\omega m}{m}{})\exp(i\pi\inner{\omega m'}{m'}{}) \exp(-i\pi\inner{\omega(m+m')}{(m+m')}{}) \\
      &\quad\quad \exp(2i\pi\inner{\omega m}{m'}{}) \\
      &=\exp(i\pi\inner{\theta m}{m'}{}) = \sigma(m,m') \text.
    \end{split}
  \end{multline*}
  Thus, $\cocycle{\theta}$ and $\sigma'$ are cohomologous, as desired.
\end{proof}

By \cite{Kleppner65}, every $2$-cocycle of $\Z^d$ is cohomologous to a skew bicharacter of the form $\cocycle{\theta}$, for an antisymmetric $d\times d$ matrix, using the notation of Theorem (\ref{Cocycle-approx-thm}). Thus, Theorem (\ref{Cocycle-approx-thm}) may be used to find approximations of any $2$-cocycle of $\Z^d$, up to cohomology.

We now can give a few more relevant examples of fuzzy tori, illustrating Theorem (\ref{Cocycle-approx-thm}).

\begin{example}\label{Clock-Shift-ex}
  Let $x\in\R\setminus\Q$. Let $\mathrm{sqrt}$ be the branch of the square root over $\C$, defined by sending any complex number $r \exp(2i\pi t)$, with $t \in (x,x+1)$ and $r> 0$, to $\sqrt{r}\exp(i\pi t)$. Note that $\mathrm{sqrt}$ is continuous at $1$.

  For each $n\in\N_\ast$, let $\varsigma_n$ be the $2$-cocycle on $\Z^2_{(n,n)}$ given by using Convention (\ref{quotient-cocycle}) on the $2$-cocycle
  \begin{equation*}
    m,m' \in \Z^2\times \Z^2 \mapsto \exp\left(2 i \pi \inner{\begin{pmatrix} 0 & \frac{-1}{n} \\ 0 & 0 \end{pmatrix}m}{m'}{}\right) \text.
  \end{equation*}

  For all $m,m' \in \Z^2_{(n,n)}$, we then set, following \cite{Kleppner74}:
  \begin{equation*}
    \sigma_n(m,m') = \overline{\mathrm{sqrt}(\varsigma_n(m,-m))\mathrm{sqrt}(\varsigma_n(m',-m'))}\mathrm{sqrt}(\varsigma_n(m+m',-m-m')) \varsigma_n(m,m') \text.
  \end{equation*}
  
  The map $\sigma_n$ is a normalized $2$-cocycle of $\Z^2_{(n,n)}$. Moreover, by construction, $(\sigma_n)_{n\in\N}$ converges to $1$ in $\mathcal{C}^2_{\infty^2}$.

  Last, a quick computation shows that
  \begin{equation*}
    W_{(n,n),\sigma_n}^{e_1} W_{(n,n),\sigma_n}^{e_2} = \exp\left(\frac{2i\pi}{n}\right)W_{(n,n),\sigma_n}^{e_2} W_{(n,n),\sigma_n}^{e_1} \text,
  \end{equation*}
  while $\left(W_{(n,n),\sigma_n}^{e_1}\right)^n = \left(W_{(n,n),\sigma_n}^{e_1}\right)^n = 1$. Thus, we define a (faithful) *-representation of $\qtd{2}{(n,n)}{\sigma_n}$ by sending $W_{(n,n),\sigma_n}^{e_1}$ to $S_n$ and $W_{(n,n),\sigma_n}^{e_2}$ to $C_n$, where we used the notation of Expression (\ref{Clock-Shift-eq}).
\end{example}

In this paper, we will construct operators modeled after the Dirac operator construction from Riemannian geometry, and thus, we will make use of Clifford algebras. These algebras are actually examples of fuzzy tori, and this will prove helpful. We will invoke, for this example, the following universal property of quantum and fuzzy tori.

\begin{theorem}[{\cite{Zeller-Meier68}}]
  Let $d\in\N_\ast$ and $k = (k(1),\ldots,k(d)) \in\Nbar_\ast^d$. Let $\theta$ be an antisymmetric $d\times d$ matrix such that
  \begin{equation*}
    \left\{ \inner{\theta m}{m'}{} : m,m' \in k\Z^d \right\} \subseteq \Z \text.
  \end{equation*}
  Let $\sigma$ be a normalized $2$-cocycle of $\Z^d_k$ such that, for all $m,m' \in \Z^d_k$, the following commutation relation holds:
  \begin{equation*}
    W_{k,\sigma}^{m} W_{k,\sigma}^{m'} = \exp\left(2i\pi\inner{\theta m}{m'}{}\right) W_{k,\sigma}^{m'} W_{k,\sigma}^m \text.
  \end{equation*}
  Note that such a cocycle of $\Z^d_k$ exists by Theorem (\ref{Cocycle-approx-thm}).
  
  For all family $V_1$,\ldots,$V_d$ of $d$ unitaries in a C*-algebra, such that, for all $j,s \in \{1,\ldots,d\}$,
  \begin{equation*}
    V_j^{k(j)} = 1\text{, and } V_j V_s = \exp(2i\pi\inner{\theta e_j}{e_s}{}) V_s V_j \text,
  \end{equation*}
  there exists a *-morphism $\pi$ from $\qt{k}{\sigma}$ to $C^\ast(V_1,\ldots,V_d)$ such that $\pi(W_{k,\sigma}^{e_j}) = V_j$ for all $j\in\{1,\ldots,d\}$.
\end{theorem}

\begin{example}\label{Clifford-ex}
  Let $d\in\N_\ast$, and write $k = (2,\ldots,2)\in\N_\ast^d$. By Theorem (\ref{Cocycle-approx-thm}), there exists a normalized $2$-cocycle $\varsigma$ of $\Z^d_k$ such that
  \begin{equation*}
    \forall m,m'\in\Z^d_k \quad W_{k,\varsigma}^{m}W_{k,\varsigma}^{m'} = \exp\left(i\pi\inner{\begin{pNiceMatrix}
          0       & 1      & \Cdots & 1 \\
          -1        & \Ddots & \Ddots & \Vdots \\
          \Vdots         &  \Ddots      &   &  1            \\
          -1        & \Cdots & -1 &  0 
        \end{pNiceMatrix}m}{m'}{} \right) W_{k,\varsigma}^{m'}W_{k,\varsigma}^{m} \text.
  \end{equation*}
  Thus, $C^\ast(\Z^d_k,\varsigma)$ is the universal C*-algebra generated by $d$ unitaries $W_{k,\varsigma}^{e_1}$,\ldots,$W_{k,\varsigma}^{e_d}$ such that, for all $j,s \in \{1,\ldots,d\}$:
  \begin{equation*}
    W_{k,\varsigma}^{e_j} W_{k,\varsigma}^{e_s} + W_{k,\varsigma}^{e_s} W_{k,\varsigma}^{e_j} = \begin{cases} 2 \text{ if $j=s$,} \\ 0 \text{ otherwise,} \end{cases}
  \end{equation*}
  and $\left(W_{k,\varsigma}^{e_j}\right)^2 = 1$.

  The fuzzy torus $\qt{k}{\varsigma}$ is, therefore, the \emph{Clifford algebra} $\alg{Cl}(\C^d)$ of $\C^d$, since $W_{k,\varsigma}^{e_1}$,\ldots,$W_{k,\varsigma}^{e_d}$ satisfy the universal conditions of the generators of $\alg{Cl}(\C^d)$. By the universal properties of both $\qtd{d}{k}{\varsigma}$ and $\alg{Cl}(\C^d)$, the two algebras are thus isomorphic as associative algebras.

  We will make use of this identification. We will however use the notation
  \begin{equation*}
    \forall j \in \{1,\ldots,d\} \quad \gamma_j = i W_{k,\varsigma}^{e_j} \text,
  \end{equation*}
  so that
  \begin{equation*}
    \forall j,s \in \{1,\ldots,d\} \quad \gamma_j \gamma_s + \gamma_s \gamma_j =
    \begin{cases}
      -2 \text{ if $j=s$,}\\
      0 \text{ if $j\neq s$,}
    \end{cases}
  \end{equation*}
  as it will simplify some notation further in this work --- of course, $\qt{k}{\varsigma} = C^\ast(\gamma_1,\ldots,\gamma_d)$. We note that $\gamma_j^\ast = -\gamma_j$ for all $j\in\{1,\ldots,d\}$.
\end{example}

\bigskip

Let $(k,\sigma)\in\Xi^d$. The representation $\pi_{k,\sigma}$ of Theorem (\ref{pi-k-sigma-lemma}) is actually simply the GNS representation of $\qt{k}{\sigma}$ for the trace which is the unique continuous extension of $f \in \ell^1(\Z^d_k) \mapsto f(0)$. Even more specifically, if $f \in \ell^1(\Z^d_k)$, and if $\xi \in \ell^2(\Z^d_k)$, then, by Young's inequality, the product $f\conv{k,\sigma} \xi$ is again in $\ell^2(\Z^d_k)$, and
\begin{equation*}
  \pi_{k,\sigma}(f)\xi = f\conv{k,\sigma} \xi \text.
\end{equation*}
From this observation, we note that $\ell^2(\Z^d_k)$ is a natural bimodule over the C*-algebra $\qt{k}{\sigma}$.

\begin{lemma}
  Let $(k,\sigma) \in \Xi^d$. For all $\xi \in \ell^2(\Z^d_k)$, let
  \begin{equation*}
    J_k \xi : m \in \Z^d_k \mapsto \overline{\xi(-m)}\text.
  \end{equation*}
  The map $J_k : \ell^2(\Z^d_k)\rightarrow\ell^2(\Z^d_k)$ is a conjugate linear, involutive, isometry. Moreover, if we set, for all bounded linear operator $a$ on $\ell^2(\Z^d_k)$, and for all $\xi \in \ell^2(\Z^d_k)$:
  \begin{equation*}
    \xi \cdot a = J_k a^\ast J_k \xi \text,
  \end{equation*}
  then, in particular,
  \begin{equation*}
    a \in \qt{k}{\sigma} \longmapsto \left( \xi\in\ell^2(\Z^d_k)\mapsto \xi\cdot a \right)
  \end{equation*}
  is a right action of $\qt{k}{\sigma}$ on $\ell^2(\Z^d_k)$ (by bounded linear operators) such that
  \begin{equation*}
    \forall a,b \in \qt{k}{\sigma} \quad \forall \xi \in \ell^2(\Z^d_k) \quad (a\xi)\cdot b = a(\xi \cdot b) \text,
  \end{equation*}
  i.e., $\ell^2(\Z^d_k)$ is a bimodule over $\qt{k}{\sigma}$. Moreover, the adjoint of $\xi\in\ell^2(\Z^d_k)\mapsto a\cdot \xi$ is $\xi \in \ell^2(\Z^d)\mapsto \xi \cdot (a^\ast)$.

  Last, if $f \in \ell^1(\Z^d_k)$ and $g\in \ell^2(\Z^d_k)$, then
  \begin{equation*}
    g \cdot \pi_{k,\sigma}(f) = g \conv{k,\sigma} f \in \ell^2(\Z^d_k) \text.
  \end{equation*}
\end{lemma}

\begin{proof}
  It is immediate that $J_k$ is involutive, conjugate linear, and isometric on $\ell^2(\Z^d_k)$. It is then easy to check that $\xi\in\ell^2(\Z^d_k)\mapsto \xi\cdot a$ is linear, for all bounded linear operator $a$ on $\ell^2(\Z^d_k)$. We thus compute, for all bounded linear operators $a,b$ on $\ell^2(\Z^d_k)$, and for all $\xi \in \ell^2(\Z^d_k)$,
  \begin{equation*}
    \xi\cdot (ab) = J_k (ab)^\ast J_k \xi = J_k b^\ast J_k^2 a^\ast J_k \xi = \left(\xi\cdot a\right) \cdot b \text,
  \end{equation*}
  and, for all $\lambda\in\C$,
  \begin{equation*}
    \xi\cdot(\lambda a + b) = J_k(\lambda a + b)^\ast J_k \xi = \lambda J_k a^\ast J_k\xi + J_k b^\ast J_k\xi = \lambda \xi \cdot a + \xi \cdot b \text.
  \end{equation*}

  Therefore,
  \begin{equation*}
    a\in\qt{k}{\sigma} \longmapsto \left( \xi \in \ell^2(\Z^d_k) \mapsto \xi\cdot a \right)
  \end{equation*}
  is a right action of $\qt{k}{\sigma}$ on $\ell^2(\Z^d_k)$ by bounded linear operator.
  Of course, $(J_k a^\ast J_k)^\ast = J_k (a^\ast)^\ast J_k$ for all $a\in \qt{k}{\sigma}$.

   Let $f,g  \in \ell^1(\Z^d_k)$. Note that $J_k f=f^\ast$ and $J_k g = g^\ast$. We thus compute:
    \begin{align*}
      g\cdot\pi_{k,\sigma}(f)
      &= J_k \pi_{k,\sigma}(f^\ast) J_k g \\
      &= \left( f^\ast \conv{k,\sigma} g^\ast \right)^\ast \\
      &= g \conv{k,\sigma} f \text.
    \end{align*}
    This relation extends to $g \in \ell^2(\Z^d_k)$, and gives a vector in $\ell^2(\Z^d_k)$ by Young's inequality (or simply by continuity of the map $\xi\in\ell^2(\Z^d_k)\mapsto \xi\cdot \pi_{k,\sigma}(f)$).

    Therefore, for all $f,g \in \ell^1(\Z^d_k)$ and $\xi \in \ell^2(\Z^d_k)$, we compute:
    \begin{align*}
      \left(\pi_{k,\sigma}(f) \xi\right) \cdot \pi_{k,\sigma}(g)
      &= \left( f \conv{k,\sigma} \xi \right) \conv{k,\sigma} g \\
      &= f\conv{k,\sigma} \left(\xi \conv{k,\sigma} g\right) \\
      &= \pi_{k,\sigma}(f)(\xi\cdot \pi_{k,\sigma}(g)) \text.
    \end{align*}

    Therefore, by continuity, for all $a,b \in \qt{k}{\sigma}$, for all $\xi \in \ell^2(\Z^d_k)$, we conclude $(a\xi)\cdot b = a(\xi\cdot b)$. This completes our proof.
\end{proof}

\begin{notation}
  If $\xi \in \ell^2(\Z^d_k)$, and if $a\in\qt{k}{\sigma}$, then we set
  \begin{equation*}
    \left[ a, \xi \right] = a\xi - \xi\cdot a \text.
  \end{equation*}

  Note that, if $f,\xi \in \ell^1(\Z^d_k)$, then
  \begin{equation*}
    [\pi_{k,\sigma}(f),\xi] = f\conv{k,\sigma} \xi - \xi \conv{k,\sigma} f
  \end{equation*}
  so there will be little risk of confusion by using the same notation for this commutator and the usual commutator on quantum and fuzzy tori.
\end{notation}

\bigskip

Quantum tori are prototypes of noncommutative manifolds, whose geometry \cite{Connes80} derives from a natural action of torus $\T^d$ on quantum tori, called the \emph{dual action}. The dual action of the Lie group $\T^d$ induces an action of the Lie algebra $\R^d$ of $\T^d$ by *-derivations on quantum tori. Thus, the starting point for the geometric considerations in this paper are the following actions of the group $\U_k^d$ on fuzzy and quantum tori. As we will construct a spectral triple on both quantum and fuzzy tori, we start with natural actions of a closed subgroup of $\T^d$ on the Hilbert spaces $\ell^2(\Z^d_{k})$, which then defines the dual action on our C*-algebras by conjugation.

\begin{notation}    
  Let $k=(k(1),\ldots,k(d)) \in \Nbar_\ast^d$. The Pontryagin dual $\widehat{\Z^d_k}$ of $\Z^d_k$ is identified with the closed subgroup
  \begin{equation*}
    \U_k^d = \left\{ (z_1,\ldots,z_d) \in \T^d : \forall j \in \{1,\ldots,d\} \quad k(j) < \infty \implies z_j^{k(j)} = 1 \right\} \text,
  \end{equation*}
  of the $d$-torus $\T^d$, where $\T = \{ z \in \C : |z| = 1 \}$. To make our identification explicit, we will use a simple notation for the dual pairing between $\Z^d_k$ and $\U^d_k$. If $z = (z_1,\ldots,z_d) \in \T^d$ and $n = (n_1,\ldots,n_d) \in \Z^d$, then we set:
  \begin{equation*}
    z^n = \prod_{j=1}^d z_j^{n_j} \text{.}
  \end{equation*}
  Now, if $z \in \U^d_k$ and $n\in\Z^d_k$, then, for any $w,w'\in\Z^d$ such that $n = w + k\Z^d = w' + k\Z^d$, we easily observe that $z^w = z^{w'}$, and we denote this element of $\U^d_k$ simply by $z^n$.

  Every character of $\Z^d_k$ is of the form $\chi_z : n\in\Z^d_k \mapsto z^n$ for a unique $z \in \U^d_k$, and the map $z\in\U_k^d\mapsto \chi_z \in \widehat{\Z^d_k}$ is indeed a topological group isomorphism, as can easily be checked.
\end{notation}

\begin{notation}\label{v-notation}
  Let $d\in\N_\ast$ and $k \in \Nbar_\ast^d$. For all $z\in\U^d_k$ and $\xi \in \ell^2(\Z^d_k)$, we define
  \begin{equation*}
    v_k^z \xi : m \in \Z^d_k \mapsto z^{m} \xi(m) \text.
  \end{equation*}
  The map $z\in\U^d_k \mapsto v_k^z$ thus defined is a strongly continuous action of the group $\U^d_k$ on $\ell^2(\Z^d_k)$ by unitaries.
\end{notation}

The dual action of $\U^d_k$ on $\qt{k}{\sigma}$ is defined by conjugation with the action $v_k$. First, note that $v_k$ restricted to $\ell^1(\widehat{G_n}) \subseteq \ell^2(\widehat{G_n})$ is an isometry of $\ell^1(\widehat{G_n})$. With this in mind:

\begin{lemma}
  Let $d\in\N_\ast$. Let $(k,\sigma)\in\Xi^d$. For all $f \in \ell^1(\Z^d_k)$,
  \begin{equation*}
    v_k^z \, \pi_{k,\sigma}(f) \, v_k^{\overline{z}} = \pi_{k,\sigma}(v_k^z(f))
  \end{equation*}
  The restriction of $v_k^z$ to $\ell^1(\widehat{G_n})$ thus defined is a *-automorphism of $(\ell^1(\Z^d_k),\conv{k,\sigma},\cdot^\ast)$, and the map $z\in \U^d_k \mapsto v_k^z$ is a strongly continuous action of $\U^d_k$ on $\ell^1(\Z^d_k)$.
\end{lemma}

\begin{proof}
  Let $z\in \U_k^d$, and $p \in \ell^2(\Z^d_k)$. By definition, we thus compute:
  \begin{align*}
    v_k^z \, W_{k,\sigma}^m \, v_k^{\overline{z}} \; \xi(p)
    &= z^p \; W_{k,\sigma}^m \, v_k^{\overline{z}} \; \xi(p) \\
    &= z^p \sigma(m,p-m) (v_k^{\overline{z}}\xi)(p-m) \\
    &= z^p \sigma(m,p-m) z^{m-p} \xi(p-m) = z^m W_{k,\sigma}^m \, \xi(p) \text.
  \end{align*}
  The rest of the lemma is proven with similarly easy computations.
\end{proof}

\begin{notation}
  The unit of $\qt{k}{\sigma}$ --- which is the identity on $\ell^2(\Z^d_k)$ --- for any $d\in \N_\ast$, and for any $(k,\sigma)\in\Xi^d$, is denoted by $\unit_k$, or even $\unit$ if no confusion may arise.
\end{notation}

\begin{corollary}\cite{Zeller-Meier68}\label{dual-action-thm}
  Let $d\in\N_\ast$. Let $(k,\sigma)\in\Xi^d$. For all $a\in \qt{k}{\sigma}$, the element
  \begin{equation*}
    \alpha_{k,\sigma}^z(a) = v_k^z \, a \, v_k^{\overline{z}}
  \end{equation*}
  is in $\qt{k}{\sigma}$. Thus defined, $z\in\U^d_k\mapsto \alpha_{k,\sigma}^z$ is a strongly continuous action of $\U^d_k$ on $\qt{k}{\sigma}$ by *-automorphisms, called the \emph{dual action} of $\U^d_k$  on $\qt{k}{\sigma}$.
  The action $\alpha_{k,\sigma}$ thus defined is ergodic, i.e.
  \begin{equation*}
    \left\{ a \in \qt{k}{\sigma} : \forall z \in \U^d_k \quad \alpha_{k,\sigma}^z(a) = a \right\} = \C \unit_k \text.
  \end{equation*}
\end{corollary}

\medskip

Let now $k = (k(1),\ldots,k(d)) \in \Nbar_\ast^d$, for $d\in\N_\ast$. The action of the Lie group $\U^d_k$ on $\ell^2(\Z^d_k)$ defines, in turn, a natural action by its Lie algebra. The Lie algebra $\alg{u}^d_k$ of the Lie group $\U^d_k$ is given by:
\begin{equation*}
  \alg{u}^d_k = \left\{ (x_1,\ldots,x_d) \in \R^d : \forall j \in \{1,\ldots,d\} \quad k(j) < \infty \implies x_j = 0 \right\}\text{,}
\end{equation*}
with the exponential function given by:
\begin{equation*}
  \exp_{\U^d_k} : (x_1,\ldots,x_d)\in\alg{u}^d_k \mapsto \left(\exp(i x_1),\ldots,\exp(i x_d)\right)\text{.}
\end{equation*}

Using the actions defined in Theorem (\ref{dual-action-thm}), we have actions of the Lie algebra $\alg{u}^d_k$ on $\ell^1(\Z^d_k)$, defined as follows.

\begin{notation}\label{Lie-action-notation}
  Let $d\in\N_\ast$. Let $k = (k(1),\ldots,k(d)) \in \Nbar_\ast^d$. For all $n = (n_1,\ldots,n_d) \in \Z^d_k$, and for all $X = (X_1,\ldots,X_d) \in \alg{u}_k$, we set:
  \begin{equation*}
    \inner{X}{n}{k} = \sum_{\substack{j \in \{1,\ldots,d\} \\ k(j) = \infty}} X_j n_j \text{.}
  \end{equation*}
\end{notation}

\begin{notation}
  Let $k\in\Nbar_\ast^d$. For all $X \in \alg{u}^d_{k}$, let:
  \begin{equation*}
    \dom{\partial_{k,X}} = \left\{ f \in \ell^2\left(\Z^d_k\right) : \left(\inner{X}{n}{k} f(n) \right)_{n\in\Z^d_k} \in \ell^2\left(\Z^d_k\right) \right\}
  \end{equation*}
  and for all $f \in \dom{\partial_{k,X}}$, we set:
  \begin{equation}\label{Lie-derivation-eq}
    \partial_{k,X}(f) : m \in \ell^2(\Z^d_k) \mapsto  i \inner{X}{m}{k} f(m) \text.
  \end{equation}
\end{notation}

\begin{lemma}
  Let $k\in\Nbar_\ast^d$. The operator $i\partial_{k,X}$ is self-adjoint from $\dom{\partial_{k,X}}$ to $\ell^2(\Z^d_k)$, and for any $\xi \in \dom{\partial_{k,X}}$, we have
  \begin{equation*}
    \forall t \in \R \quad v_k^{\exp_{\U^d_k}(t X)}\xi = \exp\left(t \partial_{k,X}\right)\xi \text.
  \end{equation*}
\end{lemma}

\begin{proof}
  It is immediate that $i \partial_{k,X}$ is symmetric. Moreover, the operator $R$ defined, for all $\xi \in \ell^2(\Z^d_k)$, by
  \begin{equation*}
    R\xi : m \in \Z^d_k \mapsto \left(-\inner{X}{m}{k} + i\right)^{-1} \xi(m) \text,
  \end{equation*}
  is bounded, and maps $\ell^2(\Z^d_k)$ onto $\dom{\partial_{k,X}}$. Furthermore, for all $\xi \in \ell^2(\Z^d_k)$, an easy computation shows that $(i\partial_{k,X}+i) R \xi = \xi$. Thus, $i\partial_{k,X}+i$ is surjective. Similarly, $i\partial_{k,X}-i$ is surjective as well. Since $R \, (i\partial_{k,X} + i)\xi=\xi$ for all $\xi \in \dom{\partial_{k,X}}$ as well, and $R$ is bounded, we conclude that $i \partial_{k,X}$ is closed, and thus by \cite[Theorem VIII.3]{ReedSimon}, it is self-adjoint.

  It is an easy computation to check that $i \partial_{k,X}$ is indeed the generator of the unitary action $t \in \R \mapsto v_k^{\exp_{\U^d_k}(tX)}$.
\end{proof}

The dual action of $\U^d_k$ also defines an action of the Lie algebra $\alg{u}^d_k$ of $\U^d_k$ on $\qt{k}{\sigma}$; in fact the following observation holds.
\begin{definition}
  For all $d\in\N_\ast$, for all $p \in [1,\infty)$, and for all $k \in \Nbar_\ast^d$, an element $f \in \ell^p(\Z^d_k)$ is \emph{finitely supported} when the \emph{support} $\{ m \in \Z^d_k : f(m) \neq 0 \}$ of $f$ is finite.
\end{definition}
  
\begin{lemma}\label{partial-commutator-lemma}
  Let $d\in\N_\ast$. Let $(k,\sigma)\in\Xi^d$. For all $f \in \ell^1(\Z^d_k)$ such that
  \begin{equation*}
    \left(\inner{X}{m}{k} f(m)\right)_{m\in\Z^d_k} \in \ell^1(\Z^d_k) \text,
  \end{equation*}
  we have
  \begin{equation*}
    \left[ \partial_{k,X}, \pi_{k,\sigma}(f) \right] = \pi_{k,\sigma}(\partial_k^X f)
  \end{equation*}
  where
  \begin{equation*}
    \partial_k^X f : m \in \Z^d_k \mapsto i \inner{X}{m}{k} f(m) \text.
  \end{equation*}
  Moreover, for any finitely supported $f \in \ell^1(\Z^d_k)$, we have
  \begin{equation}\label{partial-eq}
    \partial_k^X f = \lim_{t\rightarrow 0} \frac{\alpha_{k,\sigma}^{\exp_{\U^d_k}(tX)}(f)-f}{t} \text.
  \end{equation}
\end{lemma}

\begin{proof}
  This is a direct computation.
\end{proof}

\begin{remark}
  Expression (\ref{partial-eq}) holds for certain elements besides finitely supported ones, but this is the degree of generality which we need.
\end{remark}

\begin{remark}
  Let $(k,\sigma) \in \Xi^d$. Using the Leibniz property, a simple computation shows that, if $f \in \ell^1(\Z^d)$ is finitely supported, and $\xi \in \dom{\partial_{k,X}}$, then
    \begin{align*}
      \partial_{k,X}(\pi_{k,\sigma}(f)\xi)
      &= \partial_{k,X}(f\conv{k,\sigma} \xi) \\
      &= \partial_{k,X}(f)\conv{k,\sigma}\xi + f\conv{k,\sigma}\xi \\
      &= \partial_k^X(\pi_{k,\sigma}(f)) \xi + \pi_{k,\sigma}(f)\xi \text.
    \end{align*}
\end{remark}

\bigskip

Now, let $(k,\sigma)\in\Xi^d$. If $k \in \N_\ast^d$, then the Lie algebra $\alg{u}^d_k$ of $\U^d_k$ is $\{ 0 \}$. Thus, our work so far does not provide us with any sort of differential calculus on fuzzy tori. It does however give us some suggestions on how to proceed to define such a calculus, based upon a discrete version of Expression (\ref{partial-eq}), which we will explain in our next subsection.

\subsection{Spectral Triples on Fuzzy and Quantum Tori}

We now define the family of spectral triples on the fuzzy and quantum tori which will be the focus of our present work.

\bigskip

Our construction here aims at finding a natural family of spectral triples on fuzzy tori which approximate spectral triples on quantum tori, under the following natural conditions, which we will use throughout the remainder of this paper.
\begin{hypothesis}\label{metric-cv-hyp}
  Fix $d\in\N_\ast$. For each $n\in\Nbar$, let us be given $(k_n,\sigma_n) \in \Xi^d$ such that $(k_n,\sigma_n)_{n\in\N}$ converges to $(k_\infty,\sigma)$ in $\Xi^d$, and $k_\infty = \infty^d$. 
  
  For each $n\in\Nbar$, we write $k_n = (k_n(1),\ldots,k_n(d))$ and $\A_n = \qt{k_n}{\sigma_n}$. To simplify our notations, we assume that, for each $j\in\{1,\ldots,d\}$, we also have
  \begin{equation*}
    \left(\forall n \in \N \quad k_n(j) \in \N\right) \text{ or }\left(\forall n\in \N \quad k_n(j) = \infty\right) \text.
  \end{equation*}
  
\end{hypothesis}

\begin{remark}
  If $(x_n)_{n\in\N}$ is a sequence in $\Nbar$ converging to $\infty$, then either it is equal to $\infty$ after some $N$, or we can find a subsequence $(x_{y(n)})_{n\in\N}$ such that, for all $n\in\N$, we have $x_n \in \N \iff\exists m \in \N \quad y(m) = n$; of course $(x_{y(n)})_{n\in\N}$ converges to $\infty$. Therefore, our simplification in Hypothesis (\ref{metric-cv-hyp}) can be done without loss of generality; we only include it to make our notation easier throughout this paper.
\end{remark}

Hypothesis (\ref{metric-cv-hyp}) can, in fact, be met for any possible choice of a normalized $2$-cocycle of $\Z^d$, up to replacing the given $2$-cocycle by a normalized cohomologous $2$-cocycle, as seen in Theorem (\ref{Cocycle-approx-thm}).

\bigskip

As explained in our introduction, we generally need to embed our fuzzy tori in larger fuzzy tori, in order to construct our spectral triples. What matters is the overall consistency of this scheme along a sequence of fuzzy tori approximating some quantum torus, so we summarize the needed conditions in the following hypothesis, which we will again use throughout this paper.

\begin{hypothesis}\label{innerification-hyp}
  Assume Hypothesis (\ref{metric-cv-hyp}). Let $d'\in\N_\ast$, with $d'\geq d$. Let $f$ be an order $2$ permutation of $\{1,\ldots,d'\}$ with no fixed point. For each $n \in \Nbar$, let $(k'_n,\sigma'_n) \in \Xi^{d'}$, such that, if we write $\B_n = \qtd{d'}{k'_n}{\sigma'_n}$ and
  \begin{equation*}
    U_{n,j} \coloneqq W_{k'_n,\sigma'_n}^{e_j}
  \end{equation*}
  for all $j\in\{1,\ldots,d'\}$, then
  \begin{enumerate}
  \item the following limit holds in $\Xi^{d'}$:
    \begin{equation*}
      \lim_{n\rightarrow\infty} (k'_n,\sigma'_n) = (\infty^{(d')}, \sigma'_{\infty}) \text.
    \end{equation*}
  \item $\{d+1,\ldots,d'\} \subseteq f(\{1,\ldots,d\})$,
  \item we write $k'_n = (k'_n(1),\ldots,k'_n(d')) \in \Nbar^{d'}$, and the following holds:
    \begin{equation*}
      \forall j \in \{1,\ldots,d\} \quad k'_n(f(j)) = k'_n(j) = k_n(j)\text,
    \end{equation*}
  \item the restriction of $\sigma'_n$ to the subgroup $\Z^d_{k_n}\times\{\underbracket[1pt]{(0,\ldots,0)}_{d'-d \text{ times}}\}$ of $\Z^{d'}_{k'_n}$ is cohomologous to $\sigma_n$ (as $2$-cocycles of $\Z^d_{k_n}$),
  \item for all $j \in \{1,\ldots,d\}$, the following commutation relation holds, for all $s \in \{1,\ldots,d'\}$:
    \begin{equation*}
      U_{n,f(j)} U_{n,s} = \begin{cases}
        \exp\left(\frac{2i\pi}{k_n(j)}\right) U_{n,j} U_{n,f(j)} \text{ if $k_n(j)<\infty$, $j<f(j)$, and $j=s$,}\\
        \exp\left(\frac{-2i\pi}{k_n(j)}\right) U_{n,j} U_{n,f(j)} \text{ if $k_n(j)<\infty$, $j>f(j)$, and $j=s$,}\\
        U_{n,s} U_{n,f(j)} \text{ otherwise.}
      \end{cases}
    \end{equation*}
  \end{enumerate}
    
  By construction, the C*-subalgebra $C^\ast(U_{n,1},\ldots,U_{n,d})$ of $\B_n$ is canonically *-isomorphic to $\A_n$, with the *-isomorphism sending $U_{n,j}$ to $W_{k_n,\sigma_n}^{e_j}$ for all $j\in\{1,\ldots,d\}$. We henceforth identify $\A_n$ with $C^\ast(U_{n,1},\ldots,U_{n,d})$ in $\B_n$, and without loss of generality, we also assume that $\sigma_n$ is equal (rather than merely cohomologous) to the restriction of $\sigma'_n$ to $\Z^d_{k_n}\times\{(0,\ldots,0)\}$.

 We write $G_n = \U^{d'}_{k'_n}$ and $\widehat{G_n} = \Z^{d'}_{k'_n}$ --- of course, $\widehat{G_n}$ is the Pontryagin dual of $G_n$. The dual action $\alpha_{k'_n,\sigma'_n}$ of $G_n$ on $\B_n$, defined in Corollary (\ref{dual-action-thm}), is simply denoted by $\alpha_n$. For each $z\in G_n$, we also simply denote by $v_n^z$ the unitary $v_{k'_n}^{z}$ acting on $\ell^2(\widehat{G_n})$, as defined in Notation (\ref{v-notation}).

  Moreover, to further ease our notation, if $z\in \U^{d}_{k_n}$, then we identify $z$ with the element $\left(z,\underbracket[1pt]{1,\ldots,1}_{d'-d \text{ times}}\right)$ in $G_n$ without further mention.

  The Hilbert space $\ell^2(\widehat{G_n})$ is denoted by $\Hilbert_n$. The *-representation $\pi_{k'_n,\sigma'_n}$ on $\Hilbert_n$ is simply denoted by $\pi_n$.
  
  For all $j \in \{1,\ldots,d'\}$, we also write
  \begin{equation*}
    \partial_{n,j} = \partial_{k'_n,e_j} \text{ and }\partial_n^j = \partial_{k'_n}^{e_j} \text.
  \end{equation*}
  
  Last, let $J_n = J_{k'_n}$ be the conjugate linear isometric involution mapping $\xi \in \Hilbert_n$ to $m \in \widehat{G_n} \mapsto \overline{\xi(-m)}$.
\end{hypothesis}

We make a few easy observations. First, note that since $\{d+1,\ldots,d'\}\subseteq f(\{1,\ldots,d\})$, and since $f$ is a permutation, we must have $d'\leq 2d$. Moreover, if $j\in\{d+1,\ldots,d'\}$, then $j=f(s)$ for $s\in \{1,\ldots,d\}$, so that $f(j) = s \in \{1,\ldots,d\}$. Therefore, if $j=f(s)$ and $j>d$ then $s\leq d$. Moreover, we then see that $\{1,\ldots,d'\}=\{1,\ldots,d\}\cup f(\{1,\ldots,d\})$, and thus, Condition (3) in Hypothesis (\ref{innerification-hyp}) completely describes $k'_n$. 

We also record, by definition, that $[\partial_{n,j},U_{n,s}] = \partial_n^j(U_{n,s}) = 0$ if $j,s\in\{1,\ldots,d'\}$ with $j\neq s$, and of course, $[\partial_{n,j}, U_{n,j}] = \partial_{n}^j U_{n,j} = i U_{n,j}$ for all $j\in\{1,\ldots,d'\}$.

\bigskip

The purpose of Hypothesis (\ref{innerification-hyp}) is to implement certain *-automorphisms via unitary conjugation, in a consistent manner along a sequence of fuzzy tori. Indeed, we observe the following.

\begin{remark}\label{innerification-rmk}
  If $j \in \{1,\ldots,d\}$, and if $k_n(j) = \infty$, then our computation shows that $U_{n,f(j)}$ is central in $\B_\infty$.
\end{remark}

\begin{lemma}\label{znj-notation}
   For all $n\in\Nbar$, we define, using the notation of Hypothesis (\ref{innerification-hyp}), and for $j\in\{1,\ldots,d'\}$:
 \begin{equation*}
   z_{n,j} \coloneqq
   \begin{cases}
     \left(\underbracket[1pt]{1,\ldots,1,\exp\left(\frac{2 i \pi}{k_n(j)}\right)}_{\text{indices }1 \text{ to }j},1,\ldots,1 \right) \in G_n \text{ if $k_n(j)<\infty$,}\\
     (\underbracket[1pt]{1,\ldots,1}_{d' \text{ times}}) \in G_n \text{ otherwise.}
   \end{cases}
\end{equation*}

If $j \in \{1,\ldots,d\}$ and $j < f(j)$, then
\begin{equation*}
  \forall a \in \B_n \quad U_{n,f(j)} \, a U_{n,f(j)}^\ast = \alpha_n^{z_{n,j}}(a) \text,
\end{equation*}
while, if $f(j) < j$, then
\begin{equation*}
  \forall a \in \B_n \quad U_{n,f(j)} \, a U_{n,f(j)}^\ast = \alpha_n^{\overline{z_{n,j}}}(a) \text.
\end{equation*}
\end{lemma}

\begin{proof}
  By construction, for all $j\in\{1,\ldots,d\}$, if $j < f(j)$, then we have
  \begin{equation}\label{action-eq}
    \forall s \in \{1,\ldots,d\} \quad U_{n,f(j)} \, U_{n,s} \, U_{n,f(j)}^\ast = \alpha_n^{z_{n,j}}(U_{n,s}) \text.
  \end{equation}
  Therefore, if $j\in\{1,\ldots,d\}$, and $j<f(j)$, then for all $a\in \A_n$, we conclude that
\begin{equation*}
  U_{n,f(j)} \, a \, U_{n,f(j)}^\ast = \alpha_n^{z_{n,j}}(a) \text.
\end{equation*}

A similar computation shows that $U_{n,f(j)} \, a \, U_{n,f(j)}^\ast = \alpha_n^{\overline{z_{n,j}}}(a)$ if $j\in\{1,\ldots,d\}$ and $f(j) < j$.
\end{proof}

\begin{corollary}\label{znj-cor}
  If $j \in \{1,\ldots,d\}$ and $j < f(j)$, then
  \begin{equation*}
    \forall \xi \in \Hilbert_n \quad U_{n,f(j)} \, \xi \cdot U_{n,f(j)}^\ast = v_n^{z_{n,j}}\xi \text,
  \end{equation*}
  while, if $f(j) < j$, then
  \begin{equation*}
    \forall \xi \in \Hilbert_n \quad U_{n,f(j)} \, \xi \cdot U_{n,f(j)}^\ast = v_n^{\overline{z_{n,j}}}\xi \text.
  \end{equation*}
\end{corollary}

\begin{proof}
  This follows from the observation that, if $j<f(j)$, then
  \begin{align*}
    U_{n,f(j)}\,\delta_m \cdot U_{n,f(j)}
    &= \delta_{f(j)} \conv{k'_n,\sigma'_n} \delta_m \conv{k'_n,\sigma'_n} \delta_{f(j)} \\
    &= z_{n,j}^m \delta_m = v_n^{z_{n,j}}\delta_m \text{ by Hyp. (\ref{innerification-hyp}),}
  \end{align*}
  and thus, by linearity and continuity, we get $U_{n,f(j)}\xi\cdot U_{n,f(j)} = v_n^{z_{n,j}}\xi$ for all $\xi \in \ell^2(\widehat{G_n})$. A similar argument shows the desired expression when $j < f(j)$.
\end{proof}

We note that it may be possible to choose $(k'_n,\sigma'_n) = (k_n,\sigma_n)$ for all $n\in\Nbar$ in Hypothesis (\ref{innerification-hyp}), though this will be in rather exceptional situations. On the other hand, we now prove that we can always meet Hypothesis (\ref{innerification-hyp}) given a sequence satisfying Hypothesis (\ref{metric-cv-hyp}), if we choose $d'=2d$. 

\begin{lemma}\label{generic-innerification-lemma}
  If we assume Hypothesis (\ref{metric-cv-hyp}), then there exist an order $2$ permutation of $\{1,\ldots,2d\}$, and for all $n\in\Nbar$, an element $(k'_n,\sigma'_n) \in \Xi^{2d}$, such that Hypothesis (\ref{innerification-hyp}) holds as well.
\end{lemma}

\begin{proof}
  We will work with $d' = 2d$. Let $f : \{1,\ldots,2d\}\rightarrow\{1,\ldots,2d\}$ be defined as the order $2$ permutation such that $f(j) = d+j$ for all $j\in\{1,\ldots,d\}$. The permutation $f$ has no fixed point by construction.

  For each $n\in\N$, with the convention that $\frac{1}{\infty}=0$, we let
  \begin{equation*}
    r_n = \begin{pNiceMatrix}
      \frac{1}{k_n(1)} & 0           & \Cdots & 0 \\
      0                & \Ddots      & \Ddots & \Vdots \\
      \Vdots           & \Ddots      &        &  0 \\
      0                & \Cdots      &  0     &   \frac{1}{k_n(d)}
    \end{pNiceMatrix}
  \end{equation*}
  and
  \begin{equation*}
    \omega_n = \begin{pNiceArray}{c|c}
      0_{d\times d} & r_n \\ \hline -r_n & 0_{d\times d}
    \end{pNiceArray}
  \end{equation*}

  Note that $\lim_{n\rightarrow\infty}\omega_n = 0$. Now, using Theorem (\ref{Cocycle-approx-thm}), we obtain a sequence $c_n$ of normalized $2$-cocycles of $\Z^{2d}_{k'_n}$ such that $(c_n)_{n\in\N}$ converges in $\mathcal{C}^{2d}_{\infty^{2d}}$ to a $2$-cocycle of $\Z^{2d}$, cohomologous to $1$, and, for all $n\in\N$, and for all $m,m'\in\Z^{2d}_{k'_n}$:
  \begin{equation*}
    W_{k'_n,c_n}^{m} W_{k'_n,c_n}^{m'} = \exp\left(2 i \pi \inner{\omega_n m}{m'}{} \right) W_{k'_n,c_n}^{m'} W_{k'_n,c_n}^{m} \text.
  \end{equation*}

  Let $q_n : (m_1,\ldots,m_{2d}) \in \Z^{2d}_{k'_n}\mapsto (m_1,\ldots,m_d) \in \Z^d_{k_n}$. We then define, for all $n\in\Nbar$:
  \begin{equation*}
    \sigma'_n:m,m'\in\Z^{2d}_{k'_n}\mapsto \sigma_n(q_n(m),q_n(m')) c_n(m,m') \text.
  \end{equation*}
  By construction, $\sigma'_n$ is a normalized $2$-cocycle of $\Z^{2d}_{k'_n}$, and $(k'_n,\sigma'_n)_{n\in\N}$ converges to $(\infty^{2d},\sigma'_\infty)$. It is now easy to check that $(k'_n,\sigma'_n)_{n\in\Nbar}$ and $f$ satisfies the conditions of Hypothesis (\ref{innerification-hyp}).
\end{proof}

\begin{remark}
  The construction of Lemma (\ref{generic-innerification-lemma}) is, in fact, a C*-crossed-product construction. Namely, assume Hypothesis (\ref{metric-cv-hyp}) and fix $n\in\Nbar$. We canonically embed $\qt{k_n}{\sigma_n}$ in the C*-crossed-product $\qt{k_n}{\sigma_n}\rtimes_\beta \Z^d_{k_n}$, where the action $\beta$ of $\Z^d_{k_n}$ is defined as follows. Let $z=(z_1,\ldots,z_d) \in \U^d_{k_n}$. For each $j\in\{1,\ldots,d\}$, let
  \begin{equation*}
    w_j = \begin{cases} z_{n,j} \text{ if $k_n(j)<\infty$, }\\
      1 \text{ otherwise.}
    \end{cases}
  \end{equation*}
  Then we set $\beta^z = \alpha_{k_n,\sigma_n}^{(w_1,\ldots,w_d)}$. The C*-algebra $\qt{k_n}{\sigma_n}\rtimes_\beta \Z^d_{k_n}$ is the fuzzy or quantum torus constructed in Lemma (\ref{generic-innerification-lemma}).
\end{remark}

Lemma (\ref{generic-innerification-lemma}) gives a general mean to implement Hypothesis (\ref{innerification-hyp}), and indeed, it is the basic scheme we have in mind for our construction. However, there are situations where we can work within the given fuzzy tori, without invoking the scheme of Lemma (\ref{generic-innerification-lemma}). This is, in fact, a very commonly used example in physics \cite{Kimura01,Schreivogl13,Barrett15,Connes97}.

\begin{example}
  We give an example where $d=d' (=2)$. Let $C_n$ and $S_n$ be the clock and shift matrices given by Expression (\ref{Clock-Shift-eq}). Note that Example (\ref{Clock-Shift-ex}), the C*-algebra $C^\ast(C_n,S_n)$ is *-isomorphic to a fuzzy torus, of the form $C^\ast(\Z^2_{(n,n)},\sigma_n)$, where $(\sigma_n)_{n\in\N}$ converges to $1$ in $\mathcal{C}^2_{\infty^2}$, and where a *-isomorphism sends $S_n$ to $U_{n,1}$ and $C_n$ to $U_{n,2}$. Since:
  \begin{equation}\label{clock-shift-comm-eq}
    U_{n,1}U_{n,2} = \exp\left(\frac{2 i \pi}{n}\right) U_{n,2} U_{n,1}\text.
  \end{equation}
  we let $f : \{1,2\} \mapsto \{1,2\}$ be given by  $f(1)=2$ and $f(2)=1$, $\sigma'_n = \sigma_n$ for all $n\in\N$, $\sigma_\infty = 1$, and note that we now meet the conditions of Hypothesis (\ref{innerification-hyp}).
\end{example}

To illustrate our Hypothesis (\ref{innerification-hyp}) further, note that there is another, obvious, approximation of the classical torus by fuzzy tori. Let $n\in\Nbar$. The fuzzy torus $C(\U_{(n,n)}^2)$ is Abelian, and thus all commutators are zero. This time, we do need to use some technique, such as Lemma (\ref{generic-innerification-lemma}), to meet Hypothesis (\ref{innerification-hyp}).

We can generally mix and match the technique of Lemma (\ref{generic-innerification-lemma}) and Example (\ref{Clock-Shift-eq}), if the quantum torus $\A_\infty$ has a non-trivial center, giving us many possible examples of the convergence found in this work. When the limit quantum torus is simple, then Lemma (\ref{generic-innerification-lemma}) should be used.

\bigskip

We now define the domains of the Dirac operators for our sequence of spectral triples, based upon Hypothesis (\ref{innerification-hyp}).

\begin{hypothesis}\label{domain-hyp}
  Assume Hypothesis (\ref{innerification-hyp}). We also use Notation (\ref{znj-notation}). We fix a faithful *-representation $c$ of the Clifford algebra $\alg{Cl}(\C^{d+d'})$ of $\C^{d+d'}$ on a \emph{finite dimensional} Hilbert space $\mathscr{C}$. Let $n\in\Nbar$.

  We write
  \begin{equation*}
    \mathscr{J}_n = \Hilbert_n \otimes \mathscr{C} \text.
  \end{equation*}
  We identify $\mathscr{J}_n$ with the Hilbert space $\ell^2(\widehat{G_n},\mathscr{C})$ of functions from $\widehat{G_n}$ to $\mathscr{C}$ such that
  \begin{equation*}
    \sum_{m\in\widehat{G_n}} \norm{\xi(m)}{\mathscr{C}}^2 < \infty \text.
  \end{equation*}

  Writing $\unit_{\mathscr{C}}$ for the identity on $\mathscr{C}$, we then define for all bounded linear operator $a$ on $\Hilbert_n$ (including $a\in \B_n$):
  \begin{equation*}
     a^\circ = a \otimes \unit_{\mathscr{C}}\text.
  \end{equation*}

  For all $n\in\Nbar$, and for all $m=(m_1,\ldots,m_{d'}),m'=(m'_1,\ldots,m'_{d'}) \in \widehat{G_n}$, we set
  \begin{equation*}
    \inner{m}{m'}{n} \coloneqq \inner{m}{m'}{k'_n} = \sum_{\substack{j \in \{1,\ldots,d'\}\\ k'_n(j)=\infty}} m_j m'_j  \text.
  \end{equation*}

  We define the subspace
  \begin{equation*}
    \dom{\Dirac_n} = \left\{ \xi \in \mathscr{J}_n : \left(\left(\sqrt{1 + \inner{m}{m}{n}}\right) \xi(m)\right)_{m\in\widehat{G_n}} \in \mathscr{J}_n \right\}\text.
  \end{equation*}
\end{hypothesis}

Hypothesis (\ref{domain-hyp}) defines the domain on which the Dirac operators of our spectral triples will be defined. We establish a few properties which will be helpful toward this goal.

\begin{lemma}\label{dense-domain-lemma}\label{partial-n-j-lemma}
  Assume Hypothesis (\ref{domain-hyp}). If $n\in\Nbar$, then, for all linear map $t:\mathscr{C}\rightarrow\mathscr{C}$, then
  \begin{equation*}
    \forall \xi \in \dom{\Dirac_n} \quad (\partial_{n,j}\otimes t)\xi \in \mathscr{J}_n\text.
  \end{equation*}
  Moreover,  $\dom{\Dirac_n}$ contains all finitely supported functions from $\widehat{G}_n$ to $\mathscr{C}$. Therefore, $\dom{\Dirac_n}$ is dense in $\mathscr{J}_n$.
\end{lemma}

\begin{proof}
  The proof is an immediate computation.
\end{proof}

\begin{lemma}\label{Udom-dom-lemma}
  Assume Hypothesis (\ref{domain-hyp}). For all $n\in\Nbar$, and for all $y \in \widehat{G_n}$, 
  \begin{equation*}
    \left(W_{k'_n,\sigma'_n}^y\right)^\circ \dom{\Dirac_n} \subseteq \dom{\Dirac_n}
  \end{equation*}
  and
  \begin{equation*}
    (J_n W_{k'_n,\sigma'_n}^{-y} J_n)^\circ \; \dom{\Dirac_n} \subseteq \dom{\Dirac_n} \text.
  \end{equation*}
\end{lemma}

\begin{proof}
  Let $n\in\Nbar$. Fix $y \in \widehat{G_n}$ and write $W_y$ for the unitary $W_{k'_n,\sigma'_n}^y$. Let $\xi \in \dom{\Dirac_n}$. We compute
  \begin{align*}
    \sum_{m\in\widehat{G}} &(1+\inner{m}{m}{n})\norm{(W_y^\circ \xi)(m)}{\mathscr{C}}^2 \\
    &= \sum_{m\in\widehat{G}} (1+\inner{m}{m}{n})\norm{\sigma'_n(y,m-y)\xi(m-y)}{\mathscr{C}}^2 \\
    &= \sum_{m\in\widehat{G_n}} \frac{(1+\inner{m}{m}{n})}{(1+\inner{m-y}{m-y}{n})}(1+\inner{m-y}{m-y}{n}) \norm{\xi(m-y)}{\mathscr{C}}^2 \\
    &= \sum_{m\in\widehat{G_n}} \frac{(1+\inner{m+y}{m+y}{n})}{(1+\inner{m}{m}{n})}(1+\inner{m}{m}{n}) \norm{\xi(m)}{\mathscr{C}}^2 \text.
  \end{align*}
  Since $\xi \in \dom{\Dirac_n}$, we know that
  \begin{equation*}
    \sum_{m\in\widehat{G_n}} (1+\inner{m}{m}{n})\norm{\xi(m)}{\mathscr{C}}^2 < \infty\text.
  \end{equation*}

  On the other hand, there exists $N\in \N$ such that, if $\inner{m}{m}{n}\geq N$, then $\frac{(1+\inner{m+y}{m+y}{n})}{(1+\inner{m}{m}{n})}\leq 2$. As $\{ m \in \widehat{G_n} : \inner{m}{m}{n} < N \}$ is finite, we conclude that
  \begin{equation*}
    \sum_{m\in\widehat{G}} (1+\inner{m}{m}{n})\norm{(W_y^\circ \xi)(m)}{\mathscr{C}}^2 < \infty\text.
  \end{equation*}
  Thus $W_y^\circ \dom{\Dirac_n}\subseteq \dom{\Dirac_n}$, as claimed.
  
  \medskip
  
  Let $y \in \widehat{G_n}$. A quick computation shows that, for all $\xi \in \mathscr{J}_n$, we have:
  \begin{equation}\label{right-action-eq}
    (J_n W_{-y} J_n)^\circ \xi : m \in \widehat{G_n} \mapsto \sigma'_n(m-y,y)\xi(m-y) \text.
  \end{equation}
  A similar computation as above then allows us to conclude that our lemma holds.
\end{proof}

\medskip

We now are ready to define our spectral triples for fuzzy tori and quantum tori, by specifying the Dirac operators for these triples over the domains defined in Hypothesis (\ref{domain-hyp}).

\begin{notation}
  If $\A$ is a C*-algebra and $a\in\A$, then $\Re(a) \coloneqq \frac{a+a^\ast}{2} \in \sa{\A}$ and $\Im(a) \coloneqq \frac{a-a^\ast}{2i} \in \sa{\A}$.
\end{notation}

\begin{hypothesis}\label{Dirac-hyp}
  Assume Hypothesis (\ref{domain-hyp}). We also will use Notation (\ref{znj-notation}). For all $n\in\Nbar$, we then define the following objects.
  
  The identity of $\Hilbert_n$ is denoted by $\unit_n$. For all $j\in\{1,\ldots,d\}$, we now define the self-adjoint bounded operators
  \begin{equation*}
    X_{n,j} = \Re(U_{n,f(j)}) = \frac{1}{2}\left( U_{n,f(j)} + U_{n,f(j)}^\ast \right)
  \end{equation*}
  and
  \begin{equation*}
    Y_{n,j} = \mathrm{fsgn}(j)\Im(U_{n,f(j)}) = \frac{\mathrm{fsgn}(j)}{2i}\left( U_{n,f(j)} - U_{n,f(j)}^\ast \right) \text,
  \end{equation*}
  where
  \begin{equation*}
    \mathrm{fsgn}(j) = \begin{cases}
      1  \text{ if $f(j) > j$, }\\
      -1 \text{ if $f(j) < j$. }
    \end{cases}
  \end{equation*}
  
  We now define the following operators. Let $j\in\{1,\ldots,d\}$.
  \begin{enumerate}
  \item if $j \in \{1,\ldots,d\}$ and $k_n(j) = \infty$, then
    \begin{equation*}
      \Gamma_{n,j} = X_{n,j} \partial_{n,j} \text{ and }\Gamma_{n,d+j} = Y_{n,j} \partial_{n,j} \text;
    \end{equation*}
  \item if $j\in \{1,\ldots,d\}$ and $k_n(j) < \infty$, then we define $\Gamma_{n,j}$ and $\Gamma_{n,j+d}$ by setting, for all $\xi\in\Hilbert_n$:
    \begin{equation*}
      \Gamma_{n,j}\xi = \frac{-k_n(j)}{2 i \pi} \left[ Y_{n,j}, \xi \right] \text{ and }\Gamma_{n,d+j}\xi = \frac{k_n(j)}{2 i \pi} \left[ X_{n,j}, \xi \right] \text;
    \end{equation*}
  \item if $j\in\{d+1,\ldots,d'\}$ and $k'_n(j) = \infty$, then
    \begin{equation*}
      \Gamma_{n, d+j} = \partial_{n,j} \text;
    \end{equation*}
  \item if $j\in\{d+1,\ldots,d'\}$ and $k'_n(j) < \infty$, then
    \begin{equation*}
      \Gamma_{n,d+j} = \frac{k'_n(j)}{2 i \pi} \left(\unit_n - \Im\left(v_n^{z_{n,j}}\right)\right) \text.
    \end{equation*}
  \end{enumerate}
  
  We now define the following operator from $\dom{\Dirac_n}$ to $\mathscr{J}_n$:
  \begin{equation*}
    \Dirac_n = \sum_{j=1}^{d+d'} \Gamma_{n,j} \otimes c(\gamma_j) \text,
  \end{equation*}
  which is well-defined using Lemma (\ref{partial-n-j-lemma}).
\end{hypothesis}

Thus, our candidate for an interesting, convergent sequence of metric spectral triples on fuzzy and quantum tori is given by the sequence of triples $(\A_n,\mathscr{J}_n,\Dirac_n)$, given by Hypothesis (\ref{Dirac-hyp}).  Our first task is to prove that, indeed, these triples are spectral triples. The main concern is to prove that $\Dirac_n$ is a self-adjoint operator with a compact resolvent, for all $n\in\Nbar$. While this is obvious when $k_n\in\N_\ast^d$, the situation is more involved when $k_n$ may have infinite component.

\subsection{Self-adjointness of the proposed Dirac operators}

In order to prove that the operator $\Dirac_n$ is self-adjoint, we will use the following, standard, characterization \cite[Sec. VIII.2]{ReedSimon}: a densely defined, symmetric operator $T$ on a Hilbert space $\mathscr{T}$ is essentially self-adjoint if, and only if the ranges of both $T+i$ and $T-i$ are dense in $\mathscr{T}$. Now, suppose that, in addition, the range of $T + i$ is $\mathscr{T}$. Let $\overline{T}$ be the closure of $T$. Let $\eta\in\dom{\overline{T}}$. Since $T+i$ is surjective, there exists $\xi \in \dom{T}$ such that $(T + i)\xi = (\overline{T}+i)\eta$. Of course, since $\overline{T}$ extends $T$, we also have $(\overline{T} + i)\xi = (T+i)\xi$. Therefore, $(\overline{T}+i)\xi = (\overline{T}+i)\eta$, and since $\overline{T}$ is self-adjoint, $\overline{T}+i$ is injective, so $\xi = \eta$, i.e. $\eta \in \dom{T}$. Thus $\dom{\overline{T}} \subseteq \dom{T}$. Since $\dom{T} \subseteq \dom{\overline{T}}$, we conclude $\dom{T} = \dom{\overline{T}}$. Thus $T = \overline{T}$, and thus $T$ is self-adjoint.

Since any nonzero, real multiple of a self-adjoint operator is again self-adjoint, we, in fact, will use the following equivalence.
\begin{lemma}\label{self-adjoint-lemma}
 For a densely-defined symmetric operator $T$ on a Hilbert space $\mathscr{T}$, the following assertions are equivalent:
\begin{itemize}
\item $T$ is self-adjoint,
\item there exists $z \in i\R\setminus\{0\}$ such that the image by $T \pm z$ of $\dom{T}$ is $\mathscr{T}$.
\end{itemize}
\end{lemma}

This strategy is implemented in the following lemmas. First, we prove that $\Dirac_n$ is symmetric.

\begin{lemma}\label{symmetric-lemma}
  Assume Hypothesis (\ref{Dirac-hyp}). For all $n\in\Nbar$, the operator $\Dirac_n$ is a densely-defined symmetric operator on $\dom{\Dirac_n}$.
\end{lemma}

\begin{proof}
  By Lemma (\ref{dense-domain-lemma}), the operator $\Dirac_n$ is well-defined on $\dom{\Dirac_n}$, which is indeed dense in $\mathscr{J}_n$.
  \medskip
  
  By Lemma (\ref{partial-commutator-lemma}), for all $j\in\{1,\ldots,d'\}$ with $k'_n(j)=\infty$, we note that
  \begin{equation*}
    \forall \xi,\eta\in\dom{\partial_{n,j}} \quad \inner{\eta}{\partial_{n,j}\xi}{\Hilbert_n} = \inner{-\partial_{n,j}\eta}{\xi}{\Hilbert_n}\text.
  \end{equation*}
  On the other hand, if $j\in\{1,\ldots,d\}$, then $X_{n,j}^\ast = X_{n,j}$, and, if $k_n(j)=\infty$,  the operator $X_{n,j}$ commutes with $\partial_{n,j}$ by Hypothesis (\ref{innerification-hyp}), since $f$ has no fixed point. Thus
  \begin{equation*}
    \inner{\eta}{X_{n,j}\partial_{n,j}\xi}{\Hilbert_n} = \inner{-\partial_{n,j}X_{n,j}\eta}{\xi}{\Hilbert_n} = \inner{-X_{n,j}\partial_{n,j}\eta}{\xi}{\Hilbert_n} \text.
  \end{equation*}
  The same result holds for $Y_{n,j}$ in place of $X_{n,j}$, for all $j\in\{1,\ldots,d\}$ with $k_n(j)=\infty$.
  
  Let now $j\in\{1,\ldots,d\}$ with $k_n(j)<\infty$. We compute (noting, again, that $X_{n,j}^\ast=X_{n,j}$):
  \begin{align*} 
    \bigg(\frac{k_n(j)}{2 i \pi} & \left[X_{n,j},\cdot\right]\bigg)^\ast \\
    &= -\frac{k_n(j)}{2 i \pi} \left(\pi(X_{n,j})^\ast - (J_n\pi(X_{n,j}^\ast)J_n)^\ast\right) \\
    &= -\frac{k_n(j)}{2 i \pi} (\pi_{n,j}(X_{n,j}) - J_n \pi(X_{n,j}) J_n) = -\frac{k_n(j)}{2 i \pi} [X_{n,j},\cdot] \text.
  \end{align*}
  The same computation holds for $Y_{n,j}$ in place of $X_{n,j}$.

  Similarly, $\left(\frac{k'_n(j)}{2 i \pi} (\unit_n - \Im(v_n^{z_{n,j}}))\right)^\ast = - \frac{k'_n(j)}{2 i \pi} (\unit_n - \Im(v_n^{z_{n,j}}))$ for all $j\in\{d+1,\ldots,d'\}$ with $k'_n(j)<\infty$.

  Therefore, for all $j \in \{1,\ldots,d + d'\}$, we have seen that $\Gamma_{n,j}^\ast = -\Gamma_{n,j}$ on $\dom{\Gamma_{n,j}}$, and thus, for all $j\in\{1,\ldots,d + d'\}$, since $c(\gamma_j)^\ast = -c(\gamma_j)$, we conclude that
  \begin{equation*}
    \forall \eta,\xi \in \dom{\Dirac_n} \quad \inner{\eta}{(\Gamma_{n,j}\otimes c(\gamma_j))\xi}{\mathscr{J}_n} = \inner{(\Gamma_{n,j}\otimes c(\gamma_j))\eta}{\xi}{\mathscr{J}_n}\text.
  \end{equation*}
  
  Therefore, if $\xi,\eta\in\dom{\Dirac_n}$, then we compute
  \begin{align*}
    \inner{\eta}{\Dirac_n\xi}{\mathscr{J}_n}
    &= \sum_{j=1}^{d+d'} \inner{\eta}{\Gamma_{n,j}\otimes c(\gamma_j)\xi}{\mathscr{J}_n} \\
    &= \sum_{j=1}^{d+d'} \inner{(\Gamma_{n,j}\otimes c(\gamma_j))\eta}{\xi}{\mathscr{J}_n}    = \inner{\Dirac_n\eta}{\xi}{\mathscr{J}_n} \text.
  \end{align*}

  Thus, indeed, $\Dirac_n$ is a symmetric with the dense domain $\dom{\Dirac_n}$.
\end{proof}

We now introduce a positive compact operator, which will be used to compute the inverse of $M^2 + \Dirac_n^2$, for a constant $M$ whose choice will be justified in a few lemmas.

\begin{lemma}\label{Kn-lemma}
  We use the notation of Hypothesis (\ref{Dirac-hyp}). Let $M = 10 d$. Let $n\in\Nbar$. We define the multiplication operator $K_n$ on $\Hilbert_n$ by setting, for all $\xi \in \Hilbert_n$,
  \begin{equation*}
    K_n\xi : m \in \widehat{G_n} \longmapsto \frac{1}{M^2 + \inner{m}{m}{n}} \xi(m) \text.
  \end{equation*}
 
  The operator $K_n^\circ$ is a positive, compact operator on $\mathscr{J}_n$, with $\range{\sqrt{K_n}} \subseteq\bigcap_{\substack{j\in\{1,\ldots,d'\} \\ k_n(j)=\infty}}\dom{\partial_{n,j}}$. Moreover
  \begin{equation*}
    \opnorm{\sqrt{K_n^\circ}}{}{\mathscr{J}_n} = \frac{1}{M} \text.
  \end{equation*}

  \medskip
  
  Furthermore, for all linear map $t:\mathscr{C} \rightarrow\mathscr{C}$, the following holds:
  \begin{equation*}
    \sqrt{K_n^\circ} \mathscr{J}_n \subseteq\dom{\Dirac_n} \text{ and } (\partial_{n,j}\otimes t) K_n^\circ \mathscr{J}_n \subseteq \dom{\Dirac_n} \text,
  \end{equation*}
  so that
  \begin{equation*}
    \Dirac_n K_n^\circ \mathscr{J}_n \subseteq \dom{\Dirac_n}\text.
  \end{equation*}

  Therefore,
  \begin{equation*}
    K_n^\circ \mathscr{J}_n \subseteq \dom{\Dirac_n^2} \text{, and } K_n^\circ \mathscr{J}_n \subseteq \dom{\partial_{n,j}^2}\otimes_{\mathrm{alg}}\mathscr{C}\text,
  \end{equation*}
  for all $j\in\{1,\ldots,d'\}$ with $k'_n(j) = \infty$.
  
  Last, for all $j\in\{1,\ldots,d'\}$ with $k'_n(j) = \infty$, and for all linear map $t : \mathscr{C} \rightarrow \mathscr{C}$, the operator $(\partial_{n,j}\otimes t) \sqrt{K_n^\circ}$ is a bounded operator on $\mathscr{J}_n$, with norm at most $\frac{\sqrt{2}}{2}\opnorm{t}{}{\mathscr{C}}$.
\end{lemma}

\begin{proof}
  An obvious computation shows that $K_n$ is self-adjoint, and that $(\delta_m)_{m \in \Z^d_k}$ is a Hilbert basis of eigenvectors of $K_n$; the spectrum of $K_n$ is $\left\{\frac{1}{M^2+ \inner{m}{m}{n}} : m\in\Z^d_k\right\}$. Thus $K_n$ is positive. The norm of $K_n$ is, in particular, its spectral radius $\frac{1}{M^2}$, and, similarly, the norm of $\sqrt{K_n}$ is $\frac{1}{M}$. Last, it is easy to see that $K_n$ is compact as well.

  Since $\mathscr{C}$ is finite dimensional, $K_n^\circ$ is also compact. It is straightforward that $K_n^\circ$ is positive, with $\sqrt{K_n}^\circ=\sqrt{K_n^\circ}$, and $\opnorm{\sqrt{K_n^\circ}}{}{\mathscr{J}_n} = \frac{1}{M}$.
  
  \medskip

  Fix $n\in\Nbar$. Let $\xi \in \mathscr{J}_n$. We compute:
  \begin{align*}
    \sum_{m\in \widehat{G_n}} &(1+\inner{m}{m}{n}) \norm{(\sqrt{K_n^\circ}\xi)(m)}{\mathscr{C}}^2 \\
    &= \sum_{m\in\widehat{G_n}} (1+\inner{m}{m}{n}) \frac{1}{\left(M^2 + \inner{m}{m}{n}\right)}\norm{\xi(m)}{\mathscr{C}}^2 \\
    &\leq \sum_{m\in \widehat{G_n}} \norm{\xi(m)}{\mathscr{C}}^2 = \norm{\xi}{\mathscr{J}_n}^2 < \infty \text.
  \end{align*}
  Thus, $\sqrt{K_n^\circ}\xi \in \dom{\Dirac_n}$. Therefore, $\sqrt{K_n^\circ}\mathscr{J}_n \subseteq\dom{\Dirac_n}$.
  
  \medskip

  Now, let $\xi \in \mathscr{J}_n$. Let $j \in \{1,\ldots,d'\}$ such that $k'_n(j) = \infty$, and let $t : \mathscr{C} \rightarrow \mathscr{C}$ be linear (hence bounded, since $\mathscr{C}$ is finite dimensional by Hypothesis (\ref{domain-hyp})). An easy computation shows that if $\xi \in \dom{\Dirac_n}$, then
  \begin{equation*}
    (\partial_{n,j}\otimes t) \sqrt{K_n^\circ} \xi = \sqrt{K_n^\circ} (\partial_{n,j}\otimes t)\xi\text.
  \end{equation*}

  Since $\sqrt{K_n^\circ}\mathscr{J}_n\subseteq \dom{\Dirac_n}\subseteq\dom{\partial_{n,j}\otimes t}$, we conclude that, for all $\xi \in \mathscr{J}_n$, the following holds for all linear map $t : \mathscr{C}\rightarrow\mathscr{C}$:
  \begin{align*}
    (\partial_{n,j}\otimes t) K_n^\circ \xi
    &= (\partial_{n,j}\otimes t) \sqrt{K_n^\circ} (\sqrt{K_n^\circ}\xi) \\
    &= \sqrt{K_n^\circ}(\partial_{n,j}\otimes t) \, \sqrt{K_n^\circ}\xi \text{ since $\sqrt{K_n^\circ}\xi\in\dom{\Dirac_n}$,}\\
    &\subseteq \sqrt{K_n^\circ} \mathscr{J}_n \subseteq \dom{\Dirac_n} \text,
  \end{align*}
  so $(\partial_{n,j}\otimes t) K_n^\circ \mathscr{J}_n\subseteq\dom{\Dirac_n}$. Of course, this could also be checked with a simple, direct computation.

  \medskip
  
  Since $X_{n,j}^\circ$, $Y_{n,j}^\circ$, $[X_{n,j},\cdot]^\circ$ and $[Y_{n,j},\cdot]^\circ$ all map $\dom{\Dirac_n}$ to itself as well by Lemma (\ref{Udom-dom-lemma}), and since $\dom{\Dirac_n}$ is an algebraic subspace, we conclude that
  \begin{equation*}
    \Dirac_n K_n^\circ \mathscr{J}_n \subseteq \dom{\Dirac_n} \text.
  \end{equation*}

  Moreover, we see that if $t:\mathscr{C}\rightarrow\mathscr{C}$ is linear, and if $\xi \in \mathscr{J}_n$, then
  \begin{align*}
    \norm{(\partial_{n,j}\otimes t)\sqrt{K_n^\circ} \xi}{\mathscr{J}_n}^2
    &= \sum_{m=(m_1,\ldots,m_{d'})\in\widehat{G_n}} \frac{m_j^2}{M^2 + \inner{m}{m}{n}}\norm{t\xi(m)}{\mathscr{C}}^2 \\
    &\leq \frac{\opnorm{t}{}{\mathscr{C}}^2}{2} \sum_{m\in\widehat{G_n}} \norm{\xi(m)}{\mathscr{C}}^2 \leq \frac{1}{2}\opnorm{t}{}{\mathscr{C}}^2 \norm{\xi}{\mathscr{C}}^2
  \end{align*}
  and thus, $(\partial_{n,j}\otimes t)\sqrt{K_n^\circ}$ is bounded on $\mathscr{J}_n$, with norm at most $\frac{\sqrt{2}}{2}\opnorm{t}{}{\mathscr{C}}$. 

  Therefore, $\Dirac_n \sqrt{K_n^\circ}$, as a linear combination of bounded operators on $\mathscr{J}_n$, is bounded  on $\mathscr{J}_n$ as well. 
\end{proof}

\begin{remark}
  An alternative proof that, for all $n\in\Nbar$ and $j\in\{1,\ldots,d'\}$ with $k_n(j)=\infty$, the operator $(\partial_{n,j}\otimes t) \sqrt{K_n^\circ}$ is bounded would be to note that $\partial_{n,j}$ is a closed operator (as a skew-adjoint operator), and thus $(\partial_{n,j}\otimes t)\sqrt{K_n^\circ}$ is a closed operator as well, and thus by the closed graph theorem, since $(\partial_{n,j}\otimes t)\sqrt{K_n^\circ}$ is defined on the entire Hilbert space $\mathscr{J}_n$, it is continuous. It is helpful, however, to have an upper bound on the norm of $(\partial_{n,j}\otimes t)\sqrt{K_n^\circ}$.
\end{remark}

We now compute commutation relations between some of the unbounded operators in Hypothesis (\ref{Dirac-hyp}).

\begin{lemma}\label{Gamma-commutation-lemma}
  Assume Hypothesis (\ref{Dirac-hyp}). We use the notation of Lemma (\ref{Kn-lemma}).

  For all $j,s \in \{1,\ldots,d\}$, for all $p,q \in \{0,d\}$, if $k_n(j) = \infty$, then
  \begin{equation*}
    \opnorm{\left(\left[\Gamma_{n,j+p},\Gamma_{n,s+q}\right]\otimes t\right)\sqrt{K_n^\circ}}{}{\mathscr{J}_n} \leq
    \begin{cases}
      \sqrt{2} \opnorm{t}{}{\mathscr{C}} \text{ if $s = f(j)$,} \\
      0 \text{ otherwise.}
    \end{cases}
  \end{equation*}

  For all $j \in \{1,\ldots,d\}$, $p \in \{0,d\}$, and for all $s \in \{d+1,\ldots,d'\}$, if $k_n(j) = \infty$ or $k_n(s) = \infty$, then
  \begin{equation*}
    \opnorm{\left(\left[\Gamma_{n,j+p},\Gamma_{n,s+d}\right]\otimes t\right)\sqrt{K_n^\circ}}{}{\mathscr{J}_n} \leq
    \begin{cases}
      \frac{\sqrt{2}}{2} \opnorm{t}{}{\mathscr{C}} \text{ if $s = f(j)$,} \\
      0 \text{ otherwise.}
    \end{cases}
  \end{equation*}

  Last, for all $j,s \in \{2d+1,\ldots,d+d'\}$, we compute (over the range of $\sqrt{K_n^\circ}$):
  \begin{equation*}
    \left[\Gamma_{n,j},\Gamma_{n,s}\right] = 0 \text.
  \end{equation*}
  
\end{lemma}

\begin{proof}
  In this proof, we will make repeated use of the fact that, by Hypothesis (\ref{Dirac-hyp}), for all $n\in\Nbar$ and $j\in \{1,\ldots,d\}$, and for all $j\in\{1,\ldots,d\}$, we have $\opnorm{X_{n,j}}{}{\Hilbert_n} = 1$ and $\opnorm{Y_{n,j}}{}{\Hilbert_n} = 1$, and for all $s\in\{1,\ldots,d\}$, if $\infty\in\{k_n(j),k_n(s)\}$, the operators $X_{n,j}$,$X_{n,s}$,$Y_{n,j}$, and $Y_{n,s}$ commute (by Remark (\ref{innerification-rmk})).

  Let $j,s \in \{1,\ldots,d\}$. Assume first that $k_n(j) = k_n(s) = \infty$. We then compute (over the range of $\sqrt{K_n}$):
  \begin{align*}
    X_{n,j} \partial_{n,j} (X_{n,s} \partial_{n,s})
    &= X_{n,j} \left( \partial_{n}^j(X_{n,s}) \partial_{n,s} + X_{n,s} \partial_{n,j}\partial_{n,s}  \right) \\
    &= \begin{cases}
      X_{n,j} \left(\partial_{n}^j X_{n,f(j)} \right) \partial_{n,f(j)} + X_{n,j}X_{n,f(j)}\partial_{n,j} \partial_{n,f(j)} \text{ if $s = f(j)$,}\\
      X_{n,j} X_{n,s} \partial_{n,j}\partial_{n,s} \text{ otherwise.}
    \end{cases}
  \end{align*}
  By Hypothesis (\ref{Dirac-hyp}), $\partial_j$ and $\partial_s$ commute; moreover, $X_{n,j}$ and $X_{n,s}$ commute as well. Therefore, if $s=f(j)$, then:
  \begin{equation*}
    \left[\Gamma_{n,j},\Gamma_{n,s}\right]
    = X_{n,j}\partial_{n}^j(X_{n,f(j)})\partial_{n,f(j)} - X_{n,f(j)}\partial_{n}^{f(j)}(X_{n,j}) \partial_{n,j}
  \end{equation*}
  from which it follows, using Lemma (\ref{Kn-lemma}):
  \begin{align*}
    \opnorm{\left(\left[\Gamma_{n,j},\Gamma_{n,s}\right]\otimes t\right)\sqrt{K_n^\circ}}{}{\mathscr{J}_n}
    &\leq \opnorm{\left( X_{n,j}\partial_{n}^{j}(X_{n,f(j)})\partial_{n,f(j)} \otimes t\right)\sqrt{K_n^\circ}}{}{\mathscr{J}_n}\\
    &\quad + \opnorm{\left(X_{n,f(j)}\partial_{n}^{f(j)}(X_{n,j})\partial_{n,j} \otimes t\right)\sqrt{K_n^\circ}}{}{\mathscr{J}_n} \\
    &\leq \opnorm{X_{n,j} Y_{n,f(j)}}{}{\Hilbert_n} \opnorm{(\partial_{n,f(j)}\otimes t) \sqrt{K_n^\circ}}{}{\mathscr{J}_n} \\
    &\quad + \opnorm{X_{n,f(j)} Y_{n,j}}{}{\Hilbert_n} \opnorm{(\partial_{n,j}\otimes t) \sqrt{K_n^\circ}}{}{\mathscr{J}_n} \\
    &\leq \frac{\sqrt{2}}{2}\opnorm{t}{}{\mathscr{C}} + \frac{\sqrt{2}}{2}\opnorm{t}{}{\mathscr{C}} = \sqrt{2} \opnorm{t}{}{\mathscr{C}} \text.
  \end{align*}
  On the other hand, if $f(j)\neq s$, then
  \begin{equation*}
    \opnorm{\left(\left[\Gamma_{n,j},\Gamma_{n,s}\right]\otimes t\right)\sqrt{K_n^\circ}}{}{\mathscr{J}_n} = 0 \text.
  \end{equation*}
  
  Now, a similar computation may be applied to the cases where $j,s \in \{1,\ldots,d\}$, with $k_n(j) = k_n(s) = \infty$, with $p,q \in \{0,d\}$, to establish the first statement of our lemma.

  Let now $j\in\{1,\ldots,d\}$ such that $k_n(j) = \infty$, and let $s\in\{1,\ldots,d\}$ with $k_n(s) < \infty$. We compute, for all $\xi \in \mathrm{range}\sqrt{K_n}$ (noting that $s\neq f(j)$, so $\partial_n^j (Y_{n,s}) = 0$):
  \begin{align*}
    \frac{2 i \pi}{k_n(s)} \left[\Gamma_{n,j},\Gamma_{n,s}\right]\xi
    &= X_{n,j}\partial_{n,j}\left[Y_{n,s},\xi \right] - \left[Y_{n,s},X_{n,j}\partial_{n,j}\xi \right] \\
    &= X_{n,j}(\partial_{n}^j Y_{n,s}) \xi + X_{n,j} Y_{n,s} \partial_{n,j}\xi - X_{n,j} \partial_{n,j}(\xi) Y_{n,s} \\
    &\quad - X_{n,j}\xi\partial_{n}^j Y_{n,s} - \left[Y_{n,s},X_{n,j}\partial_{n,j}\xi\right] \\
    &= Y_{n,s} X_{n,j} \partial_{n,j}\xi - X_{n,j} \left(\partial_{n,j}\xi\right) Y_{n,s} - \left[Y_{n,s},X_{n,j}\partial_{n,j}\xi\right] \\
    &= 0 \text.
  \end{align*}
  
  The same argument applies when $j,s \in \{1,\ldots,d\}$, $p,q\in\{0,d\}$,  with $k_n(j)=\infty$ and $k_n(s)<\infty$, to conclude:
  \begin{equation*}
    \opnorm{\left(\left[\Gamma_{n,j+p},\Gamma_{n,s+q}\right]\otimes t\right)\sqrt{K_n^\circ}}{}{\mathscr{J}_n} = 0 \text.
  \end{equation*}

  Now, let $j \in \{1,\ldots,d\}$, and $s \in \{d+1,\ldots,d'\}$, with $k_n(s) = k_n(j) = \infty$. Again, via a similar computation, over the range of $\sqrt{K_n}$:
  \begin{equation*}
    \left[\Gamma_{n,j},\Gamma_{n,s+d}\right] =
    \begin{cases}
      \partial_{n}^{f(j)}(X_{n,j}) \partial_{n,j+d} \text{ if $f(j)=s$,} \\
      0 \text{ otherwise.}
    \end{cases}
  \end{equation*}
  We thus get
  \begin{equation*}
    \opnorm{\left(\left[\Gamma_{n,j},\Gamma_{n,s+d}\right]\otimes t\right)\sqrt{K_n^\circ}}{}{\mathscr{J}_n} \leq
    \begin{cases}
      \frac{\sqrt{2}}{2}\opnorm{t}{}{\mathscr{C}} \text{ if $f(j)=s$,} \\
      0 \text{ otherwise.}
    \end{cases}
  \end{equation*}
  Once more, this generalizes to $\left[\Gamma_{n,j+d},\Gamma_{s+d}\right]$, as claimed in our lemma.

  It is easy to check, using a similar computation as above, that if $j \in \{1,\ldots,d\}$, $p \in \{0,d\}$, $s\in \{d+1,\ldots,d'\}$, and if $k_n(j)=\infty,k_n(s)<\infty$ or $k_n(j)<\infty,k_n(s)=\infty$, then 
  \begin{equation*}
    \opnorm{\left(\left[\Gamma_{n,j+p},\Gamma_{n,s+d}\right]\otimes t\right)\sqrt{K_n^\circ}}{}{\mathscr{J}_n} = 0 \text.
  \end{equation*}
  
  It is also immediate that $[\Gamma_{n,j},\Gamma_{n,s}] = 0$ if $j,s \in \{2d+1,\ldots,d+d'\}$. This concludes our proof.
\end{proof}

We are now ready for the core result of this section.

\begin{lemma}\label{self-adjoint-compact-resolvent-lemma}
  Assume Hypothesis (\ref{Dirac-hyp}). If $n\in\Nbar$, then $\Dirac_n$ is a self-adjoint operator, defined on $\dom{\Dirac_n}$,  with compact resolvent.
\end{lemma}

\begin{proof}
  All the computations in this lemma are done over $\dom{\Dirac_n}$. Let
  \begin{equation*}
    \Dirac_n^\infty = \sum_{\substack{j \in \{1,\ldots,d\} \\ k_n(j)=\infty \\ p \in \{0,d\}}} \Gamma_{n,j+p} \otimes c(\gamma_{j+p}) + \sum_{\substack{j \in \{d+1,\ldots,d'\} \\ k_n(j) = \infty}} \Gamma_{n,d+j}\otimes c(\gamma_{d+j}) \text,
  \end{equation*}
  defined over the domain $\dom{\Dirac_n}$. By construction, $\Dirac_n^\infty$ is symmetric, using the methods in the proof of Lemma (\ref{symmetric-lemma}).

  By construction, $\Dirac_n = \Dirac_n^\infty + \Dirac_n^f$, where
  \begin{equation*}
    \Dirac_n^f = \sum_{\substack{j \in \{1,\ldots,d\} \\ k_n(j) < \infty \\ p \in \{0,d\}}} \Gamma_{n,j+p} \otimes c(\gamma_{j+p}) + \sum_{\substack{j \in \{d+1,\ldots,d'\} \\ k_n(j) < \infty}} \Gamma_{n,d+j}\otimes c(\gamma_{d+j})
  \end{equation*}
  is a bounded self-adjoint operator over $\mathscr{J}_n$. Our focus is on the unbounded operator $\Dirac_n^\infty$.

  A direct computation shows that
  \begin{equation*}
    (\Dirac_n^\infty)^2 = - \sum_{j \in \Upsilon} \Gamma_{n,j}^2 \otimes \unit_{\mathscr{C}}  + \sum_{\substack{j,s \in \Upsilon \\ j < s }} \left[\Gamma_{n,j},\Gamma_{n,s}\right] \otimes c(\gamma_j\gamma_s) \text,
  \end{equation*}
  where $\Upsilon = \{ j, j + d : j\in\{1,\ldots,d\}, k_n(j) = \infty \} \cup \{ j \in \{d+1,\ldots,d'\} : k_n(j) = \infty \}$.

  Let $j\in\{1,\ldots,d\}$ with $k_n(j) = \infty$. Since $X_{n,j}$, $Y_{n,j}$ and $\partial_{n,j}$ commute, and since a direct computation shows that $X_{n,j}^2 + Y_{n,j}^2 = \unit_n$, we conclude that
  \begin{equation*}
    \Gamma_{n,j}^2 + \Gamma_{n,d+j}^2 = X_{n,j}^2\partial_{n,j}^2 + Y_{n,j}^2\partial_{n,j}^2 = (X_{n,j}^2 + Y_{n,j}^2)\partial_{n,j}^2 = \partial_{n,j}^2\text.
  \end{equation*}

  Now, let
  \begin{equation*}
    \Delta_n = - \sum_{\substack{j\in\{1,\ldots,d'\} \\ k'_n(j) = \infty}} \partial_{n,j}^2 \otimes \unit_{\mathscr{C}}\text,
  \end{equation*}
  and note that, for $\xi \in \dom{\Dirac_n^2}$, we have $\Delta_n\xi:m\in\mathscr{J}_n \mapsto \inner{m}{m}{n}\xi(m)$.

  We thus have, over $\dom{\Dirac_n}$, and using Lemma (\ref{Gamma-commutation-lemma}):
  \begin{equation*}
    \left(\Dirac_n^\infty\right)^2 = \Delta_n + F_n
  \end{equation*}
  where $F_n$ is the operator defined over $\dom{\Dirac_n}$ by:
  \begin{multline*}
    F_n = \sum_{\substack{j\in\{1,\ldots,d\} \\ f(j)\leq d \\ k'_n(j)=\infty \\ p,q \in \{0,d\}}} \left[\Gamma_{n,j+p},\Gamma_{n,f(j)+q}\right]\otimes c(\gamma_{j+p} \gamma_{f(j)+q}) \\
    \quad + \sum_{\substack{j\in\{1,\ldots,d\} \\ f(j)>d \\ k'_n(j)=\infty \\ p \in \{0,d\}}} \left[\Gamma_{n,j+p},\Gamma_{n,f(j)}\right]\otimes c(\gamma_{j+p} \gamma_{f(j)+d}) \text,
  \end{multline*}

  We now use the notation of Lemma (\ref{Kn-lemma}). A direct computation shows that
  \begin{equation*}
    \forall \xi \in \mathscr{J}_n \quad (M^2 + \Delta_n) K_n^\circ \xi = \xi \text{ and }\forall \xi \in \range{K_n^\circ} \quad K_n^\circ (M^2 + \Delta_n) \xi = \xi \text.
  \end{equation*}

  Therefore, for all $\xi \in \mathscr{J}_n$,
  \begin{align*}
    (M^2 + (\Dirac_n^\infty)^2) K_n^\circ \xi
    &= (M^2 + \Delta_n) K_n^\circ \xi + F_n K_n^\circ \xi \\
    &= \xi + F_n K_n^\circ \xi = (\unit_{\mathscr{J}_n} + F_n K_n^\circ)\xi \text.
  \end{align*}

  Now, by Lemma (\ref{Gamma-commutation-lemma}), we observe that
  \begin{align*}
    \opnorm{F_n \sqrt{K_n^\circ}}{}{\mathscr{J}_n}
    &\leq \sum_{\substack{j\in\{1,\ldots,d\} \\ f(j)\leq d \\ k'_n(j)=\infty \\ p,q \in \{0,d\}}} \sqrt{2} \opnorm{\gamma_{j+p}\gamma_{f(j)+q}}{}{\mathscr{J}_n} + \sum_{\substack{j\in\{1,\ldots,d\} \\ f(j) > d \\ k'_n(j) = \infty \\ p \in \{0,d\}}} \frac{\sqrt{2}}{2}\opnorm{\gamma_{j+p}\gamma_{f(j)+d}}{}{\mathscr{J}_n} \\
    &\leq 4 d \cdot \sqrt{2} + 2 d \cdot \frac{\sqrt{2}}{2} < 10 d = M \text.
  \end{align*}

  Since $\opnorm{\sqrt{K_n^\circ}}{}{\mathscr{J}_n} = \frac{1}{M}$, we conclude
  \begin{equation*}
    \opnorm{F_n K_n^\circ}{}{\mathscr{J}_n} < 1 \text.
  \end{equation*}

  Consequently, $\unit_{\mathscr{J}_n} + F_n  K_n^\circ$ is invertible; let $R_n = (\unit_{\mathscr{J}_n} + F_n K_n^\circ)^{-1}$ (note that $R_n$ is bounded).

  We then observe that, for all $\xi \in \mathscr{J}_n$, the vector $K_n^\circ R_n \xi$ is in $\dom{\Dirac_n^2}$ by Lemma (\ref{Kn-lemma}). Moreover, we compute:
  \begin{equation*}
    \forall\xi\in\mathscr{J}_n \quad (M^2 + (\Dirac_n^\infty)^2) K_n^\circ R_n \xi = (\unit_{\mathscr{J}_n} + F_n K_n^\circ)R_n \xi = \xi \text.
  \end{equation*}

  Thus $M^2+(\Dirac_n^\infty)^2$ is surjective. Since $(\Dirac_n^\infty + i M)(\Dirac_n^\infty-i M) = (\Dirac_n^\infty -i M)(\Dirac_n^\infty +i M) = M^2 + (\Dirac_n^\infty)^2$ on $\range{K_n^\circ}$, we conclude that both $\Dirac_n^\infty +i M$ and $\Dirac_n^\infty -i M$ are surjective. Thus, $\Dirac_n^\infty$, as a densely defined symmetric operator (by Lemma (\ref{symmetric-lemma})), such that $\Dirac_n^\infty \pm i M$ are surjective, is a self-adjoint operator from $\dom{\Dirac_n}$, using Lemma (\ref{self-adjoint-lemma}).

  \medskip
  
  Moreover, since $K_n^\circ$ is compact and $R_n$ is bounded, the operator $R_n K_n^\circ$ is compact. Now, since $\Dirac_n^\infty$ is self-adjoint, the operator $M^2+(\Dirac_n^\infty)^2$ is invertible, with inverse $R_n K_n^\circ$. Therefore, $M^2+(\Dirac_n^\infty)^2$ has a compact inverse. Since, for any bounded operator $a$, the operator $a a^\ast$ is compact if, and only, if $a$ is compact, we conclude that $(\Dirac_n^\infty+i M)^{-1}$ is compact. Thus, $\Dirac_n^\infty$ has a compact resolvent.

  \medskip

  The operator $\Dirac_n^f$ is bounded, self-adjoint, so the operator $\Dirac_n = \Dirac_n^\infty + \Dirac_n^f$ is also self-adjoint over the domain $\dom{\Dirac_n}$ (and thus $\Dirac_n+i$ is invertible, with a bounded inverse). Moreover, by the resolvent identity:
  \begin{equation*}
    (\Dirac_n + i)^{-1} - (\Dirac_n^\infty + i)^{-1} = (\Dirac_n + i)^{-1}(-\Dirac_n^f)(\Dirac_n^\infty + i)^{-1} \text;
  \end{equation*}
  since $(\Dirac_n^\infty + i)^{-1}$ is compact, and since both $\Dirac_n^f$ and $(\Dirac_n+i)^{-1}$ are bounded, we conclude that $(\Dirac_n + i)^{-1}$ is compact. Thus $\Dirac_n$ is a self-adjoint operator on $\dom{\Dirac_n}$ with compact resolvent. Our lemma is thus proven.
\end{proof}

We thus conclude this section by summarizing our work in the following theorem.
\begin{theorem}\label{spectral-triple-thm}
  Assume Hypothesis (\ref{Dirac-hyp}). For all $n\in\Nbar$, the triple $(\A_n,\mathscr{J}_n,\Dirac_n)$ is a spectral triple, where $\A_n$ acts on $\mathscr{J}_n$ via the *-representation $a\in\A_n\mapsto a^\circ$, such that, if $a \in \A_n$, and if $a$ has finite support, then $a^\circ\dom{\Dirac_n} \subseteq\dom{\Dirac_n}$ and
  \begin{equation*}
    [\Dirac_n,a^\circ] = (\mathrm{id}\otimes c)\grad{n}{a}
  \end{equation*}
  where $\mathrm{id}$ is the identity function on $\B_n$, and
  \begin{multline}\label{gradiant-eq}
    \grad{n}{a} = \sum_{\substack{j \in \{1,\ldots, d\} \\ k_n(j) = \infty}} \left( X_{n,j}\partial_{n}^j(a)\otimes \gamma_j + Y_{n,j} \partial_{n}^j(a)\otimes \gamma_{d+j}\right) \\ + \sum_{\substack{j \in \{1,\ldots,d\} \\ k_n(j) < \infty}} \frac{k_n(j)}{2i \pi} \left( [X_{n,j}, a]\otimes \gamma_j - [Y_{n,j},a]\otimes \gamma_{d+j}\right) \text.
  \end{multline}
\end{theorem}

\begin{proof}
  By Lemma (\ref{self-adjoint-compact-resolvent-lemma}), the operator $\Dirac_n$ is self-adjoint, with a compact resolvent. Moreover, by Lemma (\ref{Udom-dom-lemma}), we have, for all $j \in \{1,\ldots,d\}$:
  \begin{equation*}
    U_{n,j}^\circ \dom{\Dirac_n}\subseteq\dom{\Dirac_n} \text{ and }(U_{n,j}^\ast)^\circ \dom{\Dirac_n}\subseteq \dom{\Dirac_n} \text.
  \end{equation*}
  Thus, for any finitely supported element $a$ in $\A_n$, i.e., any linear combinations of powers of $U_{n,1}$,\ldots,$U_{n,d}$ and their adjoints,
  \begin{equation*}
    a^\circ \, \dom{\Dirac_n} \subseteq \dom{\Dirac_n} \text.
  \end{equation*}

  In particular, since the finitely supported elements in $\A_n$ are dense in $\A_n$, we have shown that
  \begin{equation*}
    \left\{ a \in \A_n : a^\circ \,\dom{\Dirac_n}\subseteq \dom{\Dirac_n} \text{ and }[\Dirac_n,a^\circ] \text{ is bounded} \right\}
  \end{equation*}
  is dense in $\A_n$.
  
  Moreover, if $a\in \A_n$ is finitely supported, then, noting that
  \begin{itemize}
  \item $[\partial_{n,j},a] = \partial_n^j a$ for all $j\in\{1,\ldots,d'\}$ with $k'_n(j) = \infty$; in particular $\partial_n^j a = 0$ for $j\in\{d+1,\ldots,d'\}$,
  \item $[v_n^{z_{n,j}},a] = 0$ if $j\in\{d+1,\ldots,d'\}$ and $k'_n(j) < \infty$,
  \item $[[X_{n,j},\cdot],a]=[X_{n,j},a]$ and $[[Y_{n,j},\cdot],a] = [Y_{n,j},a]$ for all $j\in\{1,\ldots,d\}$ with $k_n(j) < \infty$,
  \item $[X_{n,j},a]=0$ and $[Y_{n,j},a]=0$ for all $j \in \{1,\ldots,d\}$ with $k_n(j) = \infty$,
  \end{itemize}
  we easily conclude that $[\Dirac_n,a^\circ]$ is indeed given by the image by $\mathrm{id}\otimes c$ of Expression (\ref{gradiant-eq}).
\end{proof}

\section{Metric Properties of the Spectral Triples}

We have constructed families of spectral triples on fuzzy and quantum tori. Our goal is to prove that, under natural assumptions, these families are continuous for the spectral propinquity. As a first, key step, we prove that our family of spectral triples induce continuous families of {\qcms s}.

\begin{hypothesis}\label{L-seminorm-hyp}
  Assume Hypothesis (\ref{Dirac-hyp}). For each $n\in\Nbar$, we define $\dom{\Lip_n}$ as the space
  \begin{equation*}
    \left\{ a \in \sa{\A_n} : a^\circ\dom{\Dirac_n}\subseteq\dom{\Dirac_n}\text{ and }\opnorm{[\Dirac_n,a^\circ]}{}{\mathscr{J}_n} < \infty \right\} \text,
  \end{equation*}
  and we define the seminorm $\Lip_n$ on $\dom{\Lip_n}$ by
  \begin{equation*}
    \forall a \in \dom{\Dirac_n} \quad \Lip_n(a) = \opnorm{\left[\Dirac_n,a^\circ\right]}{}{\mathscr{J}_n} \text.
  \end{equation*}

  If $a\in \sa{\A_n}\setminus\dom{\Lip_n}$, we set $\Lip_n(a) = \infty$.
\end{hypothesis}

We begin by recording properties which follow in general from constructing seminorms from spectral triples, as seen in \cite{Latremoliere18g}.

\begin{lemma}\label{three-lemma}
  Assume Hypothesis (\ref{L-seminorm-hyp}). For all $n\in\Nbar$:
  \begin{enumerate}
  \item the domain $\dom{\Lip_n}$ of $\Lip_n$ is dense in $\sa{\A_n}$,
  \item the set $\{ a \in \dom{\Lip_n} : \Lip_n(a) \leq 1 \}$ is closed in $\A_n$,
  \item for all $a,b \in \dom{\Lip_n}$,
    \begin{equation*}
      \frac{a b + b a}{2}, \frac{a b - b a}{2i} \in \dom{\Lip_n}\text,
    \end{equation*}
    and
    \begin{equation*}
      \max\left\{ \Lip_{n}\left(\frac{a b + b a}{2}\right), \Lip_n\left(\frac{a b - b a}{2 i}\right)\right\} \leq \norm{a}{\A_n} \Lip_{n}(b) + \Lip_{n}(a) \norm{b}{\A_n}\text{.}
    \end{equation*}
  \end{enumerate}
\end{lemma}

\begin{proof}
  By Theorem (\ref{spectral-triple-thm}), the domain $\dom{\Lip_n}$ contains all finitely supported elements in $\A_n$ (i.e. elements of the form $\pi_n(f)$ for $f$ finitely supported in $\ell^1(\widehat{G_n})$), and thus $\dom{\Lip_n}$ is dense in $\sa{\A_n}$.

  \bigskip
  
  We refer to \cite{Rieffel00,Latremoliere18g} for the proof of that $\Lip_n$ is lower semi-continuous on $\A_n$, since $\Dirac_n$ is self-adjoint. Therefore, $\{a\in\dom{\Lip_n}:\Lip_n(a)\leq 1\}$ is closed in $\A_n$.
  
  \bigskip

  Let $a,b\in\dom{\Lip_n}$. First, note that $b^\circ \, \dom{\Dirac_n}\subseteq\dom{\Dirac_n}$, and thus
  \begin{equation*}
    (a b)^\circ \dom{\Dirac_n} = a^\circ b^\circ \, \dom{\Dirac_n} \subseteq a^\circ \dom{\Dirac_n} \subseteq \dom{\Dirac_n} \text.
  \end{equation*}
  We then easily compute:
  \begin{equation*}
    \left[\Dirac_n,(a b)^\circ\right] = a^\circ \left[\Dirac_n,b^\circ\right] + \left[\Dirac_n,a^\circ\right] b^\circ \text,
  \end{equation*}
  and all the operators on the right hand side in the previous equation are bounded by assumption on $\dom{\Lip_n}$. Thus $\left[\Dirac_n,(a b)^\circ \right]$ is bounded, and
  \begin{equation*}
    \opnorm{\left[\Dirac_n,(ab)^\circ\right]}{}{\mathscr{J}_n} \leq \norm{a}{\A_n} \Lip_n(b) + \Lip_n(a) \norm{b}{\A_n}\text.
  \end{equation*}

  It is then an easy computation to see that $\left(\frac{a b + b a}{2}\right)^\circ \dom{\Dirac_n}\subseteq\dom{\Dirac_n}$ and
  \begin{align*}
    \Lip_n\left(\frac{ab + ba}{2}\right)
    &\leq \frac{1}{2}\left(\opnorm{\left[\Dirac_n,(ab)^\circ\right]}{}{\mathscr{S}_n} + \opnorm{\left[\Dirac_n,(ba)^\circ\right]}{}{\mathscr{J}_n}\right) \\
    &\leq \norm{a}{\A_n}\Lip_n(b) + \Lip_n(a)\norm{b}{\A_n}\text,
  \end{align*}
  and similarly with the Lie product.
\end{proof}

\medskip

The main tool which we need to both prove that the seminorms $\Lip_n$ are L-seminorms on $\A_n$, and that the sequence $(\A_n,\Lip_n)_{n\in\N}$ of {\qcms s} converges, in the sense of the propinquity, to $(\A_\infty,\Lip_\infty)$, is given by a result about certain approximation of elements in $\dom{\Lip_n}$ by finitely supported elements, where the error in the approximation is controlled by the seminorms $\Lip_n$. This is the matter of the next subsection.

\subsection{A Mean Value Theorem for our Spectral Triples}

We first observe that the dual actions of fuzzy and quantum tori are by Lipschitz automorphisms \cite{Latremoliere16b}. We establish this in several steps.

\begin{lemma}\label{Vz-domain-lemma}
  Assume Hypothesis (\ref{Dirac-hyp}). Let $n\in\Nbar$. If $z\in G_n$, then
  \begin{equation*}
    \left(v_n^z\right)^\circ \dom{\Dirac_n}  = \dom{\Dirac_n} \text.
  \end{equation*}
\end{lemma}

\begin{proof}
  Let $z = (z_1,\ldots,z_{d'})\in G_n$. Let $\xi \in \dom{\Dirac_n}$. We compute:
  \begin{align*}
    \sum_{m\in\Z^{d'}_{k'_n}} (1+\inner{m}{m}{n})\norm{((v_n^z)^\circ \xi)(m)}{\mathscr{C}}^2
    &= \sum_{m\in\Z^{d'}_{k'_n}} (1+\inner{m}{m}{n})\norm{z^m \xi(m)}{\mathscr{C}}^2\\
    &= \sum_{m\in\Z^{d'}_{k'_n}} (1+\inner{m}{m}{n})\norm{\xi(m)}{\mathscr{C}}^2 < \infty
  \end{align*}
  and thus $(v_n^z)^\circ \xi \in \dom{\Dirac_n}$. Therefore, for all $z\in G_n$, we have $(v_n^z)^\circ \, \dom{\Dirac_n}\subseteq\dom{\Dirac_n}$; as $(v_n^z)^\circ$ is a unitary, we then conclude
  \begin{equation*}
    (v_n^z)^\circ \, \dom{\Dirac_n} \subseteq \dom{\Dirac_n} = (v_n^z)^\circ (v_n^{\overline{z}})^\circ \dom{\Dirac_n} \subseteq (v_n^z)^\circ \, \dom{\Dirac_n} \text,
  \end{equation*}
  i.e.
  \begin{equation*}
    \forall z \in G_n \quad (v_n^z)^\circ \, \dom{\Dirac_n} = \dom{\Dirac_n^z} \text.
  \end{equation*}
  This establishes our lemma.
\end{proof}

We now rewrite our operator $\Dirac_n$ in a convenient form.
\begin{lemma}\label{Clifford-anticommutation-lemma}
  Assume Hypothesis (\ref{L-seminorm-hyp}). Let $n\in\Nbar$. Let $E_j = \unit_n \otimes c(\gamma_j)$ for all $j\in\{1,\ldots,d + d'\}$. The following identity holds between operators on $\dom{\Dirac_n}$:
  \begin{equation*}
    \forall j \in \{1,\ldots,d+d'\} \quad E_j \Dirac_n + \Dirac_n E_j = -2 \Gamma_{n,j}\otimes\unit_{\mathscr{C}} \text,
  \end{equation*}
  and, therefore,
  \begin{equation*}
    \Dirac_n = \frac{-1}{2}\sum_{j=1}^{d+d'} \left( E_j \Dirac_n + \Dirac_n E_j \right) E_j \text.
  \end{equation*}
\end{lemma}

\begin{proof}
  Note that, by construction, $\Dirac_n = \sum_{j=1}^{d+d'} \Gamma_{n,j} E_j$.
  
  Owing to the standard relations in the Clifford algebra $\alg{Cl}(\C^{d+d'})$, we have, for all $j,s \in \{1,\ldots,d+d'\}$:
  \begin{equation*}
    E_j E_s + E_s E_j = \begin{cases}
      -2 \text{ if $j=s$,} \\
      0  \text{ otherwise.}
    \end{cases}
  \end{equation*}
  and, by construction, $\Gamma_{n,s}$ and $E_j$ commute. Of course, $E_j$ maps $\dom{\Dirac_n}$ to itself. Thus, on $\dom{\Dirac_n}$,
  \begin{align*}
    E_j \Dirac_n + \Dirac_n E_j
    &= \sum_{s=1}^{d + d'} \Gamma_{n,s}  (E_j E_s + E_s E_j) \\
    &= -2 \Gamma_{n,j}\otimes \unit_{\mathscr{C}} \text.
  \end{align*}
  Our lemma follows.
\end{proof}

From Lemma (\ref{Clifford-anticommutation-lemma}), we deduce the following helpful estimates relating our L-seminorms candidates to the quantized calculus on fuzzy and quantum tori.

\begin{lemma}\label{norm-comparison-lemma}
  Assume Hypothesis (\ref{L-seminorm-hyp}). Let $n\in\Nbar$. If $a\in\dom{\Lip_n}$, then, for all $j\in\{1,\ldots,2d\}$, the following inequality holds:
  \begin{equation*}
    \opnorm{\left[\Gamma_{n,j}, a \right]}{}{\Hilbert_n} \leq \Lip_n(a) \text.
  \end{equation*}

  Consequently, for all $j\in\{1,\ldots,d\}$, if $k_n(j) = \infty$, then,
  \begin{equation*}
    \opnorm{\left[U_{n,f(j)} \partial_{n,j}, a \right]}{}{\Hilbert_n}\leq 2 \Lip_n(a) \text,
  \end{equation*}
  while, if $k_n(j) < \infty$, then,
  \begin{equation*}
    \frac{k_n(j)}{2\pi} \norm{\left[U_{n,j},a \right]}{\B_n} \leq 2 \Lip_n(a) \text.
  \end{equation*}
\end{lemma}

\begin{proof}
  Fix $a\in\dom{\Lip_n}$. Let $E_j = \unit_n \otimes c(\gamma_j)$ for $j\in\{1,\ldots,d+d'\}$. For the following computation, note that since $E_j\, \dom{\Dirac_n}\subseteq \dom{\Dirac_n}$, and since $a^\circ \, \dom{\Dirac_n}\subseteq \dom{\Dirac_n}$, we also have $a^\circ E_j \, \dom{\Dirac_n}\subseteq\dom{\Dirac_n}$; we also note that $[E_j,a^\circ] = 0$ by construction.

  By Lemma (\ref{Clifford-anticommutation-lemma}), if $n\in\Nbar$, $j\in\{1,\ldots,d\}$ and $k_n(j) = \infty$, then, we compute (over the space $\dom{\Dirac_n}$):
  \begin{align*}
    [\Gamma_{n,j},a]^\circ
    &= [X_{n,j} \partial_{n,j} ,a]^\circ \\
    &= \frac{-1}{2} \left[E_j \Dirac_n + \Dirac_n E_j, a^\circ \right] \\
    &= \frac{-1}{2} \left(E_j \left[\Dirac_n,a^\circ\right] + \underbracket[1pt]{[E_j,a^\circ]}_{=0}\Dirac_n + [\Dirac_n,a^\circ] E_j + \Dirac_n\underbracket[1pt]{[E_j,a^\circ]}_{=0}\right) \text,
  \end{align*}
  and therefore, since $\opnorm{E_j}{}{\mathscr{J}_n}\leq 1$, we conclude
  \begin{equation*}
    \opnorm{[X_{n,j}\partial_{n,j},a]^\circ}{\mathscr{J}_n}{\dom{\Dirac_n}} \leq \opnorm{[\Dirac_n,a^\circ]}{\mathscr{J}_n}{\dom{\Dirac_n}} = \Lip_n(a) \text.
  \end{equation*}

  Thus, $[X_{n,j}\partial_{n,j},a]^\circ$ is bounded on $\dom{\Dirac_n}$, and thus has a unique extension to $\mathscr{J}_n$, and thus $[X_{n,j}\partial_{n,j},a]$ is bounded as well, with the same norm, on $\Hilbert_n$. So
  \begin{equation*}
     \opnorm{[\Gamma_{n,j},a]}{}{\Hilbert_n} = \opnorm{[X_{n,j}\partial_{n,j},a]}{}{\Hilbert_n} \leq \Lip_n(a) \text.
  \end{equation*}
  
  The same method would show that
  \begin{equation*}
    \opnorm{[\Gamma_{n,j+d},a]}{}{\Hilbert_n} = \opnorm{[Y_{n,j}\partial_{n,j},a]}{}{\Hilbert_n} = \opnorm{[Y_{n,j}\partial_{n,j},a]^\circ}{\mathscr{J}_n}{\dom{\Dirac_n}} \leq  \Lip_n(a) \text.
  \end{equation*}

  Therefore, since $U_{n,f(j)} = X_{n,j} + i \mathrm{fsgn}(j) Y_{n,j}$, we conclude:
  \begin{equation*}
    \opnorm{\left[U_{n,f(j)}\partial_{n,j},a\right]}{}{\Hilbert_n}
    \leq \opnorm{\left[X_{n,j}\partial_{n,j},a\right]}{}{\Hilbert_n} + \opnorm{\left[Y_{n,j}\partial_{n,j}, a \right]}{}{\Hilbert_n} \leq 2\Lip_n(a)\text.
  \end{equation*}
  It also immediately follows that $\opnorm{\left[U_{n,f(j)}^\ast\partial_{n,j},a\right]}{}{\Hilbert_n} \leq 2 \Lip_n(a)$.
    \medskip

    Similarly, if $j\in\{1,\ldots,d\}$ and $k_n(j)<\infty$, and if $a\in\dom{\Lip_n}$, then
    \begin{equation*}
      \opnorm{[\Gamma_{n,j},a]}{}{\Hilbert_n} = \opnorm{\frac{k_n(j)}{2\pi} [[Y_{n,j},\cdot],a]^\circ}{}{\mathscr{J}_n}\leq \Lip_n(a) \text;
    \end{equation*}
    on the other hand, for all $\xi \in \Hilbert_n$:
    \begin{align*}
      [[Y_{n,j},\cdot],a]\xi
      &= [Y_{n,j},a\xi] - a[Y_{n,j},\xi] \\
      &= Y_{n,j} a \xi - a\xi \cdot Y_{n,j} - a Y_{n,j}\xi + a\xi\cdot Y_{n,j}\\
      &= Y_{n,j} a \xi - a Y_{n,j}\xi = [Y_{n,j},a]\xi \text.
    \end{align*}
    Therefore, $[[Y_{n,j},\cdot],a] = [Y_{n,j},a]$ and thus,
    \begin{equation*}
      \opnorm{[\Gamma_{n,j},a]}{}{\Hilbert_n} = \opnorm{\frac{k_n(j)}{2\pi}[Y_{n,j},a]}{}{\Hilbert_n} \leq \Lip_n(a)\text.
    \end{equation*}

    The remaining inequalities of our lemma are now proven similarly to the above argument.
\end{proof}

We are now able to compute the Lipschitz seminorms of the dual actions, with respect to our L-seminorms.

\begin{notation}
  If $z = (z_1,\ldots,z_{d'}) \in G_n$, then we let
  \begin{equation*}
    \dil{z} = 2 \sum_{j \in \{1,\ldots,d\}} |1-z_{f(j)}| \text.
  \end{equation*}
\end{notation}

\begin{lemma}\label{dilation-lemma}
  Assume Hypothesis (\ref{L-seminorm-hyp}). Let $n\in\Nbar$. If $a\in\dom{\Lip_n}$ and $z\in G_n$, then $\alpha_n^z(a) \in \dom{\Lip_n}$, and
  \begin{equation*}
    \left|\Lip_n(a) - \Lip_n(\alpha_n^z(a)) \right| \leq \dil{z}\Lip_n(a) \text.
  \end{equation*}
\end{lemma}

\begin{proof}
  We note that, since $a^\circ \, \dom{\Dirac_n}\subseteq\dom{\Dirac_n}$, since $(v_n^z)^\circ \, \dom{\Dirac_n}=\dom{\Dirac_n}$ for all $z\in G_n$, and since $\alpha_n^z(a)) = v_n^z a v_n^{\overline{z}}$, we conclude that
  \begin{equation*}
    \alpha_n^z(a)^\circ \, \dom{\Dirac_n}\subseteq \dom{\Dirac_n} \text.
  \end{equation*}

  Note that, for all $z\in G_n$ and for all $j\in\{1,\ldots,d'\}$, we have $\partial_{n,j} v_n^z = v_n^z \partial_{n,j}$ on $\dom{\partial_{n,j}}$.
  
  Let $z = (z_1,\ldots,z_{d'}) \in G_n$. By construction, since $v_n^{\overline{z}} U_{n,f(j)} v_n^{z} = \overline{z_{f(j)}} U_{n,f(j)}$, if $j\in\{1,\ldots,d\}$, and $k_n(j) = \infty$, then (performing the following computations over the dense subspace $\mathscr{F}_n$ of $\Hilbert_n$ consisting of finitely supported elements):
  \begin{align}\label{dilation-lemma-eq-1}
    [U_{n,f(j)} \partial_{n,j},\alpha_n^z(a)]
    &= [U_{n,f(j)} \partial_{n,j}, v_n^z  a  v_n^{\overline{z}}] \\
    &= U_{n,f(j)} \partial_{n,j} v_n^z  a  v_n^{\overline{z}} - v_n^z  a  v_n^{\overline{z}} U_{n,f(j)} \partial_{n,j} \nonumber \\
    &= U_{n,f(j)} v_n^z \partial_{n,j} a  v_n^{\overline{z}} - v_n^z   a \overline{z_{f(j)}} U_{n,f(j)} v_{n,j}^{\overline{z}} \partial_{n,j} \nonumber \\
    &= \overline{z_{f(j)}} v_n^z \left( U_{n,f(j)} \partial_{n,j} a - a U_{n,j} \partial_{n,j} \right) v_n^{\overline{z}} \nonumber \\
    &=  v_n^z \left( \overline{z_{f(j)}} \left[ U_{n,f(j)} \partial_{n,j}, a \right] \right) v_n^{\overline{z}} \text. \nonumber
  \end{align}
  In particular,
  \begin{equation*}
    \opnorm{[U_{n,f(j)}\partial_{n,j}, \alpha_n^z(a)]}{\Hilbert_n}{\mathscr{F}_n} = \opnorm{[U_{n,f(j)}\partial_{n,j},a]}{\Hilbert_n}{\mathscr{F}_n} < \infty \text.
  \end{equation*}

  Thus, using the density of $\mathscr{F}_n$, we conclude that
  \begin{equation*}
    \opnorm{[U_{n,f(j)}\partial_{n,j},\alpha_n^z(a)]}{}{\Hilbert_n}  < \infty \text.
  \end{equation*}
  
  Since $a\in\dom{\Lip_n}$, we know that $a=a^\ast$, and since $f(j)\neq j$ by assumption on $f$, so $\partial_{n,j}$ and $U_{n,f(j)}$ commute, we easily conclude as well that
  \begin{equation*}
    \opnorm{[U_{n,f(j)}^\ast\partial_{n,j},\alpha_n^z(a)]}{}{\Hilbert_n} = \opnorm{[U_{n,f(j)}^\ast\partial_{n,j},a]}{}{\Hilbert_n} < \infty \text.
  \end{equation*}
  Therefore, for all $j \in \{1,\ldots,d\}$, if $k_n(j) = \infty$, then
  \begin{equation*}
    \opnorm{\left[\Gamma_{n,j},\alpha_n^z(a)\right]}{}{\Hilbert_n} < \infty \text{ and }\opnorm{\left[\Gamma_{n,d+j},\alpha_n^z(a)\right]}{}{\Hilbert_n} < \infty \text.
  \end{equation*}
  By construction, $v_n^z$ commutes with $\partial_{n,j}$ and $v_n^w$, for all $j\in\{d+1,\ldots,d'\}$ and $w\in G_n$. Thus, if $j\in\{2d+1,\ldots,d+d'\}$, then by construction, $v_n^z \Gamma_{n,j} v_n^{\overline{z}} = \Gamma_{n,j}$, so $[\Gamma_{n,j},\alpha_n^z(a)] = v_n^z [\Gamma_{n,j},a] v_n^{\overline{z}}$. Thus, for all $j\in\{2d+1,\ldots,d+d'\}$:
  \begin{equation*}
    \opnorm{[\Gamma_{n,j},\alpha_n^z(a)]}{}{\Hilbert_n} = \opnorm{[\Gamma_{n,j},a]}{}{\Hilbert_n}\text.
  \end{equation*}
  
  We thus conclude that
  \begin{equation*}
    \Lip_n(\alpha_n^z(a)) = \opnorm{[\Dirac_n,\alpha_n^z(a)^\circ ]}{}{\Hilbert_n} < \infty\text,
  \end{equation*}
  and therefore, $\alpha_n^z(a) \in \dom{\Lip_n}$, as claimed.

  \bigskip

  To conclude our estimation of the Lipschitz constant for the dual action, we make a few quick observations. Fix $z\in G_n$. If $j\in\{1,\ldots,d\}$ and $k_n(j) = \infty$, then, by Expression (\ref{dilation-lemma-eq-1}):
  \begin{equation*}
    [U_{n,f(j)} \partial_{n,j}, \alpha_n^z(a)]
    = v_n^z \left( \overline{z_{f(j)}} [U_{n,f(j)} \partial_{n,j}, a] \right) v_n^{\overline{z}} \text.
  \end{equation*}
  A similar computation shows that
  \begin{equation*}
    [U_{n,f(j)}^\ast \partial_{n,j}, \alpha_n^z(a)] = v_n^z \left( z_{f(j)} [U_{n,f(j)}^\ast \partial_{n,j}, a] \right) v_n^{\overline{z}} \text.
  \end{equation*}

  Thus, a direct computation shows that, for all $s\in\{1,\ldots,d\}$, and $k_n(s) = \infty$, if $j \in \{s,s+d\}$, then, using Lemma (\ref{norm-comparison-lemma}),
  \begin{equation*}
    \opnorm{v_n^z [\Gamma_{n,j}, a] v_n^{\overline{z}} - [\Gamma_{n,j},\alpha_n^z(a)]}{}{\Hilbert_n} \leq 2 |1-z_{f(j)}| \Lip_n(a) \text.
  \end{equation*}

  \medskip
  
  Since $v_n^z$ is unitary, we have
  \begin{equation*}
    \Lip_n(a) = \opnorm{(v_n^z)^\circ [\Dirac_n,a^\circ] (v_n^{\overline{z}})^\circ}{}{\mathscr{J}_n} \text.
  \end{equation*}
  Therefore, we conclude:
  \begin{align*}
    \left|\Lip_n(a) - \Lip_n(\alpha_n^z(a))\right|
    &\leq \opnorm{(v_n^z)^\circ [\Dirac_n,a^\circ] (v_n^{\overline{z}})^\circ - [\Dirac_n,\alpha_n^z(a)^\circ]}{}{\mathscr{J}_n} \\
    &\leq \sum_{j=1}^{d+d'} \opnorm{v_n^z[\Gamma_{n,j},a]v_n^{\overline{z}} - [\Gamma_{n,j},\alpha_n^z(a)]}{}{} \\
    &\leq \sum_{j\in\{1,\ldots,d\}} 2 |1-z_{f(j)}| \Lip_n(a) \\
    &\leq \dil{z} \Lip_n(a) \text.
  \end{align*}

  This concludes our proof.
\end{proof}

We can approximate elements of a quantum or fuzzy torus by finitely supported elements, in a manner which gives us control over the support of the approximations. For this purpose, we use integral operators on $\A_n$ defined using kernels, which, ultimately, we will choose to be Fejer kernels.

\begin{notation}
  We use the notations of Hypothesis (\ref{L-seminorm-hyp}). If $f \in C(\T^d)$ is a $\C$-valued continuous function over the $d$-torus $\T^d$, if $n\in\Nbar$, and if $\mu_n$ is the Haar probability measure on $\U^d_{k_n}$, then we then define, for all $a\in \A_n$, the following element of $\A_n$:
  \begin{equation*}
    \alpha_n^f (a) = \int_{\U^d_{k_n}} f(z) \alpha_n^z(a) \, d\mu_n(z) \text.
  \end{equation*}
\end{notation}

\begin{lemma}\label{near-iso-lemma}
  Assume Hypothesis (\ref{L-seminorm-hyp}), and let $n\in\Nbar$. If $f \in C(\U^d_{k_n})$, and if $f(\U^d_{k_n}) \subseteq [0,\infty)$,  then
  \begin{equation*}
    \forall a \in \dom{\Lip_n} \quad \Lip_n(\alpha_n^f(a)) \leq  \left(\int_{\U_{k_n}^d} f(z)(1 + \dil{z}) \, d\mu_n(z)\right)\Lip_n(a) \text.
  \end{equation*}
\end{lemma}

\begin{proof}
  First, we note that $\alpha_n^f(a) \in \sa{\A_n}$ whenever $a\in\sa{\A_n}$. In this proof, we recall that $\Lip_n(a)$ is well-defined, if possibly infinite, for all $a\in\A_n$.

  \medskip
  
  Since $\Lip_n$ is a lower semi-continuous seminorm, we conclude
  \begin{align*}
    \Lip_n(\alpha_n^f(a))
    &\leq \int_{\U^d_{k_n}} |f(z)| \Lip_n(\alpha_n^z(a)) \, d\mu_n(z) \\
    &\leq \int_{\U^d_{k_n}} f(z)(1 + \dil{z}) \Lip_n(a) \, d\mu_n(z) \text{ by Corollary (\ref{dilation-lemma}) and $f\geq 0$,}\\
    &= \left(\int_{\U^d_{k_n}} f(z)(1 + \dil{z}) \, d\mu_n(z)\right) \Lip_n(a) \text.
  \end{align*}
  This completes our proof.
 \end{proof}

\begin{lemma}\label{first-approx-lemma}
  Assume Hypothesis (\ref{L-seminorm-hyp}). Let $n\in\Nbar$. If $(f_m)_{m\in\N}$ is a sequence of nonnegative elements in $C(\U^d_{k_n})$ such that, for all $g \in C(\U^d_{k_n})$,
  \begin{equation*}
    \lim_{m\rightarrow\infty} \int_{\U^d_{k_n}} f_m(z)g(z)\, d\mu_n(z) = g(1,\ldots,1)\text,
  \end{equation*}
  then
  \begin{equation*}
    \forall a \in \A_n \quad \lim_{m\rightarrow\infty} \norm{a-\alpha_n^{f_m}(a)}{\A_n} = 0 \text.
  \end{equation*}
\end{lemma}

\begin{proof}
  Let $n\in\Nbar$ and let $a\in\A_n$. For any $m\in\N$, we simply compute:
  \begin{align*}
    0
    &\leq \norm{a-\alpha_n^{f_m}(a)}{\A_n} \leq \int_{\U^d_{k_n}} f_m(z) \norm{a-\alpha_n^z(a)}{\A_n} \,d\mu_n(z) \\
    &\xrightarrow{m\rightarrow\infty} \norm{a-a}{\A_n} = 0 \text.
  \end{align*}
  Thus, our lemma is proven.
\end{proof}

We thus obtain a first approximation result for Lipschitz elements. To this end, we use the following well-known lemma as a source of kernels --- in the given reference, these kernels are F{\'e}jer kernels.

\begin{lemma}[{\cite{Latremoliere05}}]\label{Fejer-lemma}
  Let $q\in\N_\ast$. Let $k\in \Nbar_\ast^{q}$, and let $\mu_k$ be the Haar probability measure on $\U_k^d$. There exists a sequence $(f_m)_{m\in\N}$ in $C(\U_{k}^q)$ such that:
  \begin{enumerate}
  \item $\forall m \in \N \quad f(\U_{k}^q) \subseteq [0,\infty)$,
  \item $\forall m \in \N \quad f_m(0) \neq 0$,
  \item $\forall m \in \N \quad \int_{\U_{k}^q} f_m(z) \, d\mu_m(z) = 1$,
  \item $\forall g \in C(\U_{k_n}^q) \quad \lim_{m\rightarrow\infty} \int_{\U_{k}^d} f_m(z) g(z) \, d\mu_k(z) = g(1,\ldots,1)$,
  \item for all $m\in\N$, the function $f_m$ is a linear combination of characters of $\U_k^d$.
  \end{enumerate}
\end{lemma}

\begin{proof}
  We refer to \cite[Lemma 3.1, Lemma 3.2, Lemma 3.6]{Latremoliere05} for this lemma.
\end{proof}

Putting these observations together, we conclude:

\begin{corollary}\label{second-approx-cor}
  Assume Hypothesis (\ref{L-seminorm-hyp}). Let $n\in\Nbar$. For all $\varepsilon > 0$, and for all $a \in \dom{\Lip_n}$, there exists $b \in \dom{\Lip_n}$, with $b$ \emph{finitely supported}, such that
  \begin{equation*}
    \Lip_n(b) \leq (1+\varepsilon) \Lip_n(a) \text{ and }\norm{a - b}{\A_n} < \varepsilon \text.
  \end{equation*}
\end{corollary}

\begin{proof}
  Let $n\in\Nbar$ and let $(f_m)_{m\in\N}$ be a sequence in $C(\U_{k_n}^d)$ given by Lemma (\ref{Fejer-lemma}) (for $q=d$ and $k=k_n$).
  
  Let $a \in \dom{\Lip_n}$. Let $\varepsilon > 0$. By Lemma (\ref{first-approx-lemma}), there exists $M\in\N$ such that, if $m\geq M$, then $\norm{a - \alpha_n^{f_m}(a)}{\A_n} < \varepsilon$. There exists $M'\in \N$ such that, if $m\geq M'$, then $\left| 1  - \int_{\U^d_{k_n}} f(z)(1 + \dil{z})\,d\mu_n(z) \right| < \varepsilon$ --- since $1+\dil{1,\ldots,1} = 1$. Let $f = f_{\max\{M,M'\}}$,and let $b = \alpha_n^{f}(a)$.

  By Lemma (\ref{near-iso-lemma}), $\Lip_n(b) \leq (1+\varepsilon)\Lip_n(a)$.  By \cite[Lemma 3.2]{Latremoliere05}, if $S$ is the support of the Fourier transform of $f$, the range of $\alpha_n^{f}$ is the image by $\pi_n$ of
  \begin{equation*}
    \left\{ f \in \ell^1(\Z^d_{k_n}) : f\left( \Z^d_{k_n}\setminus S \right) = \{ 0 \} \right\} \text.
  \end{equation*}
  Since $S$ is finite by assumption on $f$, the range of $\alpha_n^f$ consists of finitely supported elements in $\A_n$. This concludes our lemma.
\end{proof}

Theorem (\ref{spectral-triple-thm}) gives us, for all $n\in\Nbar$, an explicit formula for $\Lip_n(a)$ whenever $a\in\dom{\Lip_n}$ is finitely supported. We now use this expression to obtain a natural bound on $\Lip_n(a)$ in terms of the action of $\U^d_{k_n}$ on $\A_n$.

\begin{lemma}\label{derivations-Ln-lemma}
  Assume Hypothesis (\ref{L-seminorm-hyp}). Let $n\in\Nbar$. If $a\in\dom{\Lip_n}$ is \emph{finitely supported}, then,
  \begin{equation*}
    \max\left\{ \norm{\partial_{n}^j(a)}{\A_n} : j \in \{1,\ldots,d\}, k_n(j) = \infty \right\} \leq 2 \Lip_n(a)
  \end{equation*}
  and 
  \begin{equation*}
    \max\left\{ \norm{\frac{k_n(j)}{2\pi} (\alpha_n^{z_{n,j}}(a) - a)}{\A_n} : j \in \{1,\ldots,d\}, k_n(j) < \infty \right\} \leq 2 \Lip_n(a) \text.
  \end{equation*}
\end{lemma}

\begin{proof}
  If $a \in \sa{\A_n}$ is finitely supported, then $[\partial_{n,j},a] = \partial_{n}^j(a)$ by Lemma (\ref{partial-commutator-lemma}). We then conclude the following, for all $j\in\{1,\ldots,d\}$. If $k_n(j) = \infty$, then
  \begin{align*}
    \norm{\partial_{n}^j(a)}{\A_n}
    &= \opnorm{[\partial_{n,j},a]}{}{\Hilbert_n} \\
    &= \norm{U_{n,f(j)} [\partial_{n,j},a]}{\B_n} \text{ as $U_{n,f(j)}$ is unitary,} \\
    &= \norm{[U_{n,f(j)}\partial_{n,j},a]}{\B_n} \text{ as $U_{n,f(j)}$ commutes with $a$ by Hyp. (\ref{innerification-hyp}),} \\
  &\leq  2 \Lip_n(a) \text{ using Lemma (\ref{norm-comparison-lemma}).}
  \end{align*}
  
  Similarly, if $k_n(j) < \infty$, then, precisely thanks to Hypothesis (\ref{innerification-hyp}), if $j < f(j)$, then
  \begin{align*}
    \norm{\frac{k_n(j)}{2 \pi} (\alpha_n^{z_{n,j}}(a)-a)}{\A_n}
    &\leq \norm{\frac{k_n(j)}{2 \pi} (U_{n,f(j)} a U_{n,f(j)}^\ast - a)}{\B_n} \\
    &= \norm{\frac{k_n(j)}{2 \pi} [U_{n,f(j)},a] U_{n,f(j)}^\ast}{\B_n} \\
    &= \norm{\frac{k_n(j)}{2 \pi} [U_{n,f(j)},a]}{\B_n} \\
    &\leq 2 \Lip_n(a) \text{ using Lemma (\ref{norm-comparison-lemma}).}
  \end{align*}
  If, instead,  $k_n(j) < \infty$, and if $f(j) < j$, then, similarly,
  \begin{align*}
    \norm{\frac{k_n(j)}{2 \pi} (\alpha_n^{z_{n,j}}(a)-a)}{\A_n}
    &\leq \norm{\frac{k_n(j)}{2 \pi} (U_{n,f(j)}^\ast a U_{n,f(j)} - a)}{\B_n} \\
    &= \norm{\frac{k_n(j)}{2 \pi} [U_{n,f(j)}^\ast,a] U_{n,f(j)}}{\B_n} \\
    &\leq 2 \Lip_n(a) \text.
  \end{align*}

  This completes our proof.
\end{proof}

We now prove some basic estimates, relating the geometry of the torus and the quantum torus. To this end, we define a translation invariant Riemannian metric on $\T^d$ by endowing the Lie algebra $\alg{u}^d_{\infty^d} = \R^d$ of $\T^d$ with the usual inner product $\inner{\cdot}{\cdot}{}$ of $\R^d$ (whose associated norm we denote by $\norm{\cdot}{\R^d}$). The distance from the unit of $\T^d$ to any point for the path metric induced by this Riemannian metric is a continuous length function on $\T^d$, given by the expression:
\begin{equation*}
  \forall z \in \T^d \quad \length{z} = \min\left\{ \norm{X}{\R^d} : z = \exp_{\T^d}(X) \right\} \text.
\end{equation*}
While $\U^d_k$ has a trivial Lie group structure when $k\in\N_\ast^d$, it does inherit from the metric structure on $\T^d$ a length function by restricting $\length{\cdot}$ to it.

\begin{remark}
  Assume Hypothesis (\ref{Dirac-hyp}). For all $n\in\Nbar$, and for all $j\in\{1,\ldots,d\}$, the following equality holds:
  \begin{equation*}
    \length{z_{n,j}} = \begin{cases}
      \frac{2\pi}{k_n(j)} \text{ if $k_n(j) <\infty$, }\\
      0 \text{ if $k_n(j) = \infty$.}
    \end{cases}
  \end{equation*}
\end{remark}

We now compute an estimate for $\norm{a-\alpha_{k,\sigma}^z(a)}{\A_n}$, valid for any $(k,\sigma)\in\Xi^d$, for all $a\in\qt{k}{\sigma}$, and for all $z\in \U^d_k$. We obtain this estimate in three steps. 

\begin{lemma}\label{mvt-discrete-lemma}
  Let $d\in\N_\ast$. Let $(k,\sigma) \in \Xi^d$ with $k=(k(1),\ldots,k(d))$. If $j\in\{1,\ldots,d\}$ such that $k(j) < \infty$, if
  \begin{equation*}
    z_{k,j} = \left(\underbracket[1pt]{1,\ldots,1,\exp\left(\frac{2 i \pi}{k(j)}\right)}_{j \text{ elements}},1,\ldots,1\right) \in \U^d_k \text,
  \end{equation*}
  then for all $a\in\qt{k}{\sigma}$, for all $n\in\Z$, and setting $z=z_{k,j}^n$, we compute:
  \begin{equation*}
    \norm{a - \alpha_{k,\sigma}^{z}(a)}{\qt{k}{\sigma}} \leq \length{z} \norm{\frac{\alpha_{k,\sigma}^{z_{k,j}}(a)-a}{\frac{2\pi}{k(j)}}}{\qt{k}{\sigma}}
  \end{equation*}
\end{lemma}

\begin{proof}
  Let $(k,\sigma) \in \Xi^d$ with $k\neq\infty^d$. Let $j \in \{1,\ldots,d\}$ such that $k(j) < \infty$, and let $a\in \qt{k}{\sigma}$.

  Let $z = z_{k,j}^n$ for $n \in \N$. Our lemma's conclusion is trivial when $n = 0$, so we assume that $z\neq 1$.

  Let $\beta = \alpha_{k,\sigma}^{z_{k,j}}$ and note that $\beta$ is an isometry of $\qt{k}{\sigma}$. Now:
\begin{align*}
  \norm{a - \alpha_{k,\sigma}^z(a)}{\qt{k}{\sigma}}
  &\leq \sum_{w=0}^{n-1} \norm{\beta^{w}(a)-\beta^{w+1}(a)}{\qt{k}{\sigma}} \\
  &= \sum_{w=0}^{n-1} \norm{\beta^{w}\left(a-\beta(a)\right)}{\qt{k}{\sigma}} \\
  &= \sum_{w=0}^{n-1} \norm{a-\beta(a)}{\qt{k}{\sigma}} \\
  &= n \length{z_{k,j}} \norm{\frac{\beta(a)-a}{\length{z_{k,j}}}}{\qt{k}{\sigma}}\text{.}
\end{align*}

If $n \in \Z$, and $z = z_{k,j}^n$, then
\begin{align*}
  \norm{a-\alpha_{k,\sigma}^z(a)}{\qt{k}{\sigma}}
  &= \norm{\alpha_{k,\sigma}^{\overline{z}}(a)-a}{\qt{k}{\sigma}} \\
  &\leq (-n)\length{z_{k,j}}\norm{\frac{\beta(a)-a}{\length{z_{k,j}}}}{\qt{k}{\sigma}}\text{.}
\end{align*}

Therefore,
\begin{align*}
  \norm{a-\alpha_{k,\sigma}^z(a)}{\qt{k}{\sigma}}\leq |n|\length{z_{k,j}}\norm{\frac{\beta(a)-a}{\length{z_{k,j}}}}{\qt{k}{\sigma}}\text{,}
\end{align*}
and therefore, since $\length{z} = \min\left\{ \frac{2\pi |n|}{k(j)} : z_{k,j}^n = z \right\}$, we conclude that our lemma holds.
\end{proof}

We follow \cite[Theorem 3.1]{Rieffel98a} for our next lemma.

\begin{lemma}\label{mvt-continuous-lemma}
  Let $d\in\N_\ast$, $(k,\sigma) \in \Xi^d$, with $k=(k(1),\ldots,k(d))$. If $j \in \{1,\ldots,d\}$, if $k(j) = \infty$, and if
  \begin{equation*}
    z = \exp_{\U^d_k}\left( t e_j \right)
  \end{equation*}
  for some $t \in \R$, then for all finitely supported element $a \in\qt{k}{\sigma}$:
  \begin{equation*}
    \norm{ a - \alpha_{k,\sigma}^z(a) }{\qt{k}{\sigma}} \leq \length{z} \norm{\partial_{k}^{e_j}(a)}{\qt{k}{\sigma}} \text{.}
  \end{equation*}
\end{lemma}

\begin{proof}
  Let $a$ be finitely supported. Let $z = \exp_{\U^d_k}(te_j)$, for some $t\in\R$. Now, let $f : s \in \R\mapsto \alpha_{k,\sigma}^{\exp_{\U_k^d}(s e_j)}(a)$. Note that, for all $s,r \in\R$, with $r\neq 0$, we have:
  \begin{align*}
    \frac{1}{r}\left(\alpha_{k,\sigma}^{\exp((s+r)e_j)}(a) - \alpha_{k,\sigma}^{\exp(s e_j)}(a)\right)
    &= \alpha_{k,\sigma}^{\exp_{\U_k^d}(s e_j)}\left(\frac{1}{r}\left(\alpha_{k,\sigma}^{\exp_{\U^d_k}(r e_j)}(a) - a \right)\right) \\
    &\xrightarrow{r\rightarrow 0} \alpha_{k,\sigma}^{\exp_{\U^d_k}(s e_j)}(\partial_{k}^{e_j}(a)) \text.
  \end{align*}

  Therefore,
  \begin{align*}
    \norm{ a - \alpha_{k,\sigma}^z(a) }{\qt{k}{\sigma}}
    &= \norm{\int_0^t \alpha_{k,\sigma}^{\exp_{\U_k^d}(s e_j)}\left(\partial^{e_j}_{k,\sigma}(a)\right) \, ds}{\qt{k}{\sigma}}\\
    &\leq \int_0^t \norm{\alpha_{k,\sigma}^{\exp_{\U^d_k}(s e_j)}\left(\partial^{e_j}_{k,\sigma}(a)\right)}{\qt{k}{\sigma}} \, ds \\
    &\leq |t| \norm{\partial_{k}^{e_j}(a)}{\qt{k}{\sigma}} \text{.}
  \end{align*}

  Now, $\min\{|t| :\exp_{\U^d_k}(t e_j) = z\} = \length{z}$ and thus our lemma is proven.
\end{proof}

While Lemma (\ref{mvt-continuous-lemma}) could be improved to provide a similar estimate for any $z$ in the connected component of $\U^d_k$, the estimate in Lemma (\ref{mvt-discrete-lemma})  does not work quite as nicely. Instead, we obtain the following lemma, which will suffice.

\begin{notation}
  Let $d\in\N_\ast$. For all $(z_1,\ldots,z_d) \in \T^d$, we define
  \begin{equation*}
    \mathsf{slen}(z_1,\ldots,z_d) = \sum_{j=1}^d \length{\underbracket[1pt]{1,\ldots,1,z_j}_{j \text{ elements}},1,\ldots,1}\text.
  \end{equation*}

\end{notation}

\begin{lemma}\label{mvt-lemma}
  We use the notations of Lemmas (\ref{mvt-discrete-lemma}) and (\ref{mvt-continuous-lemma}). Let $d\in\N_\ast$. Let $(k,\sigma) \in \Xi^d$, where $k=(k(1),\ldots,k(d))$. Let $\Upsilon = \{ j \in \{1,\ldots,d\} : k(j) = \infty\}$. 

  For all finitely supported $a \in \qt{k}{\sigma}$, and for all $z \in \U^d_k$, we estimate: 
  \begin{multline*}
    \norm{a - \alpha_{k,\sigma}^z(a)}{\qt{k}{\sigma}} \leq \\  \mathsf{slen}(z) \max\left\{ \max_{j\in\Upsilon} \norm{\partial_{k}^{e_j}(a)}{\qt{k}{\sigma}},\max_{j\in\{1,\ldots,d\}\setminus\Upsilon} \norm{\frac{\alpha_{k,\sigma}^{z_{k,j}}(a)-a}{\frac{2\pi}{k(j)} }}{\qt{k}{\sigma}} \right\} \text{.}
  \end{multline*}
\end{lemma}

\begin{proof}
Let $K = \max\left\{ \max_{j\in\Upsilon} \norm{\partial_{k}^{e_j}(a)}{\qt{k}{\sigma}}, \max_{\substack{j\in\{1,\ldots,d\}\\ j\notin \Upsilon}} \norm{\frac{\alpha_{k,\sigma}^{z_{k,j}}(a)-a}{\length{z_{k,j}}}}{\qt{k}{\sigma}} \right\}$. We conclude, using both Lemma (\ref{mvt-discrete-lemma}) and (\ref{mvt-continuous-lemma}):
\begin{align*}
  \norm{ a - \alpha_{k,\sigma}^z(a) }{\qt{k}{\sigma}}
  &\leq\norm{a-\alpha_{k,\sigma}^{z_1,1,\ldots,1}(a)}{\qt{k}{\sigma}} + \norm{\alpha_{k,\sigma}^{z_1,1,\ldots,1}(a)-\alpha_{k,\sigma}^{z_1,z_2,\ldots,1}(a)}{\qt{k}{\sigma}} \\
  &\quad+ \cdots + \norm{\alpha_{k,\sigma}^{z_1,z_2,\ldots,z_{d-1},1}(a)-\alpha_{k,\sigma}^{z_1,\ldots,z_d}(a)}{\qt{k}{\sigma}} \\
  &=\norm{a-\alpha_{k,\sigma}^{z_1,1,\ldots,1}(a)}{\qt{k}{\sigma}} + \norm{\alpha_{k,\sigma}^{z_1,1,\ldots,1}\left(a-\alpha_{k,\sigma}^{1,z_2,\ldots,1}(a)\right)}{\qt{k}{\sigma}} \\
  &\quad + \cdots + \norm{\alpha_{k,\sigma}^{z_1,z_2,\ldots,z_{d-1},1}\left( a -\alpha_{k,\sigma}^{1,\ldots,1,z_d}(a)\right)}{\qt{k}{\sigma}} \\
  &=\norm{a-\alpha_{k,\sigma}^{z_1,1,\ldots,1}(a)}{\qt{k}{\sigma}} + \norm{a-\alpha_{k,\sigma}^{1,z_2,\ldots,1}(a)}{\qt{k}{\sigma}} \\
  &\quad + \cdots + \norm{a -\alpha_{k,\sigma}^{1,\ldots,1,z_d}(a)}{\qt{k}{\sigma}} \\
  &\leq \sum_{j=1}^d \length{\underbracket[1pt]{1,\ldots,1,z_j}_{j \text{ elements}},1,\ldots,1} K  \\
  &\leq \mathsf{slen}(z) K \text{,}
\end{align*}
as desired.
\end{proof}

We thus arrive at our conclusion for this section.

\begin{corollary}\label{mvt-cor}
  Assume  Hypothesis (\ref{L-seminorm-hyp}). For all $n\in\Nbar$, if $a\in\dom{\Lip_n}$, then
  \begin{equation}\label{mvt-eq}
    \forall z \in \U^d_{k_n} \quad \norm{a-\alpha_n^z(a)}{\A_n} \leq 2 \; \mathsf{slen}(z) \Lip_n(a) \text.
  \end{equation}
\end{corollary}

\begin{proof}
  Expression (\ref{mvt-eq}) holds if $a$ is finitely supported by Lemmas (\ref{derivations-Ln-lemma}) and (\ref{mvt-lemma}).

  Now, let $a\in\dom{\Lip_n}$. Let $\varepsilon > 0$. Let
  \begin{equation*}
    \varepsilon' = \frac{\varepsilon}{2(1 + \mathsf{slen}(z)\Lip_n(a))} > 0 \text.
  \end{equation*}

  By Corollary (\ref{second-approx-cor}), there exists a \emph{finitely supported} element $b \in \dom{\Lip_n}$ such that
  \begin{equation*}
    \Lip_n(b) \leq \left(1+\varepsilon'\right) \Lip_n(a) \text{ and }\norm{a-b}{\A_n} < \varepsilon' \; \Lip_n(a)  \text.
  \end{equation*}

  Therefore, for all $z\in \U^d_{k_n}$:
  \begin{align*}
    \norm{a-\alpha_n^z(a)}{\A_n}
    &\leq \norm{a-b}{\A_n} + \norm{b-\alpha_n^z(b)}{\A_n} + \norm{\alpha_n^z(b - a)}{\A_n} \\
    &< \varepsilon' + 2 \; \mathsf{slen}(z) \Lip_n(b) + \varepsilon' \\
    &\leq 2 \varepsilon' + 2\;\mathsf{slen}(z) \left(1 + \varepsilon'\right) \Lip_n(a) \\
    &\leq \varepsilon'(2 + 2\mathsf{slen}(z)\Lip_n(a)) + 2\; \mathsf{slen}(z)\Lip_n(a)\\
    &\leq \varepsilon + 2 \; \mathsf{slen}(z)\Lip_n(a) \text.
  \end{align*}
  As $\varepsilon > 0$ is arbitrary, we have proven our lemma.
\end{proof}

\bigskip

As a consequence of Corollary (\ref{mvt-cor}), we obtain one more necessary property toward the proof that our spectral triples are, indeed, metric.

\begin{corollary}\label{zero-cor}
  Assume Hypothesis (\ref{L-seminorm-hyp}). For all $n\in \Nbar$, and for all $a\in\dom{\Lip_n}$,
  \begin{equation*}
    \Lip_n(a) = 0 \iff a \in \R\unit_{\A_n}\text.
  \end{equation*}
\end{corollary}

\begin{proof}
  Of course, $\Lip_n(t\unit_{\A_n}) = 0$ for all $t\in\R$.

  Now, let $a\in\dom{\Lip_n}$ such that $\Lip_n(a) = 0$. By Corollary (\ref{mvt-cor}),
  \begin{equation*}
    \forall z \in \U^d_{k_n} \quad \norm{a-\alpha_n^z(a)}{\A_n} \leq 2 \mathsf{slen}(z) \Lip_n(a) = 0
  \end{equation*}
  and thus $\alpha_n^z(a) = a$ for all $z\in \U^d_{k_n}$. As the action of $\U^d_{k_n}$ induced by $\alpha_n$ on $\A_n$ is ergodic, we conclude that $a \in \R\unit_{\A_n}$, as desired.
\end{proof}

We now bring together the tools developed in this section, to obtain our principal approximation theorem.

\begin{theorem}\label{L-approx-thm}
  Assume Hypothesis (\ref{L-seminorm-hyp}). For all $\varepsilon > 0$, there exists $N\in\N$, and a function $f \in C(\T^d)$, whose Fourier transform is supported on a finite subset $S$ of $\Z^{d}$ with $0\in S$, and such that for all $n \in \Nbar$, if $n \geq N$, and for all $a \in \dom{\Lip_n}$, we have:
    \begin{equation*}
      \norm{ a - \alpha_n^f(a) }{\A_n} \leq \varepsilon \Lip_n(a) \text{ and }\Lip_n(\alpha_n^f(a)) \leq (1+\varepsilon)\Lip_n(a) \text.
    \end{equation*}
    In particular, $\alpha_n^f(a)$ is in the linear span of $\left\{ \delta_m : m \in q(S) \right\}$, where $q : \Z^d \twoheadrightarrow \Z^d_{k_n}$ is the canonical surjection.
\end{theorem}

\begin{proof}
  Let $\varepsilon > 0$.
  
  By Lemma (\ref{Fejer-lemma}), there exists a linear combination $f : \T^d \rightarrow [0,\infty)$ of characters of $\T^d$ such that $\int_{\T^d} f \, d\mu_{\infty} = 1$ and
  \begin{equation*}
    \int_{\T^d} f(z)\mathsf{slen}(z)\, d\mu_{\infty}(z) \leq \frac{\varepsilon}{4} \text{ and } \int_{\T^d} f(z)(1+\dil{z})\,d\mu_\infty(z) \leq 1 + \frac{\varepsilon}{2} \text,
  \end{equation*}
  since $\mathsf{slen}(1,\ldots,1) = 0$.
  
  The Fourier transform $\widehat{f}$ of $f$ is an element of $\ell^1(\Z^d)$ with finite support, denoted by $S$. Moreover, the range of $\alpha_n^f$ is $\{ f \in \ell^1(\Z^d) : \forall z \notin S \quad f(z) = 0 \}$. In particular, $\alpha_n^f$ has finite rank.

  By \cite[Lemma 3.6]{Latremoliere05}, since both $f$ and $f\mathsf{slen}$ are continuous, we conclude that:
  \begin{equation*}
    \lim_{n\rightarrow\infty} \int_{\U^d_{k_n}} f \, d\mu_n = \int_{\T^d} f \, d\mu_{\infty} = 1 \text{,}
  \end{equation*}
  and
  \begin{equation*}
    \lim_{n\rightarrow\infty} \int_{\U^d_{k_n}} f(z) \mathsf{slen}(z) \, d\mu_n(z) = \int_{\T^d} f(z) \mathsf{slen}(z) \, d\mu_{\infty}(z) \text{.}
  \end{equation*}

  Therefore, there exists $N\in\N$ such that for all $n \in \Nbar$, if $n\geq N$ then:
  \begin{equation*}
     \int_{\U^d_{k_n}} f(z)\dil{z} \, d\mu_n(z) \leq 1 + \varepsilon \text{ and } \int_{\U^d_{k_n}} f(z)\mathsf{slen}(z) \, d\mu_n(z) \leq \frac{\varepsilon}{2} \text.
  \end{equation*}
  
  By Lemma (\ref{near-iso-lemma}), we conclude that
  \begin{equation*}
    \Lip_n\left(\alpha_n^f(a)\right) \leq (1+\varepsilon)\Lip_n(a) \text.
  \end{equation*}

 From Corollary (\ref{mvt-cor}), for all $a\in \dom{\Lip_n}$, we compute:
  \begin{align*}
    \norm{ a - \alpha_n^f(a) }{\A_n}
    &\leq \int_{\U^d_{k_n}} f(z) \norm{a-\alpha_n^z(a)}{\A_n} \, d\mu_n(z) \\
    &\leq \int_{\U^d_{k_n}} f(z) \; 2 \Lip_n(a) \, \mathsf{slen}(z) \, d\mu_n(z) \\
    &= 2 \, \Lip_n(a) \int_{\U^d_{k_n}} f(z) \mathsf{slen}(z) \, d\mu_n(z) \\
    &\leq \varepsilon \Lip_n(a) \text{.}
  \end{align*}
  as desired.
\end{proof}

We can finish the proof that $(\A_n,\mathscr{J}_n,\Dirac_n)$ is a metric spectral triple, for all $n\in\Nbar$.

\begin{theorem}\label{metric-spectral-triple-thm}
  Assume Hypothesis (\ref{L-seminorm-hyp}). For all $n\in\Nbar$, the ordered pair $(\A_n,\Lip_n)$ is a {\qcms}. 
\end{theorem}

\begin{proof}
  Let
  \begin{equation*}
    \tau : a\in \A_n \mapsto \int_{\U^d_{k_n}} \alpha_n^z(a) \, d\mu_n(z)\text.
  \end{equation*}
  Since $\mu_n$ is the probability Haar measure of $\U^d_{k_n}$, it is easy, and well-known, is that $\tau$ is a state (and a trace) of $\A_n$, invariant for the action $\alpha_n$.

We now prove that the set
  \begin{equation*}
    \alg{T} = \left\{ a \in \dom{\Lip_n} : \Lip_n(a) \leq 1 \text{ and }\tau(a) = 0 \right\}
  \end{equation*}
  is totally bounded in $\A_n$ (note that we already know, by Lemma (\ref{three-lemma}), that $\alg{T}$ is closed, and thus complete, since $\A_n$ is complete; thus $\alg{T}$ is totally bounded if, and only if it is compact).

  First, by Corollary (\ref{mvt-cor}), we then note that if $a\in\dom{\Lip_n}$ and $\Lip_n(a)\leq 1$, then
  \begin{equation}\label{lip-bound-eq}
    \norm{a-\tau(a)}{\A_n} \leq \int_{\U_{k_n}^d} \norm{a-\alpha_n^z(a)}{\A_n} \, d\mu_n(z) \leq 2 \max_{z\in\T^d}\mathsf{slen}(z) < 4 d \pi \text.
  \end{equation}
  Therefore, $\alg{T}$ is bounded in $\A_n$.

  Let $\varepsilon > 0$.  By Theorem (\ref{L-approx-thm}), there exists $f \in C(\T^d)$ such that $\alpha_n^f$ has finite dimensional range, and
  \begin{equation*}
    \forall a \in \alg{T} \quad \norm{a-\alpha_n^f(a)}{\A_n} \leq \frac{\varepsilon}{2} \text.
  \end{equation*}

  Now, $\alpha_n^f(\alg{T})$ is a bounded subset, as it is the image, by a continuous linear map, of a bounded set (by Expression (\ref{lip-bound-eq})). Therefore, as a bounded subset of the finite dimensional space $\alpha_n^f(\A_n)$, the set $\alpha_n^f(\alg{T})$ is actually totally bounded. Thus, there exists a $\frac{\varepsilon}{2}$-dense finite subset $F\subseteq \alpha_n^f(\alg{T})$ of $\alpha_n^f(\alg{T})$. Consequently, for all $a\in\alg{T}$, there exists $b\in F$ such that
\begin{equation*}
  \norm{a-b}{\A_n} \leq \norm{a-\alpha_n^f(a)}{\A_n} + \norm{\alpha_n^f(a)-b}{\A_n} < \varepsilon \text.
\end{equation*}
Therefore, since $\varepsilon > 0$ is arbitrary, we conclude that $\alg{T}$ is totally bounded.

 Using Lemma (\ref{three-lemma}) and Corollary (\ref{mvt-cor}), we thus have completed our proof.
\end{proof}

\subsection{Convergence of the Quantum Metrics}

We now prove that our {\qcms s} $(\A_n,\Lip_n)$ converge to $(\A_\infty,\Lip_\infty)$ for the propinquity, when $n$ goes to $\infty$. The method we employ is based upon the work in \cite{Latremoliere13c}, which also involves fuzzy and quantum tori, though for a different family of quantum metrics.

\begin{notation}\label{I-notation}
  Assume Hypothesis (\ref{Dirac-hyp}). Let $n\in\Nbar$ and $j\in\{1,\ldots,d'\}$. For all $x \in \R$, we denote $\max\{ n\in\Z : n \leq x\}$ as $\lfloor x \rfloor$. If $k'_n(j) < \infty$ then we let:
  \begin{equation*}
    C_n^j = \left\{ \left\lfloor \frac{1 - k'_n(j)}{2} \right\rfloor,  \left\lfloor \frac{1 - k'_n(j)}{2} \right\rfloor + 1, \ldots, \left\lfloor \frac{k'_n(j) - 1}{2} \right\rfloor \right\} 
  \end{equation*}
  while, if $k'_n(j) = \infty$, then we let $C_n^j = \Z$. We then define the sets:
  \begin{equation*}
    C_n = \prod_{j=1}^{d'} C_n^j \subseteq \Z^{(d')}  \text.
  \end{equation*}
  
  The canonical surjection $q_n : \Z^{(d')} \twoheadrightarrow \widehat{G_n}$ restricts to an bijection from $C_n$ onto $\widehat{G_n}$, whose inverse, by abuse of notation, we denote by $q_n^{-1} : \widehat{G_n} \hookrightarrow C_n$. Note that $q_{\infty}^{-1}$ is the identity map.
\end{notation}

To ease our notations further, we also adopt the following convention.

\begin{notation}
  If $k\in\Nbar^{(d')}$, if $F\subseteq\Z^{(d')}_k$, and $p\in [1,\infty)$, then we define
  \begin{equation*}
    \ell^p(\Z^{(d')}_{k}|F) = \left\{ \xi \in \ell^p(\Z^{(d')}_k) : \forall m \in \Z^{(d')}_k \quad m \notin F \implies \xi(m) = 0 \right\}\text.
  \end{equation*}
\end{notation}


\begin{notation}
  If $n\in\Nbar$, and if $\xi \in \Hilbert_n$, then we define
  \begin{equation*}
    \rho_n\xi : m \in \Z^{(d')} \mapsto
    \begin{cases}
      \xi(q_n(m)) \text{ if $m \in C_n$,} \\
      0 \text{ otherwise.}
    \end{cases}
  \end{equation*}
  and note that $\rho_n$ is an isometry from $\Hilbert_n$ into $\Hilbert_\infty$, such that
  \begin{equation*}
    \forall \xi \in \Hilbert_\infty \quad \rho_n^\ast \xi : m\in \widehat{G_n} \mapsto \xi(q_n^{-1}(m)) \text.
  \end{equation*}
\end{notation}

The following lemma establishes the asymptotic behavior of the operators $\Gamma_{n,j}$, for all $j\in\{1,\ldots,d+d'\}$, as $n\rightarrow\infty$; this result is the basis of the idea that our sequence of spectral triples in Hypothesis (\ref{Dirac-hyp}) converges to the expected limit.

\begin{lemma}\label{Gamma-cv-lemma}
  Assume Hypothesis (\ref{Dirac-hyp}). Let $F\subseteq\Z^{(d')}$ be finite. For all $j \in \{1,\ldots,d+d'\}$, the following limit holds:
  \begin{equation*}
    \lim_{n\rightarrow\infty} \opnorm{\rho_n\Gamma_{n,j}\rho_n^\ast - \Gamma_{\infty,j}}{\Hilbert_\infty}{\ell^2(\Z^{(d')}|F)} = 0 \text.
  \end{equation*}
\end{lemma}

\begin{proof}
  For this proof, we recall that, for each $j\in\{1,\ldots,d'\}$, by Hypothesis (\ref{innerification-hyp}), either $k'_n(j)$ is finite for all $n\in\N$, or $k'_n(j) = \infty$ for all $n\in\N$. Either way, of course, $\lim_{n\rightarrow\infty} k'_n(j) = \infty$.
  
  We record, of course, that for all $m \in \Z$, and for all $j \in \{1,\ldots,d'\}$ such that $k'_n(j) < \infty$ for all $n\in\Nbar$,
  \begin{equation*}
    \frac{ \exp\left(\frac{2\pi i m}{k'_n(j)}\right) - 1}{\frac{2\pi}{k'_n(j)}} \xrightarrow{n\rightarrow\infty} i m \text{.}
  \end{equation*}
  
  Let $\varepsilon > 0$. Since $F$ is finite, and therefore, the set
  \begin{equation*}
    S = \left\{ m \in \Z^{(d')} : m \in F \text{ or }\exists j \in \{1,\ldots,d'\} \quad m + e_j \in F\text{ or } m - e_j \in F \right\}
  \end{equation*}
  is finite, there exists $K \in \N$ such that, if $n \geq K$, then
  \begin{enumerate}
  \item $S \subseteq C_n$,
  \item the following holds,
    \begin{equation*}
      \max_{\substack{j \in \{1,\ldots,d'\} \\ k'_n(j) < \infty}} \; \max_{m = (m_1,\ldots,m_{d'})\in S} \left| \frac{\exp\left(\frac{2\pi i m}{k'_n(j)}\right) - 1}{\frac{2\pi}{k'_n(j)}} - i m_j \right| \leq \frac{\varepsilon}{2} \text{,}
    \end{equation*}
  \item  since $(\sigma'_n)$ converges to $(\sigma_\infty)$ in $\mathcal{C}_{\infty^{d'}}^{d'}$ by Hypothesis (\ref{innerification-hyp}),
    \begin{equation*}
      \max_{m,m'\in S}\left|\sigma_n(q_n(m),q_n(m'))-\sigma'_\infty(m,m')\right| < \frac{\varepsilon}{2} \text.
    \end{equation*}
  \end{enumerate}
  
  Let now $n\geq K$. First, let $\xi\in\Hilbert_\infty$.  We then note that, for all $m\in\widehat{G}_n$,
  \begin{align*}
    \rho_n v_n^{z_{n,j}} \rho_n^\ast \xi(m)
    &= \begin{cases}
      (z_{n,j}^{\left( \cdot \right)} \xi(q_n^{-1}(\cdot)))(q_n(m)) \text{ if $m\in C_n$,} \\
      0 \text{ otherwise;}
    \end{cases} \\
    &= \begin{cases}
      z_{n,j}^m \xi(m) \text{ if $m \in C_n$,} \\
      0 \text{ otherwise;}
    \end{cases} \\
    &= v_{\infty}^{z_{n,j}} \rho_n\rho_n^\ast \xi(m) \text,
  \end{align*}
  where we used that $z_{n,j}^{q_n(m)} = z_{n,j}^m$ by definition of this notation. Therefore,
  \begin{equation*}
    \rho_n v_n^{z_{n,j}} \rho_n^\ast \xi = v_\infty^{z_{n,j}} \rho_n\rho_n^\ast \xi \text.
  \end{equation*}
  Thus, if $\xi \in \ell^2(\Z^{d'}|F)$, then $\rho_n v_n^{z_{n,j}} \rho_n^\ast \xi = v_{\infty}^{z_{n,j}} \xi$, since $F\subseteq S\subseteq C_n$, we conclude that $q_n^{-1}\circ q_n$, restricted to $S$, is the identity function --- and therefore, $\rho_n\rho_n^\ast$ is the orthogonal projection of $\Hilbert_\infty$ on $\ell^2(\Z^{(d')}|C_n)$.

  Now, fix $\xi \in \ell^2(\Z^{(d')}|F)$ with $\norm{\xi}{\Hilbert_\infty} \leq 1$. We compute that, for all $j\in\{1,\ldots,d'\}$ with $k'_n(j) < \infty$,
  \begin{align}\label{Gamma-cv-eq1}
    & \norm{\frac{k'_n(j)}{2 \pi} \rho_n \left(v_n^{z_{n,j}}-\unit_n\right) \rho_n^\ast\xi - \partial_{\infty,j}\xi}{\Hilbert_\infty}^2  \\
    &= \norm{\frac{k'_n(j)}{2\pi}(v_\infty^{z_{n,j}} - 1)\xi - \partial_{\infty,j} \xi}{\Hilbert_\infty}^2 \nonumber \\
    &= \sum_{m=(m_1,\ldots,m_{d'})\in F} |\xi(m)|^2  \left| \frac{\exp\left(\frac{2\pi i m}{k'_n(j)}\right) - 1}{\frac{2\pi}{k'_n(j)}} - i m_j \right|^2 \nonumber \\
    &\leq \max_{\substack{j \in \{1,\ldots,d' \} \\ k'_n(j) < \infty}}\max_{m=(m_1,\ldots,m_{d'}) \in F} \left| \frac{\exp\left(\frac{2\pi i m}{k'_n(j)}\right) - 1}{\frac{2\pi}{k'_n(j)}} - i m_j \right|^2 \sum_{m\in\widehat{G_n}} |\xi(m)|^2 \nonumber \\
    &\leq \max_{\substack{j \in \{1,\ldots,d' \} \\ k'_n(j) < \infty}} \;  \max_{m=(m_1,\ldots,m_{d'}) \in S}  \left| \frac{\exp\left(\frac{2\pi i m}{k'_n(j)}\right) - 1}{\frac{2\pi}{k'_n(j)}} - i m_j \right|^2 \norm{\xi}{\Hilbert_\infty}^2 \nonumber \\
    &\leq \left(\frac{\varepsilon}{2}\right)^2 \text{.} \nonumber
  \end{align}
  
  Thus, in particular, if $j\in\{d+1,\ldots,d'\}$ and $k'_n(j) < \infty$, and if $\xi \in \ell^2(\Z^{d'}|F)$, then
  \begin{equation*}
    \lim_{n\rightarrow\infty} \norm{\rho_n\Gamma_{n,d+j}\rho_n^\ast \xi - \Gamma_{\infty,d+j} \xi}{\Hilbert_\infty} = 0 \text.
  \end{equation*}
  
  Now, fix $j\in\{1,\ldots,d\}$ with $k_n(j) < \infty$. We now recall that if $\xi\in\Hilbert_n$ and $j\in\{1,\ldots,d'\}$, then
  \begin{equation}\label{Un-eq}
    U_{n,j} \xi : m \in \widehat{G_n} \mapsto \sigma'_n(e_{j},m-e_j) \xi\left(m - e_j\right)\text.
  \end{equation}
  Moreover, by Equation (\ref{right-action-eq}),
  \begin{equation*}
    \xi\cdot U_{n,j} : m \in \widehat{G_n} \mapsto \sigma'_n(m-e_j,e_j)\xi(m-e_j) \text.
  \end{equation*}
  
  In particular, if $\xi\in\ell^2(\Z^{(d')}|F)$, then $\xi\cdot U_{n,j} (m) = 0$ whenever $m \in \Z^{(d')}\setminus S$.

  By Hypothesis (\ref{innerification-hyp}), $U_{\infty,f(j)}$ is central in $\B_\infty$. By construction, if $\xi \in \ell^1(\Z^{d'}) \subseteq \Hilbert_n$, then
  \begin{equation*}
    U_{\infty,f(j)}\xi = \delta_{f(j)} \conv{\infty^d,\sigma'_\infty} \xi = \xi \conv{\infty^d,\sigma'_\infty} \delta_{f(j)} = \xi\cdot U_{\infty,f(j)}  \text.
  \end{equation*}
  Since $\ell^1(\Z^{d'})$ is dense in $\Hilbert_n$, we conclude that $[U_{\infty,f(j)},\cdot] = 0$ as an operator on $\Hilbert_n$.

    Note that the right action $\cdot$ of $\B_n$ on $\Hilbert_n$ obviously depends on $n$, but the context will make it clear which right action we are using at all time.

  Let $\xi\in\ell^2\left(\Z^{d'}|F\right)$ with $\norm{\xi}{\Hilbert_\infty} \leq 1$. A simple computation then shows that, if $\xi \in \ell^2(\widehat{G_n}|F)$, and if $\norm{\xi}{\Hilbert_\infty}\leq 1$, and if we let
  \begin{equation*}
    \xi_j : m \in \widehat{G_n} \mapsto \xi\left( m - e_{j} \right) \text,
  \end{equation*}
  then
  \begin{multline}\label{Un-Uinfty-eq}
    \norm{\rho_n( \rho_n^\ast\xi \cdot U_{n,f(j)}) - \xi\cdot U_{\infty,f(j)}}{\Hilbert_\infty}^2 \\
    \begin{aligned}
      &= \sum_{m\in C_n} \left|(\xi\circ q_n^{-1} \cdot U_{n,f(j)})\circ q_n(m) - \xi \cdot U_{\infty,f(j)}(m) \right|^2 \\
      &= \sum_{m\in S} \left| \sigma'_n(q_n(m) - e_{f(j)},e_{f(j)}) \xi_{f(j)}(m) - \sigma'_\infty(m-e_{f(j)},e_{f(j)})\xi_{f(j)}(m) \right|^2 \\
      &\leq \left( \frac{\varepsilon}{2} \right)^2 \text.
  \end{aligned}
\end{multline}

With this in mind, for all $n\geq K$, and for all $j\in\{1,\ldots,d\}$, if $k_n(j) < \infty$ and $f(j) > j$, and if $\xi \in \ell^2(\Z^{d'}|F)$ with $\norm{\xi}{\Hilbert_\infty} \leq 1$, then we conclude that, since $\rho_n^\ast\rho_n=1$, and, crucially, since Hypothesis (\ref{innerification-hyp}) implies:
\begin{equation*}
\forall\eta\in\Hilbert_n \quad U_{n,f(j)} \eta \cdot U_{n,f(j)}^\ast - \eta = (v_n^{z_{n,j}}-1)\eta\text,
\end{equation*}
the following computation holds:
  \begin{align*}
    & \norm{\frac{k_n(j)}{2\pi} \rho_n[U_{n,f(j)},\rho_n^\ast \xi] - U_{\infty,f(j)}\partial_{\infty,j}\xi}{\Hilbert_\infty} \\
    &\leq \norm{\left(\frac{k_n(j)}{2\pi} \rho_n\left( U_{n,f(j)} \rho_n^\ast\xi \cdot U_{n,f(j)}^\ast - \rho_n^\ast\xi \right)\cdot U_{n,j}\right) - \left(\partial_{\infty,j}  \xi\right)\cdot U_{\infty,f(j)}}{\Hilbert_\infty} \\
    &= \norm{\frac{k_n(j)}{2\pi}\rho_n\Big((v_n^{z_{n,j}}-1)\rho_n^\ast\xi\cdot U_{n,f(j)}\Big) - \left(\partial_{\infty,j}  \xi\right)\cdot U_{\infty,f(j)}}{\Hilbert_\infty} \\
    &= \norm{\frac{k_n(j)}{2\pi}\rho_n\Big(\rho_n^\ast \rho_n(v_n^{z_{n,j}}-1)\rho_n^\ast\xi\cdot U_{n,f(j)}\Big) - \left(\partial_{\infty,j}  \xi\right)\cdot U_{\infty,f(j)}}{\Hilbert_\infty} \\
    &= \norm{\frac{k_n(j)}{2\pi}\rho_n\Big(\rho_n^\ast (\underbracket[1pt]{\rho_n(v_n^{z_{n,j}}-1)\rho_n^\ast\xi}_{\coloneqq \eta})\cdot U_{n,f(j)}\Big) - \left(\partial_{\infty,j}  \xi\right)\cdot U_{\infty,f(j)}}{\Hilbert_\infty} \\
    &\leq \norm{\frac{k_n(j)}{2\pi}\underbracket[1pt]{\rho_n\Big( (v_n^{z_{n,j}}-1)\rho_n^\ast\xi \Big)}_{=\eta} \cdot U_{\infty,f(j)} - (\partial_{\infty,j}\xi) \cdot U_{\infty,f(j)}}{\Hilbert_\infty} + \frac{\varepsilon}{2} \text{ by Eq. (\ref{Un-Uinfty-eq})}\\
    &\leq \norm{\rho_n\Big(\frac{k_n(j)}{2\pi}(v_n^{z_{n,j}}-1)\rho_n^\ast\xi \Big) - (\partial_{\infty,j}\xi)}{\Hilbert_\infty} + \frac{\varepsilon}{2} \\
    &\leq \varepsilon \text{ by Eq. (\ref{Gamma-cv-eq1}).}
  \end{align*}

  Similarly, we can prove with the same method as for the proof of Expression (\ref{Un-Uinfty-eq}) that for all $n\in\Nbar$, $j \in \{1,\ldots,d\}$ such that $k_n(j) < \infty$, and for all $\xi \in \Hilbert_\infty$,
  \begin{equation*}
    \norm{\rho_n(U_{n,f(j)}^\ast \rho_n^\ast \xi) - U_{\infty,j}^\ast \xi}{\Hilbert_\infty} \leq \frac{\varepsilon}{2} \text;
  \end{equation*}
  and therefore, if $f(j) > j$, then
  \begin{equation*}
    \norm{\frac{k_n(j)}{2\pi} \rho_n[U_{n,j}^\ast,\rho_n^\ast\xi] + U_{\infty,j}^\ast\partial_{\infty,j}\xi}{\Hilbert_\infty} \leq \varepsilon \text,
  \end{equation*}
  using the observation that
  \begin{equation*}
    \left[U_{n,f(j)}^\ast,\rho_n^\ast\xi\right] = U_{n,f(j)}^\ast\left((\unit_n - v_n^{z_{n,j}})\rho_n^\ast\xi\right) \text.
  \end{equation*}

  Therefore, for all $n\geq K$, if $j \in\{1,\ldots,d\}$, if $k_n(j)<\infty$, and if $f(j) > j$, then we conclude that
  \begin{align*}
       & \norm{(\rho_n \Gamma_{n,j}\rho_n^\ast - \Gamma_{\infty,j})\xi}{\Hilbert_\infty} \\
    &= \norm{\frac{-k_n(j)}{2i\pi} \rho_n\left[Y_{n,j},\rho_n^\ast\xi\right] - X_{\infty,j}\partial_{\infty,j}\xi}{\Hilbert_\infty} \\
    &\leq \frac{1}{2}\bigg(\norm{\frac{k_n(j)}{2\pi} \rho_n\left[U_{n,j},\rho_n^\ast\xi\right] - U_{\infty,j}\partial_{\infty,j} \xi}{\Hilbert_\infty} \\
    &\quad + \norm{\frac{k_n(j)}{2\pi} \rho_n\left[U^\ast_{n,j},\rho_n^\ast\xi\right] + U_{\infty,j}^\ast\partial_{\infty,j} \xi}{\Hilbert_\infty} \bigg) \leq \varepsilon \text,
  \end{align*}
  and similarly,
  \begin{align*}
    \norm{(\rho_n \Gamma_{n,d+j} \rho_n^\ast - \Gamma_{\infty,d+j})\xi}{\Hilbert_\infty}
    &= \norm{\frac{k_n(j)}{2 i \pi}\rho_n[X_{n,j},\rho_n^\ast\xi] - Y_{\infty,j}\partial_{\infty,j}\xi}{\Hilbert_\infty} \\
    &\leq \varepsilon \text.
  \end{align*}

  The only difference when working with $n\in\N$, $j \in \{1,\ldots,d\}$, and $f(j) < j$, if that the roles of $U_{n,j}$ and $U_{n,j}^\ast$ are reversed in the above computation, so we obtain the same inequalities.

  Thus, we conclude that for all $j\in\{1,\ldots,d+d'\}$ with $k_n(j) < \infty$:
  \begin{equation*}
    \norm{(\rho_n \Gamma_{n,j} \rho_n^\ast - \Gamma_{\infty,j})\xi}{\Hilbert_\infty} \leq \varepsilon \text.
  \end{equation*}
  
  Last, if $j \in \{1,\ldots,d\}$ and $k'_n(j) = \infty$, then an easy computation shows, using similar method to the ones used for Expression (\ref{Un-Uinfty-eq}) once more, that, for all $s\in\{0,d\}$:
  \begin{equation*}
    \norm{(\rho_n \Gamma_{n,j+s} \rho_n^\ast - \Gamma_{\infty,j+s})\xi}{\Hilbert_\infty} \leq \varepsilon  \text.
  \end{equation*}
  
  This concludes our proof.
\end{proof}

A first major consequence of Lemma (\ref{Gamma-cv-lemma}) is a form of continuity for the L-seminorms from Hypothesis (\ref{L-seminorm-hyp}). Quantum and fuzzy tori are the fibers of a continuous field of C*-algebras, with the C*-algebra of continuous sections given by a twisted groupoid C*-algebra, where the groupoid is a bundle over the base space $\Xi^d$, with fiber $\Z^d_k$ over each point $(k,\sigma)\in\Xi^d$, as explained in \cite{Latremoliere05}. 

\begin{theorem}[{\cite{Latremoliere05}}]\label{continuous-clifford-field-thm}
  Assume Hypothesis (\ref{L-seminorm-hyp}). If $F\subseteq \Z^d$ be a nonempty finite subset of $\Z^d$, and if $a \in \ell^1(\Z^d|F)$, then
  \begin{equation*}
    \lim_{n\rightarrow\infty} \Lip_n(\pi_n(a\circ q_n^{-1})) = \Lip_\infty(\pi_\infty(a))\text.
  \end{equation*}
\end{theorem}

\begin{proof}
  In this proof, as seen in Example (\ref{Clifford-ex}), the C*-algebra $\alg{Cl}(\C^{d+d'})$ is a fuzzy torus $\qtd{d+d'}{(2,\ldots,2)}{\varsigma}$, for the normalized $2$-cocycle defined in that example. As $\alg{Cl}(\C^{d+d'})$ is finite dimensional, we note that, \emph{as vector spaces}, $\alg{Cl}(\C^{d+d'}) = \ell^1(\Z^{d+d'}_{2,\ldots,2})$ --- the norms are different, of course.
  
  We also identify $\B_n$ with the $C^\ast$ completion of $\ell^1(\widehat{G_n})$ via the *-isomorphism $\pi_n$, without further mention, to keep notations simpler.

  \medskip
  
  Let $\mathscr{G} = \coprod_{n\in\Nbar} \left(\widehat{G_n}\times \Z^{d+d'}_{(2,\ldots,2)}\right)$ be the disjoint union of the spaces $\widehat{G_n}\times\Z^{d+d'}_{(2,\ldots,2)}$ --- i.e. $x\in\mathscr{G}$ if and only if $x=(n,m,l)$ for $n\in\Nbar$, $m\in \widehat{G_n}$, and $l\in \Z^{d+d'}_{(2,\ldots,2)}$.
  
  The function
  \begin{equation*}
    (n,m,l) \in \mathscr{G} \rightarrow (n,q^{-1}(m),l) \in \Nbar\times\Z^{d'}\times \Z^{d+d'}_{(2,\ldots,2)}
  \end{equation*}
  is trivially an injection, with range $\{ (n,m,l) : n\in\Nbar,m\in C_n,l\in\Z^{d+d'}_{(2,\ldots,2)}\}$; we endow $\mathscr{G}$ with the unique topology such that this injective map is an homeomorphism onto its image, where $\Nbar\times\Z^{d'}\times\Z^{d+d'}_{(2,\ldots,2)}$ is endowed with the product topology. The topology on $\mathscr{G}$ is thus metrizable, and a sequence $(n_p,m_p,l_p)_{p\in\N}$ in $\mathscr{G}$ converges if, and only if, it is eventually constant, or $\lim_{p\rightarrow\infty} n_p = \infty$ and $(m_p,l_p)_{p\in\N}$ is eventually constant.

  \medskip
  
  The set $\mathscr{G}$ is trivially a groupoid, where the source and target maps are both given by
  \begin{equation*}
    s : (n,m,l) \in \mathscr{G} \mapsto n \in \Nbar
  \end{equation*}
  (and thus, the space of $\mathscr{G}^{(0)}$ is $\Nbar$), the partial multiplication is defined on the space
  \begin{equation*}
    \mathscr{G}^{(2)} = \left\{ ((n,m,l),(n',m',l')) \in \mathscr{G} : n=n' \right\}
  \end{equation*}
  by setting:
  \begin{equation*}
    \forall ((n,m,l),(n,m',l')) \in \widehat{G}^{(2)} \quad (n,m,l)(n,m',l') = (n,m + m', l + l') \text,
  \end{equation*}
  and the inverse of $(n,m,l) \in \widehat{G}$ is $(n,-m,-l)$. If, moreover, for all $n\in\Nbar$, we let $\lambda_n$ be the counting measure on the discrete group $\widehat{G}_n\times\Z^{d+d'}_{(2,\ldots,2)}$, then $\mathscr{G}$ is easily checked to be a locally compact groupoid with (left) Haar measure $(\lambda_n)_{n\in\Nbar}$.

  For any $(n,m,l), (n,m',l') \in \mathscr{G}^{(2)}$, we then set:
  \begin{equation*}
    \beta((n,m,l),(n,m',l')) = \sigma'_{n}(m,m')\varsigma(l,l') \text.
  \end{equation*}
  Thus defined, $\beta$ is a continuous $2$-cocycle of $\mathscr{G}$ ((see, e.g.,  \cite[Sec. 2.1.]{Latremoliere05} and of course, \cite{Renault80})).

  For each $n\in\Nbar$, we record that $s^{-1}(\{n\}) = \{n\}\times\widehat{G_n}\times\Z^{d+d'}_{(2,\ldots,2)}$, and that $\beta$, restricted to $s^{-1}(\{n\})$, is 
  \begin{equation*}
    \beta_n : (n,m,l),(n,m,'l')\in\widehat{G_n}\times \Z^{d+d'}_{(2,\ldots,2)} = \sigma'_n(m,m') \varsigma(l,l') \text,
  \end{equation*}
  so that
  \begin{equation*}
    \B_n\otimes \alg{Cl}(\C^{d+d'}) \text{ is *-isomorphic to } C^\ast(\widehat{G_n}\times \Z^{d+d'}_{2,\ldots,2},\beta_n) \text,
  \end{equation*}
  where our chosen *-isomorphism sends any element of the form $a\otimes b \in \ell^1(\widehat{G_n})\otimes \alg{Cl}(\C^{d+d'})$ to the function $(m,l) \in \widehat{G_n}\times \Z^{d+d'}_{(2,\ldots,2)} \mapsto a(m) b(l)$ (the fact that this map extends to a *-isomorphism follows form a standard argument based on the universality of the C*-algebras involved, which are all fuzzy tori).

  \medskip
  
  We therefore conclude by \cite[Theorem 2.6]{Latremoliere05} that $C^\ast(\mathscr{G},\beta)$ is the C*-algebra of continuous sections for the family $(\B_n\otimes\alg{Cl}(\C^{d+d'}))_{n\in\Nbar}$ of C*-algebras.
  
  In particular, if $g : \mathscr{G} \rightarrow\C$ is continuous and compactly supported on $\mathscr{G}$, and if, for all $n\in\Nbar$, we define $g^n$ by
  \begin{equation*}
    g^n : (m,l) \in \widehat{G_n}\times\Z^{d+d'}_{(2,\ldots,2)} \mapsto g(n,m,l)\text,
  \end{equation*}
  then $\lim_{n\rightarrow\infty} \norm{g^n}{\B_n\otimes \alg{Cl}(\C^{d+d'})} = \norm{g^\infty}{\B_\infty\otimes\alg{Cl}(\C^{d+d'})}$.
  
  \medskip

  Let now $a \in \ell^1(\Z^{(d')}|F)$, and let $N \in \N$ be chosen such that, for all $n\in \Nbar$ with $n\geq N$, we have $F\subseteq C_n$. We define the function:
  \begin{equation*}
    g : (n,m,l) \in \mathscr{G} \mapsto
    \begin{cases}
      \grad{n}{\left(a\circ q_n^{-1}\right)}(m,l) \text{ if $n\geq N$,}\\
      0 \text{ otherwise,}
    \end{cases}
  \end{equation*}
  using the notation defined  in Expression (\ref{gradiant-eq}).

  It is immediate to check that $g$ is compactly supported on $\mathscr{G}$. Let now $(n_p,m_p,l)_{p\in\N}$ be a sequence in $\mathscr{G}$ such that $\lim_{p\rightarrow\infty} n_p = \infty$, and $q_{n_p}(m_p) = m$ for all $p \in \N$. We aim at showing that $(g(n_p,m_p,l))_{p\in\N}$ converges to $g(\infty,m,l)$.

  We make the simple observation that, for all $n\in\Nbar$ with $n\geq N$, since $a\circ q_n^{-1}$ is finitely supported, it is an element of $\Hilbert_n$, and with this observation,
  \begin{equation*}
    \grad{n}{a} = \sum_{j=1}^{2d} \Gamma_{n,j}(a\circ q_n^{-1})\otimes \gamma_j \text.
  \end{equation*}

  Using Expression (\ref{gradiant-eq}), we then compute, for all $p\in\N$:
  \begin{align*}
    \big| \grad{n_p}{a\circ q_{n_p}^{-1}}(m_p,l) &- \grad{\infty}{a}(m,l) \big| \\
    &=\left| \grad{n_p}{\left(a\circ q_{n_p}^{-1}\right)}(q_{n_p}(m),l) - \grad{\infty}{a}(m,l) \right| \\
    &= \left| \sum_{j=1}^{2d} \left( \Gamma_{n_p,j}(a\circ q_{n_p}^{-1})(q_{n_p}(m)) - \Gamma_{\infty,j}(a)(m)\right) \gamma_j(l) \right| \\
    &\leq \sum_{j=1}^{2d} \left|\left(\Gamma_{n_p,j}(a\circ q_{n_p}^{-1})(q_{n_p}(m)) - \Gamma_{\infty,j}(a)(m)\right)\right| |\gamma_j(l)| \\
    &\leq \sum_{j=1}^{2d} \norm{\rho_{n_p} \Gamma_{n_p,j} \rho_{n_p}^\ast a - \Gamma_{\infty} a}{\Hilbert_\infty}\cdot 1 \text{ by Eq. \ref{Gamma-cv-lemma} } \\
    &\xrightarrow{p\rightarrow\infty} 0 \text.
  \end{align*}

  Therefore, the function $g$ is continuous on $\mathscr{G}$.

  Consequently, using Theorem (\ref{spectral-triple-thm}),
  \begin{multline*}
    \Lip_\infty(a) = \norm{\grad{\infty}{a}}{\B_\infty\otimes\alg{Cl}(\C^{d+d'})} = \lim_{n\rightarrow\infty} \norm{\grad{n}{\left(a \circ q_n^{-1}\right)}}{\B_n\otimes\alg{Cl}(\C^{d+d'})} \\ = \lim_{n\rightarrow\infty} \Lip_n(a\circ q_n^{-1}) \text.
  \end{multline*}
  This concludes our proof.
\end{proof}

Now, Theorem (\ref{L-approx-thm}), Theorem (\ref{metric-spectral-triple-thm}), and Theorem (\ref{continuous-clifford-field-thm}) are all we need to apply the methods of \cite{Latremoliere13b} to conclude:

\begin{theorem}\label{qcms-thm}
  If we assume Hypothesis (\ref{L-seminorm-hyp}), then
  \begin{equation*}
    \lim_{n\rightarrow\infty} \dpropinquity{}\left( \left( \A_n, \Lip_{n} \right), \left( \A_\infty, \Lip_{\infty} \right) \right) = 0 \text{.}
  \end{equation*}
\end{theorem}

The proof of Theorem (\ref{qcms-thm}) is essentially the same as the proof of \cite[Theorem 5.2.5]{Latremoliere13}, once we replace some intermediate results in \cite{Latremoliere13c} with results proven here.

\begin{proof}[{Proof of Theorem \ref{qcms-thm}}]
  For this proof, we will implicitly use the *-isomorphism $\pi_n$ from the completion of $(\ell^1(\Z^d_{k_n}),\conv{k_n,\sigma_n},\cdot^\ast)$ for its C*-norm, onto $\A_n$. Let $S\subseteq\Z^{d}$ be finite, and let $V = \ell^1(\Z^{d}|S)\cap\sa{\A_n}$. Set $F = S\times\{(0,\ldots,0)\} \subseteq\Z^{(d')}$, and let $N\in\N$ be chosen so that $n\geq N\implies F \subseteq C_n$. We write $\Nbar_N = \{ n \in \Nbar: n\geq N\}$.

  If $f\in V$, then $n\in\Nbar\mapsto \Lip_n(f\circ q_n^{-1})$ is continuous, by Theorem (\ref{continuous-clifford-field-thm}), and since $\Nbar$ is compact, we conclude that $\sup_{n\in\Nbar} \Lip_n(f \circ q_n^{-1}) < \infty$. We thus define the seminorm $\mathrm{m} : f \in V \mapsto \sup_{n\in\Nbar} \Lip_n(f \circ q_n^{-1})$. Therefore, by \cite[Lemma 4.2.5]{Latremoliere13b}, and by Theorem (\ref{continuous-clifford-field-thm}), the map
  \begin{equation*}
    (n,f) \in \Nbar_N\times V \mapsto \Lip_n(f\circ q_n^{-1})
  \end{equation*}
  is continuous (since $V$ is finite dimensional, it is not important to specify the norm on  $V$, though of course, it is natural to choose $\mathrm{m}$). This can now be used in place of the conclusion of \cite[Theorem 4.2.7]{Latremoliere05}, together with Theorem (\ref{L-approx-thm}), in \cite[Section 5]{Latremoliere13b}, to conclude our proof. We record the following main tools used in \cite[Theorem 5.2.5]{Latremoliere13b}.

  Let $\varepsilon > 0$. There exist a finite subset $S\subseteq \Z^d$, a natural number $N\in\N$ and a finite rank operator $T_N$ on $\ell^2(\Z^d)$ with $\opnorm{T_N}{}{\ell^2(\Z^d)} = 1$, such that, for all $n\geq N$, there exists a *-representation $\vartheta_n$ of $\A_n$ on $\ell^2(\Z^d)$, with the following property: if we set, for all $(a,b) \in \sa{\A_n\oplus\A_\infty}$, 
  \begin{itemize}
  \item $\TLip_n(a,b) = \max\left\{ \Lip_n(a), \Lip_\infty(b), \frac{1}{\varepsilon} \opnorm{\vartheta_n(a) T_N - T_N \vartheta_\infty(b)}{}{\ell^2(\Z^d)} \right\} \text{,}$
  \item $y_n(a,b) = a \text{ and }y_{\infty}(a,b) = b \text{,}$
  \end{itemize}
  then:
  \begin{equation*}
    \tau_n = \left(\A_n\oplus\A_\infty, \TLip_{n}, y_n, y_{\infty} \right)
  \end{equation*}
  is a tunnel from $(\A_n,\Lip_n)$ to $(\A_\infty,\Lip_\infty)$ whose extent is at most $\varepsilon$. This concludes our proof, but we record two more details about the structure of our tunnels.

  First, if $a \in \ell^1(\Z^d|S)$, then:
  \begin{equation*}
    \opnorm{[T_n,\vartheta_n(a)]}{}{\ell^2(\Z^d)} \leq \varepsilon \Lip_n(a) \text{.}
  \end{equation*}

  This completes the summary of the properties of the tunnels $\tau_n$ constructed in \cite[Theorem 5.2.5]{Latremoliere13b}.
\end{proof}

We now adjust the presentation of the tunnels constructed in the proof of Theorem (\ref{qcms-thm}), in preparation for our work with the Dirac operators. To this end, we use the same notation as in the proof of Theorem (\ref{qcms-thm}).

\begin{notation}\label{new-tunnel-notation}
  The Hilbert space $\mathscr{J}_\infty=\ell^2(\Z^{d'},\mathscr{C})$ is isometrically isomorphic to the Hilbert space $\ell^2(\Z^d)\otimes\ell^2(\Z^{d'-d},\mathscr{C})$, via the usual unitary, which extends the function which sends $\xi\otimes\eta\in \ell^2(\Z^d)\otimes\ell^2(\Z^{d'-d},\mathscr{C})$ to
  \begin{equation*}
    (m_1,\ldots,m_{d'}) \in \Z^{(d')} \longmapsto \xi(m_1,\ldots,m_d) \eta(m_{d+1},\ldots,m_{d'}) \text,
  \end{equation*}
  with the convention that a function of $0$ variables (if $d'=d$ above) is a constant.

  With this identification, let $\theta_n(a) = \vartheta_n(a) \otimes \unit_{\ell^2(\Z^{d'-d},\mathscr{C})}$, where $\unit_{\ell^2(\Z^{d'-d},\mathscr{C})}$ is the identity of $\ell^2(\Z^{d'-d},\mathscr{C})$.

  In particular, by \cite{Latremoliere13c}, $\theta_\infty(a) = a^\circ$ for all $a\in\A_\infty$.  Again, using this identification, we then note that if we set $R_N = T_N \otimes \unit_{\ell^2(\Z^{d'-d},\mathscr{C})}$, then, for all $a\in\A_n$ and $b \in \A_\infty$:
  \begin{equation}\label{qt-tunnel-eq}
    \opnorm{\theta_n(a) R_N - R_N \theta_\infty(b)}{}{\mathscr{J}_\infty} = \opnorm{\vartheta_{n}(a) T_N - T_N \vartheta_{\infty}(b)}{}{\ell^2(\Z^d)} \text.
  \end{equation}
  
  Therefore, the L-seminorm $\TLip_n$ of the tunnels $\tau_n$ of the proof of Theorem (\ref{qcms-thm}) can be rewritten, for all $a\in\sa{\A_n}$ and $b \in \sa{\A_\infty}$, as
  \begin{equation*}
    \TLip_n(a,b) = \max\left\{\Lip_n(a),\Lip_\infty(b),\frac{1}{\varepsilon} \opnorm{\theta_n(a) R_N - R_N\theta_\infty(b)}{}{\mathscr{J}_\infty} \right\} \text.
  \end{equation*}

\end{notation}

We now record two more properties proven in \cite[Theorem 5.2]{Latremoliere13c}, adjusted to our current notation.  First, we also note that, for all $n\geq N$, if $a\in\sa{\A_n}$ is finitely supported, then
  \begin{equation}\label{tunnel-commutator-eq}
    \opnorm{\left[a^\circ ,R_N\right]}{}{\mathscr{J}_n} = \opnorm{\left[a,T_N\right]}{}{\ell^2(\Z^d)} \leq \varepsilon \Lip_n(a) \text.
  \end{equation}
  
Second, we also note that, for any finite subset $F\subseteq\Z^{(d')}$, we can always choose $R_N$ such that $R_N$ restricted to $\{\xi \in \mathscr{J}_\infty : \forall n\in\Z^{d'}\setminus F \quad \xi(n) = 0 \}$ is the identity.

\medskip

The advantage of this new formulation is that, now, our tunnels $\tau_n$ are built using the Hilbert space $\mathscr{J}_n$, on which our spectral triples are also constructed. We are now ready to conclude that our sequence of metric spectral triples from Hypothesis (\ref{Dirac-hyp}) converge.

\subsection{Convergence of the Spectral Triples}

We now conclude with our main result about the convergence of the spectral triples of Hypothesis (\ref{Dirac-hyp}), where we use the spectral propinquity, as explained in the introduction.

\begin{notation}\label{mcc-notation}
  Assume Hypothesis (\ref{L-seminorm-hyp}). For all $n \in \Nbar$ and for all $\xi \in \dom{\Dirac_n}$, we set:
  \begin{equation*}
    \CDN_n(\xi) = \norm{\xi}{\mathscr{J}_n} + \norm{ \Dirac_n \xi }{\mathscr{J}_n}\text{.}
  \end{equation*}

  We thus have:
  \begin{equation*}
    \mcc{\A_n}{\mathscr{J}_n}{\Dirac_n} = \left( \mathscr{J}_n, \CDN_n, \C, 0, \A_n, \Lip_n \right)
  \end{equation*}
  where we regard $\mathscr{J}_n$ as a module over $\A_n$ via the *-representation $a\in\A_n\mapsto a^\circ$. By \cite{Latremoliere18g}, $\mcc{\A_n}{\mathscr{J}_n}{\Dirac_n}$ is a metrical $C^\ast$-correspondence.
\end{notation}

We construct a modular tunnel between the {\gQVB s} given by our spectral triples, namely the quadruples of the form $(\mathscr{J}_n,\CDN_n,\C,0)$. We note that $(\C,0)$ is a trivial {\qcms} (the space with one point), but we will actually need to work with Hilbert modules over $\C\oplus\C$ when building our tunnel --- so we do need the generality of Definition (\ref{mcc-def}).

\bigskip

\begin{notation}
  For any set $F\subseteq\Z^{(d')}$, we write $\mathscr{J}_n^F$ for the subspace of $\mathscr{J}_n$ consisting of vectors $\xi : \widehat{G}_n \rightarrow \mathscr{C}$ of $\mathscr{J}_n$ such that $\xi(n) = 0$ whenever $n\in\widehat{G}_n\setminus F$.
\end{notation}

A corollary of Lemma (\ref{Gamma-cv-lemma}) gives us a first form of continuity for our D-norms.

\begin{lemma}\label{CDN-continuity-lemma}
 Assume Hypothesis (\ref{Dirac-hyp}). If $F\subseteq\Z^{(d')}$ be a finite subset of $\Z^{(d')}$, and if, for all $n\in\Nbar$, we set $\varrho_n = \rho_n\otimes\unit_{\mathscr{C}}$, then we conclude:
  \begin{equation*}
    \lim_{n \rightarrow \infty} \opnorm{ \varrho_n\circ\Dirac_n\circ \varrho_n^{\ast} - \Dirac_{\infty} }{}{\mathscr{J}_\infty^F} = 0 \text{,}
  \end{equation*}
  and thus in particular:
  \begin{equation*}
    \lim_{n\rightarrow\infty} \sup_{\substack{\xi \in \mathscr{J}_\infty^F \\ \norm{\xi}{\mathscr{J}_n} \leq 1}} \left| \CDN_n(\varrho_n^{\ast}\xi) - \CDN_{\infty}(\xi) \right| = 0 \text{.}
  \end{equation*}
\end{lemma}

\begin{proof}
We compute, for all $n\in\Nbar$,
\begin{align*}
  0 &\leq \opnorm{\varrho_n\circ \Dirac_n\circ\varrho_n^{\ast} - \Dirac_{\infty}}{}{\mathscr{J}_\infty^F} \\
  &\leq \sum_{j=1}^{d+d'} \opnorm{\left(\rho_n\circ \Gamma_{n,j} \circ\rho_n^{\ast}-\Gamma_{\infty,j}\right)\otimes c(\gamma_j)}{}{\mathscr{J}_\infty^F}\\
    &\leq \sum_{j=1}^{d+d'} \opnorm{\rho_n\Gamma_{n,j}\rho_n^{\ast} - \Gamma_{\infty,j}}{}{\ell^2(\Z^{d'}|F)} \opnorm{c(\gamma_j)}{}{\mathscr{C}}\\
  &\xrightarrow{n \rightarrow \infty} 0 \text{, by Lemma (\ref{Gamma-cv-lemma}).}
\end{align*}

In particular, for all $n\in\Nbar$, since $\varrho_n^\ast$ is, by construction, an isometry from $\mathscr{J}_\infty^F$, we compute:
\begin{align*}
  0 &\leq \sup_{\substack{\xi \in \mathscr{J}_\infty^F \\ \norm{\xi}{\mathscr{J}_\infty} \leq 1}} \left| \CDN_{n}(\varrho_n^{\ast}\xi) - \CDN_{\infty}(\xi) \right| \\
    &\leq \sup_{\substack{\xi \in \mathscr{J}_\infty^F \\ \norm{\xi}{\mathscr{J}_\infty} \leq 1}} \left|\underbracket[1pt]{\norm{\varrho_n^{\ast}\xi}{\mathscr{J}_n}-\norm{\xi}{\mathscr{J}_{\infty}}}_{=0} + \norm{\Dirac_n\varrho_n^\ast \xi}{\mathscr{J}_n} - \norm{\Dirac_{\infty}\xi}{\mathscr{J}_{\infty}}\right|\\
    &= \sup_{\substack{\xi \in \mathscr{J}_\infty^F \\ \norm{\xi}{\mathscr{J}_\infty} \leq 1}} \left|\norm{\varrho_n\circ\Dirac_n\circ\varrho_n^{\ast}\xi}{\mathscr{J}_\infty} - \norm{\Dirac_{\infty}\xi}{\mathscr{J}_\infty}\right|\\
    &\leq \sup_{\substack{\xi \in \mathscr{J}_\infty^F \\ \norm{\xi}{\mathscr{J}_\infty} \leq 1}} \norm{\varrho_n\circ\Dirac_n\circ\varrho_n^{\ast}\xi - \Dirac_{\infty}\xi}{\mathscr{J}_\infty}\\
    &= \opnorm{\varrho_n\circ \Dirac_n\circ\varrho_n^{\ast} - \Dirac_{\infty}}{}{\mathscr{J}_\infty^F}\\
  &\xrightarrow{n\rightarrow\infty} 0 \text{.}
\end{align*}

This concludes our proof.
\end{proof}

The same argument as given in Lemmas (\ref{mvt-discrete-lemma}) and Lemma (\ref{mvt-continuous-lemma}) applies to give the following result.
\begin{lemma}\label{Hilbert-mvt-lemma}
  Assume Hypothesis (\ref{Dirac-hyp}). For all $n\in\Nbar$, and for all $\xi\in\mathscr{J}_n$, the following holds
  \begin{equation*}
    \forall z \in G_n \quad \norm{\xi - \left(v_n^z\right)^\circ(\xi)}{\mathscr{J}_n} \leq 2 \, \mathsf{slen}(z) \CDN(\xi) \text.
  \end{equation*}
\end{lemma}

\begin{proof}
  For all $\xi \in \dom{\Dirac_n}$ and $j\in\{1,\ldots,d+d'\}$, 
  \begin{equation*}
    \norm{(\Gamma_{n,j}\otimes c(\gamma_j))\xi}{\mathscr{J}_n} \leq \norm{\frac{1}{2}\left((\unit_n\otimes c(\gamma_j))\Dirac_n + \Dirac_n (\unit_n\otimes c(\gamma_j)) \right)\xi}{\mathscr{J}_n} \leq \CDN_n(\xi) \text.
  \end{equation*}
  Therefore,
  \begin{align}\label{CDN-norm-eq}
    \norm{(\Gamma_{n,j}\otimes\unit_{\mathscr{C}})\xi}{\mathscr{J}_n}
    &= \norm{(\Gamma_{n,j}\otimes c(\gamma_j))(\unit_n\otimes c(\gamma_j))\xi}{\mathscr{J}_n} \\
    &\leq \norm{(\Gamma_{n,j}\otimes c(\gamma_j))\xi}{\mathscr{J}_n} \leq \CDN_n(\xi) \nonumber \text.
  \end{align}

  Now, the same proof as in Lemmas (\ref{mvt-discrete-lemma}) and (\ref{mvt-continuous-lemma}) shows that, for all $z\in G_n$,
  \begin{align*}
    \norm{(v_n^z-1)^\circ \xi}{\mathscr{J}_n} \leq 2 \, \mathsf{slen}(z) K
  \end{align*}
  where
  \begin{equation*}
    K = \max\left\{ \max_{\substack{j \in \{1\ldots,d'\}\\ k'_n(j)=\infty}} \norm{(\partial_{n,j}\otimes \unit_{\mathscr{C}}) \xi}{\mathscr{J}_n}, \max_{\substack{j\in\{1,\ldots,d'\} \\ k'_n(j)<\infty}}  \norm{\frac{1}{\length{z_{n,j}}}(v_n^{z_{n,j}}-\unit_{\mathscr{J}_n})^\circ \xi}{\mathscr{J}_n} \right\}\text.
  \end{equation*}

  Therefore,
  \begin{align*}
    \norm{(v_n^z-1)^\circ \xi}{} \leq 2 \, \mathsf{slen}(z) \CDN_n(\xi) \text,
  \end{align*}
  as desired.
\end{proof}

From this, and the fact that our Dirac operators are closed, since they are self-adjoint, we now conclude the following lemma.

\begin{notation}
  We use the notation of Hypothesis (\ref{Dirac-hyp}). Let $n\in\Nbar$. For all $f \in C(G_n)$, and for all $\xi \in \mathscr{J}_n$, we set
  \begin{equation*}
    V_n^f \xi = \int_{G_n} f(z) \left(v_n^z\right)^\circ \xi \, d\mu_n(z)
  \end{equation*}
  where $\mu_n$ is the Haar probability measure on $G_n$.
\end{notation}

\begin{lemma}\label{Hilbert-Fejer-lemma}
  For all $\varepsilon > 0$, there exists a positive function $f \in C(\T^{(d')})$, whose Fourier transform is supported on a finite set $S\subseteq\Z^{d'}$, and $N\in\N$, such that, for all $n\geq N$, and for all $\xi \in \dom{\Dirac_n}$,
  \begin{equation*}
    \norm{\xi-V_n^f \xi}{\mathscr{J}_n} \leq \varepsilon \CDN_n(\xi)\text,
  \end{equation*}
  and
  \begin{equation*}
    \CDN_n(V_n^f(\xi)) \leq (1 + \varepsilon) \CDN_n(\xi) \text,
  \end{equation*}
  while $S\subseteq C_n$.
  
  Note that $V_n^f\xi$ is finitely supported, with support included in $q_n(S)$.
\end{lemma}

\begin{proof}
  Let $\varepsilon > 0$. Note that, if $j\in\{1,\ldots,d\}$ and $k_n(j) = \infty$, then for all $z=(z_1,\ldots,z_{d'}) \in G_n$, over the space $\dom{\Dirac_n}$,
  \begin{equation*}
    \left[\Gamma_{n,j}\otimes c(\gamma_j),\left(v_n^z\right)^\circ\right] = \frac{1}{2}\left(v_n^z\right)^\circ \left( ((\overline{z_{f(j)}}-1)U_{n,f(j)} + (z_{f(j)}-1)U_{n,f(j)}^\ast) \partial_{n,j}\right)\otimes c(\gamma_j) \text.
  \end{equation*}
  Therefore, if $\xi \in \dom{\Dirac_n}$, then  
  \begin{align*}
    &\norm{\left[\Gamma_{n,j}\otimes c(\gamma_j), \left(v_n^z\right)^\circ\right]\xi}{\mathscr{J}_n} \\
    &\quad \leq \frac{\dil{z}}{2}\left(\norm{(U_{n,f(j)}\partial_{n,j})\otimes c(\gamma_j) \xi}{\mathscr{J}_n} + \norm{(U_{n,f(j)}^\ast\partial_{n,j})\otimes c(\gamma_j) \xi}{\mathscr{J}_n} \right) \\
    &\quad \leq \dil{z}\left(\norm{X_{n,j}\partial_{n,j}\otimes c(\gamma_j) \xi}{\mathscr{J}_n} + \norm{Y_{n,j}\partial_{n,j}\otimes c(\gamma_j) \xi}{\mathscr{J}_n} \right) \\
    &\quad \leq 2 \; \dil{z} \CDN_n(\xi) \text{ by Eq. \ref{CDN-norm-eq}.}
  \end{align*}
  
  A similar computations show that, if $j \in \{1,\ldots,d\}$ and $k_n(j) = \infty$, then
  \begin{equation*}
    \norm{\left[\Gamma_{n,d+j}\otimes c(\gamma_j),(v_n^z)^\circ\right]\xi}{\mathscr{J}_n} \leq 2 \; \dil{z} \CDN_n(\xi) \text.
  \end{equation*}
  
  Let now $j \in \{1,\ldots,d\}$ and $k_n(j) < \infty$. For all $\xi \in \Hilbert_n$, $z\in G_n$ and $m\in\widehat{G_n}$, we compute:
  \begin{equation*}
    J_n v_n^z\xi(m) = \overline{z^{-m}\xi(-m)} = z^m \overline{\xi(-m)} = v_n^z J_n \xi(m) \text,
  \end{equation*}
  so
  \begin{align*}
    (v_n^z \xi)\cdot U_{n,f(j)}
    &= J_n U_{n,f(j)}^\ast J_n v_n^z \xi \\
    &= J_n U_{n,f(j)}^\ast v_n^z J_n \xi \\
    &= J_n z_{f(j)} v_n^z U_{n,f(j)}^\ast J_n \xi \\
    &= \overline{z_{f(j)}} v_n^z J_n U_{n,f(j)}^\ast J_n \xi = \overline{z_{f(j)}} v_n^z \left(\xi\cdot U_{n,j}\right) \text.
  \end{align*}

  Therefore, we compute:
  \begin{align*}
    \left[\left[U_{n,f(j)},\cdot\right],v_n^z\right]\xi
    &= U_{n,f(j)} v_n \xi - (v_{n}^z \xi) U_{n,f(j)} - v_{n}^z U_{n,f(j)} \xi + v_n^z(\xi\cdot U_{n,f(j)})\\
    &= v_n^z \overline{z_{f(j)}} U_{n,f(j)}\xi - \overline{z_{f(j)}} v_n^z (\xi\cdot U_{n,f(j)}) - v_n^z U_{n,f(j)}\xi + v_n^z (\xi\cdot U_{n,f(j)}) \\
    &= (\overline{z_{f(j)}}-1) v_n^z [U_{n,j},\xi] \text.
  \end{align*}

  A similar computation shows that
  \begin{equation*}
    \left[\left[U_{n,f(j)}^\ast,\cdot\right],v_n^z\right] = (z_{f(j)}-1) v_n^z [U_{n,f(j)}^\ast,\cdot] \text.
  \end{equation*}
  Therefore, with the same notation, we deduce:
  \begin{align*}
    \opnorm{[\Gamma_{n,j}, v_n^z]}{}{\Hilbert_n}
    &= \opnorm{\frac{-\mathsf{fsgn}(j)k_n(j)}{2 i \pi} v_n^z \left((\overline{z_{f(j)}-1})[U_{n,f(j)},\cdot] - (z_{f(j)}-1) [U^\ast_{n,f(j)},\cdot] \right)}{}{\Hilbert_n} \\ 
    &\leq \dil{z}\frac{k_n(j)}{2 \pi} \left(\opnorm{[X_{n,j},\cdot]}{}{\Hilbert_n} + \opnorm{[Y_{n,j},\cdot]}{}{\Hilbert_n} \right) \\
    &\leq \dil{z} \left(\opnorm{[\Gamma_{n,j},\cdot]}{}{\Hilbert_n} + \opnorm{[\Gamma_{n,j+d},\cdot]}{}{\Hilbert_n}\right)  \text.
  \end{align*}
  Thus, for all $\xi \in \mathscr{J}_n$, we conclude:
  \begin{equation*}
    \norm{[\Gamma_{n,j}\otimes c(\gamma_j),(v_n^z)^\circ]\xi}{\mathscr{J}_n} \leq 2 \; \dil{z} \CDN_n(\xi) \text.
  \end{equation*}
  
  Similarly,
  \begin{equation*}
    \norm{[\Gamma_{n,d+j}\otimes c(\gamma_j),(v_n^z)^\circ]\xi}{\mathscr{J}_n} \leq 2 \; \dil{z} \CDN_n(\xi) \text.
  \end{equation*}

  Last, by Hypothesis (\ref{Dirac-hyp}), $[\Gamma_{n,j},v_n^z] = 0$ if $j \in \{2d+1,\ldots,d+d'\}$.
  
  Using Lemma (\ref{Fejer-lemma}), there exists $f \in C(\T^{(d')})$, with $f \geq 0$, and, for all $n \geq N$,
  \begin{equation*}
    \int_{G_n} f(z) \dil{z} \, d\mu_n(z)
    < \frac{\varepsilon}{2(d+d')} \text{ and }\int_{G_n} f(z)\mathsf{slen}(z)  \, d\mu_n(z) < \frac{\varepsilon}{2} \text.
  \end{equation*}
  
  Note that, for all $\xi \in \mathscr{J}_n$, the vector $V_n^f \xi$ is finitely supported, so in particular, $V_n^f\xi \in \dom{\Dirac_n}$. As $\Dirac_n$ is closed, we then easily deduce that
  \begin{equation*}
    \Dirac_n V_n^f\xi = \int_{G_n} f(z) \Dirac_n V_n^z\xi \, d\mu_n(z) \text.
  \end{equation*}

  Therefore, if $\xi \in \dom{\Dirac_n}$, then
  \begin{align*}
    \norm{V_n^f\Dirac_n\xi - \Dirac_nV_n^f\xi}{\mathscr{J}_n}
    &\leq \sum_{j=1}^{d+d'} \int_{G_n} f(z) \norm{\left[\Gamma_{n,j},v_n^z\right]\xi}{\mathscr{J}_n} \, d\mu_n(z) \\
    &\leq \sum_{j=1}^{d+d'} \int_{G_n} f(z) \cdot 2\dil{z} \CDN_n(\xi) \, d\mu_n(z) \\
    &\leq \sum_{j=1}^{d+d'} \frac{\varepsilon}{d+d'} \CDN_n(\xi) = \varepsilon \CDN_n(\xi) \text.
  \end{align*}

  \medskip

  Therefore, for all $n\geq N$ and $\xi \in \dom{\Dirac_n}$,
  \begin{align*}
    \CDN_n(V_n^f\xi)
    &= \norm{V_n^f\xi}{\mathscr{J}_n} + \norm{\Dirac_nV_n^f\xi}{\mathscr{J}_n} \\
    &\leq \norm{\xi}{\mathscr{J}_n} + \norm{V_n^f\Dirac_n\xi}{\mathscr{J}_n} + \norm{[\Dirac_n,V_n^f]\xi}{\mathscr{J}_n} \\
    &\leq \norm{\xi}{\mathscr{J}_n} + \norm{\Dirac_n\xi}{\mathscr{J}_n} + \varepsilon \CDN_n(\xi) \\
    &= (1 + \varepsilon) \CDN_n(\xi) \text.
  \end{align*}

  Moreover, as in Corollary (\ref{mvt-cor}), using Lemma (\ref{Hilbert-mvt-lemma}),
  \begin{align*}
    \norm{\xi - V_n^f \xi}{\mathscr{J}_n}
    &\leq \int_{G_n} f(z) \;2\mathsf{slen}(z)  \, d\mu_n(z) \CDN_n(\xi) \\
    &\leq \varepsilon \CDN_n(\xi) \text.
  \end{align*}
  
  This concludes our proof.
\end{proof}

\bigskip

We now have the tools needed to conclude our proof of convergence for spectral triples. We begin with our modular tunnels.

\begin{theorem}\label{mod-conv-thm}
  For all $\varepsilon > 0$, there exists $N\in\N$ such that, if $n \geq N$, and if we set
  \begin{enumerate}
  \item for all $\xi\in\dom{\Dirac_n}$ and $\eta\in\dom{\Dirac_\infty}$,
    \begin{equation*}
      \TDN_n(\xi,\eta) = \max\left\{ \CDN_{n}(\xi), \CDN_{\infty}(\eta), \frac{1}{\varepsilon} \norm{\varrho_n(\xi) - \eta }{\mathscr{J}_{\infty}} \right\} \text{,}
    \end{equation*}
  \item $\mathsf{Q}(z,w) = \frac{1}{\varepsilon}|z-w|$ for all $z,w \in \C$,
  \item $Z_n : (\xi,\eta)\in\mathscr{J}_n\oplus\mathscr{J}_{\infty} \mapsto \xi$ and $Z_{\infty} : (\xi,\eta)\in\mathscr{J}_n\oplus\mathscr{J}_{\infty} \mapsto \eta$,
  \item $x_n : (z,w) \in \C\oplus\C \mapsto z$ and $x_{\infty} : (z,w) \mapsto w$,
  \end{enumerate}
  then
  \begin{equation}\label{mod-tunnel-eq}
    \tau_n^{\mathrm{mod}} = \left(  \mathscr{J}_n \oplus \mathscr{J}_{\infty}, \TDN_n, \C\oplus\C, \mathsf{Q}, (Z_n,x_n), (Z_{\infty},x_{\infty}) \right)
  \end{equation}
  is a modular tunnel from $(\mathscr{J}_n,\CDN_{n},\C,0)$ to $(\mathscr{C}_{\infty},\CDN_{\infty},\C,0)$, with extent at most $\varepsilon$.
\end{theorem}

\begin{proof}
  Let $\varepsilon \in (0,1)$. Let $f \in C(\T^{(d')})$, $S\subseteq\ell^1(\Z^{(d')})$ and $N\in\N$ be given by Lemma (\ref{Hilbert-Fejer-lemma}), for $\frac{\varepsilon}{4}$ in place of $\varepsilon$.
  
  By Lemma (\ref{CDN-continuity-lemma}), there exists $N'\in\N$ such that if $n\geq N'$, and if $\xi \in \mathscr{J}_\infty^S$, then
  \begin{equation*}
    \left|\CDN_n(\varrho_n^{\ast}\xi) - \CDN_{\infty}(\xi)\right| \leq \frac{\varepsilon}{4} \norm{\xi}{\mathscr{J}_\infty} \text.
  \end{equation*}

  Let $n \geq \max\{N,N'\}$. Let $\eta \in \dom{\Dirac_n}$ with $\CDN_k(\eta) \leq 1$. Let
  \begin{equation*}
    \chi = \varrho_n(V_n^f(\eta)) \in \dom{\Dirac_\infty} \text.
  \end{equation*}

  By construction, using Lemma (\ref{Hilbert-Fejer-lemma}), since $\rho_n^\ast\rho = 1$,
  \begin{align*}
    \CDN_\infty(\chi)
    &\leq \left(1+\frac{\varepsilon}{4}\right)\CDN_n(\rho_n^\ast\rho_n V_n^f\eta) \\
    &= \left(1+\frac{\varepsilon}{4}\right)\CDN_n(V_n^f\eta) \leq \left(1+\frac{\varepsilon}{4}\right)^2
  \end{align*}
  and
  \begin{align*}
    \norm{\varrho_n(\eta) - \chi}{\mathscr{J}_{\infty}}
    &= \norm{\varrho_n(\eta - V_n^f(\eta))}{\mathscr{J}_{\infty}} \\
    &= \norm{\eta - V_n^f(\eta)}{\mathscr{J}_n} \leq \frac{\varepsilon}{4} \text{.}
  \end{align*}

  So
  \begin{align*}
    \norm{\varrho_n(\eta) - \frac{1}{(1+\frac{\varepsilon}{4})^2}\chi}{\mathscr{J}_{\infty}}
    &\leq \norm{\varrho_n(\eta) - \chi}{\mathscr{J}_{\infty}} + \left(\frac{\varepsilon}{2}+\frac{\varepsilon^2}{16}\right)\norm{\chi}{\mathscr{J}_{\infty}}\\
    &\leq \varepsilon \text.
  \end{align*}
  
  Therefore, $\TDN_n\left(\frac{1}{\left(1+\frac{\varepsilon}{4}\right)^2}\chi,\eta\right) \leq 1$. Since $\TDN_n(\xi,\eta)\geq\CDN_n(\eta)$ for all $\xi \in \mathscr{J}_{\infty}$ by construction, we conclude that $\CDN_n$ is the quotient of $\TDN_n$ by the canonical surjection $Y_n$ from $\mathscr{J}_{n}\oplus\mathscr{J}_{\infty}$ onto $\mathscr{J}_n$.

  \bigskip

  A similar reasoning applies to show that $\CDN_{\infty}$ is the quotient of $\TDN_n$ by the canonical surjection $Y_{\infty}$ from $\mathscr{J}_{n}\oplus\mathscr{J}_{\infty}$ onto $\mathscr{J}_{\infty}$.

  \bigskip

  Now, by construction, the closed unit ball of $\TDN_n$ is a closed subset of the product of the closed unit ball of $\CDN_n$ and $\CDN_{\infty}$, both of which are compact, and thus, the closed unit ball of $\TDN_n$ is compact as well.

  \bigskip
  
  For all $(\xi,\eta) \in \mathscr{J}_n\oplus\mathscr{J}_{\infty}$, we compute:
  \begin{equation*}
    \TDN_n(\xi,\eta) \geq \max\{ \CDN_n(\xi), \CDN_{\infty}(\eta) \} \geq \max\{ \norm{\xi}{\mathscr{J}_n}, \norm{\eta}{\mathscr{J}_\infty} \} = \norm{(\xi,\eta)}{\mathscr{J}_n\oplus\mathscr{J}_\infty} \text{.}
  \end{equation*}

  \bigskip

  We now consider $\mathscr{J}_n\oplus\mathscr{J}_{\infty}$ as a Hilbert module over $\C\oplus\C$, with inner product:
  \begin{equation*}
    \inner{(\xi,\eta)}{(\xi',\eta')}{\C\oplus\C} = \left( \inner{\xi}{\xi'}{\mathscr{J}_n}, \inner{\eta}{\eta'}{\mathscr{J}_\infty} \right)
  \end{equation*}
  and
  \begin{equation*}
    (\xi,\eta)(z,w) = (z\xi,w\eta)
  \end{equation*}
  for all $(z,w)\in\C\oplus\C$, $\xi,\xi'\in\mathscr{J}_n$ and $\eta,\eta'\in\mathscr{J}_{\infty}$.

  \bigskip

  It is easy to check, using Cauchy-Schwarz, that:
  \begin{align*}
    \mathsf{Q}(\inner{(\xi,\eta)}{(\xi',\eta')}{\C\oplus\C})
    &=\frac{1}{\varepsilon}\left|\inner{\xi}{\xi'}{\mathscr{J}_n} - \inner{\eta}{\eta'}{\mathscr{J}_{\infty}}\right|\\
    &=\frac{1}{\varepsilon}\left( \inner{\varrho_n\xi}{\varrho_n\xi'}{\mathscr{J}_{\infty}} - \inner{\eta}{\eta'}{\mathscr{J}_{\infty}}  \right) \\
    &\leq \frac{1}{\varepsilon}\left( \norm{\varrho_n\xi-\eta}{\mathscr{J}_{\infty^d}} \norm{\xi'}{\mathscr{J}_n} + \norm{\eta}{\mathscr{C}_{\infty}}\norm{\varrho_n\xi'-\eta'}{\mathscr{J}_{\infty}}\right) \\
    &\leq 2 \TDN_n(\xi,\eta)^2 \TDN_n(\xi',\eta')^2 \text{.}
  \end{align*}
  
  \bigskip

  It is then trivial to check that the canonical surjections $x_n$ and $x_{\infty}$ are quantum isometries from $(\C\oplus\C,\mathsf{Q})$ onto $(\C,0)$. Thus we have established that Expression (\ref{mod-tunnel-eq}) does define a modular tunnel, as claimed. The extent of our tunnel is no more than $\varepsilon$ (it is no more than the distance between the two points in the spectrum of $\C\oplus\C$ for the metric induced by the Lipschitz seminorm $\mathsf{Q}$).
\end{proof}

\begin{remark}
  Theorem (\ref{mod-conv-thm}) implies that:
  \begin{equation*}
    \lim_{n \rightarrow \infty} \dmodpropinquity{}\left( (\mathscr{J}_n, \CDN_n, \C, 0), (\mathscr{J}_{\infty},\CDN_{\infty},\C,0) \right) = 0
  \end{equation*}
  where $\dmodpropinquity{}$ is the modular propinquity \cite{Latremoliere16c,Latremoliere18c}.
\end{remark}

\bigskip

We now discuss how close the metrical C*-correspondence in Notation (\ref{mcc-notation}) are. All which is needed is to check that the $\C\oplus\C$ modules in the modular tunnels of Theorem (\ref{mod-conv-thm}) are indeed modules over the appropriate fuzzy or quantum torus.

\begin{theorem}\label{metrical-cv-thm}
  Assume Hypothesis (\ref{L-seminorm-hyp}). We conclude:
  \begin{equation*}
    \lim_{n\rightarrow\infty} \dmetpropinquity{4}\big( \mcc{\A_n}{\mathscr{J}_n}{\Dirac_n}, \mcc{\A_\infty}{\mathscr{J}_{\infty}}{\Dirac_{\infty}} \big) = 0\text{.}
  \end{equation*}
\end{theorem}

\begin{proof}
  Let $\varepsilon \in (0,1)$. Let $N_1\in\N$, $f \in C(\T^{(d')})$ and $S\subseteq\Z^{d'}$ be given by Lemma (\ref{Hilbert-Fejer-lemma}), Let $N_2 \in \N$ be given by Theorem (\ref{mod-conv-thm}). Let $N = \max\{N_1,N_2\}$. For each $n\geq N$, let $\tau_n$ be the tunnel given by Theorem (\ref{mod-conv-thm}), with $R_N$ adjusted so that $R_N$ restricted to $\mathscr{J}_n^S$ is the identity.
  
  Let $a\in\dom{\Lip_n}$ and $b\in\dom{\Lip_\infty}$. By Expression (\ref{tunnel-commutator-eq}), we have
  \begin{equation*}
    \opnorm{\left[b^\circ,R_N\right]}{}{\mathscr{J}_\infty} \leq \varepsilon \Lip_\infty(b) \text.
  \end{equation*}

  For all $n\in\Nbar$, we set $\varrho_n=\rho_n\otimes\unit_{\mathscr{C}}$.
  
  Let $\xi \in \mathscr{J}_n^{S}$ and $\eta\in\mathscr{J}_\infty^{S}$. Since $R_N\mathscr{J}_\infty \subseteq \mathscr{J}_\infty^{S}$, we then check that:
\begin{align*}
  \norm{\theta_n(a)\varrho_n \xi - \theta_{\infty}(b)\eta}{\mathscr{J}_{\infty}}
  &= \norm{\theta_n(a) R_N \varrho_n\xi - b^\circ R_N \eta}{\mathscr{J}_{\infty}} \\
  &\leq \norm{\theta_n(a) R_N \varrho_n \xi - R_N b^\circ \eta}{\mathscr{J}_{\infty}} + \varepsilon \Lip_{\infty}(b) \\
  &\leq \norm{\theta_{n}(a) R_N - R_N b^\circ}{\mathscr{J}_\infty}\norm{\eta}{\mathscr{J}_\infty} \\
  &\quad + \norm{a}{\A_\infty} \norm{\varrho_n\xi - \eta}{\mathscr{J}_\infty} + \varepsilon \Lip_{\infty}(b) \\
  &\leq \varepsilon \left((1 + \varepsilon)\TLip_n(a,b) + \norm{a,b}{\A_n\oplus\A_\infty}\right)\TDN_n(\xi,\eta) \text{.}
\end{align*}

Now, let $\xi \in \dom{\Dirac_n}$ and $\eta\in\dom{\Dirac_\infty}$. We compute
\begin{align*}
  \norm{\theta_n(a)\varrho_n \xi - \theta_{\infty}(b)\eta}{\mathscr{J}_{\infty}}
  &\leq \norm{a}{\A_n}\norm{\xi-V_n^f \xi}{\mathscr{J}_n} + \norm{b}{\A_\infty}\norm{\eta-V_\infty^f\eta}{\mathscr{J}_\infty} \\
  &\quad + \norm{\theta_n(a)\varrho_n V_n^f \xi - \theta_{\infty}(b) V_n^f \eta}{\mathscr{J}_{\infty}} \\
  &\leq \varepsilon \left( \norm{a}{\A_n} + \norm{b}{\A_\infty} \right) \TDN_n(\xi,\eta) \\
  &\quad + \norm{\theta_n(a)\rho_n V_n^f \xi - \theta_{\infty}(b) V_n^f \eta}{\mathscr{J}_{\infty}} \\
  &\leq 2 \varepsilon \norm{(a,b)}{\A_n\oplus\A_\infty} \TDN_n(\xi,\eta) \\
  &\quad + \varepsilon \left((1 + \varepsilon)\TLip_n(a,b) + \norm{a,b}{\A_n\oplus\A_\infty}\right)(1+\varepsilon)\TDN_n(\xi,\eta) \\
  &\leq 4 \varepsilon \left(\norm{(a,b)}{\A_n\oplus\A_\infty} + \TLip_n(a,b) \right) \TDN_n(\xi,\eta) \text{.}
\end{align*}

From this, we easily conclude that
\begin{equation*}
  \TDN_n(a\xi,b\eta) \leq 4\left(\norm{(a,b)}{\A_n\oplus\A_\infty} + \TLip_n(a,b) \right) \TDN_n(\xi,\eta) \text.
\end{equation*}

Thus the extent of the metrical tunnel $\tau_n^{\mathrm{met}} = (\tau_n^{\mathrm{mod}},\tau_n)$ is the maximum of the modular tunnel $\tau^{\mathrm{mod}}_n$ and the tunnel $\tau_n$, which is no more than $\varepsilon > 0$. This completes our proof.
\end{proof}

We now establish our main theorem, proving that the sequence of spectral triples constructed in Hypothesis (\ref{L-seminorm-hyp}) is convergent, and giving its limit.

\begin{theorem}
  If Hypothesis (\ref{L-seminorm-hyp}) is assumed, then
  \begin{equation*}
    \lim_{n\rightarrow\infty} \spectralpropinquity{}\big( (\A_n,\mathscr{J}_n,\Dirac_n), (\A_\infty, \mathscr{J}_\infty, \Dirac_\infty) \big) = 0 \text{.}
  \end{equation*}
\end{theorem}

\begin{proof}
  For all $n\in\Nbar$ and $t\in\R$, we set
  \begin{equation*}
    S_n^t = \exp(i t \Dirac_n) \text.
  \end{equation*}

  Let $\varrho_n = \rho_n\otimes\unit_{\mathscr{C}}$ for all $n\in\Nbar$.

  \medskip
  
  Now, let $\varepsilon \in (0,1)$. Using Lemma (\ref{Fejer-lemma}), Lemma (\ref{CDN-continuity-lemma}), and Lemma (\ref{Hilbert-Fejer-lemma}), there exists a function $f \in C(\T^{(d')})$, whose Fourier transform is an element in $\ell^1(\Z^{d'})$ supported on some finite set $S$, and there exists $N_1\in \N$ such that, if $n\geq N_1$, then $S \subseteq C_n$, and
  \begin{equation*}
    \opnorm{\varrho_n \Dirac_n \varrho_n^{-1} - \Dirac_{\infty}}{}{\ell^2(S)\otimes\mathscr{C}} \leq \varepsilon^2\text,
  \end{equation*}
  while, for all $\xi \in \dom{\Dirac_n}$,
  \begin{equation*}
    \norm{\xi- V_n^f\xi}{\mathscr{J}_n} \leq \varepsilon \CDN_n(\xi) \text{ while }\CDN_n(V_n^f\xi) \leq (1 + \varepsilon) \CDN_n(\xi) \text.
  \end{equation*}

  Let $N_2 \in \N$ be given, for our chosen $\varepsilon > 0$, as in the proof of Theorem (\ref{metrical-cv-thm}), adjusting $R_{N_2}$ as needed so that $R_{N_2}$ restricted to $\ell^2(\Z^d|S)$ is the identity. We use the notation employed in that proof, as well.
  
  Let $n\geq N = \max\{N_1,N_2\}$. Let $\xi \in \dom{\Dirac_\infty}$. Let $\eta = \rho_n(V_n^f(\xi))$. We compute:
  \begin{align*}
    0
    &\leq\sup_{\substack{\zeta=(\zeta_1,\zeta_2) \in \mathscr{J}_n\oplus\mathscr{J}_\infty \\ \TDN_n(\zeta)\leq 1}} \left| \inner{S_n^t \xi}{Z_n(\zeta)}{\mathscr{J}_n} - \inner{S_{\infty}^t\eta}{Z_\infty(\zeta)}{\mathscr{J}_\infty} \right| \\
    &=\sup_{\substack{\zeta=(\zeta_1,\zeta_2) \in \mathscr{J}_n\oplus\mathscr{J}_\infty \\ \TDN_n(\zeta)\leq 1}} \left| \inner{S_n^t \xi}{\zeta_1}{\mathscr{J}_n} - \inner{S_{\infty}^t\eta}{\zeta_2}{\mathscr{J}_\infty} \right| \\
    &= \sup_{\substack{\zeta=(\zeta_1,\zeta_2) \in \mathscr{J}_n\oplus\mathscr{J}_\infty \\ \TDN_n(\zeta)\leq 1}} \left| \inner{S_n^t\xi - S_n^t V_n^f(\xi)}{\zeta_1}{\mathscr{J}_n} + \inner{S_n^tV_n^f(\xi)}{\zeta_1}{\mathscr{J}_n} - \inner{S_{\infty}^t\eta}{\zeta_2}{\mathscr{J}_\infty} \right| \\
    &\leq \norm{\xi-V_n^f(\xi)}{\mathscr{J}_n} + \sup_{\substack{\zeta=(\zeta_1,\zeta_2) \in \mathscr{J}_n\oplus\mathscr{J}_\infty \\ \TDN_n(\zeta)\leq 1}}  \left| \inner{\rho_n S_n^t V_n^f(\xi)}{\rho_n \zeta_1}{\mathscr{J}_\infty} - \inner{S_{\infty}^t\eta}{\zeta_2}{\mathscr{J}_\infty} \right| \\
    &= \varepsilon + \sup_{\substack{\zeta=(\zeta_1,\zeta_2) \in \mathscr{J}_n\oplus\mathscr{J}_\infty \\ \TDN_n(\zeta)\leq 1}} \left|\inner{\rho_n S_n^t \rho_n^{\ast} \eta - S_{\infty}^t\eta}{\rho_n\zeta_1}{\mathscr{J}_\infty} + \inner{S_{\infty}^t\eta}{\zeta_2 - \rho_n\zeta_1}{\mathscr{J}_\infty}\right| \\
    &\leq 2\varepsilon + \norm{\left( \exp(i t \rho_n \Dirac_n \rho_n^{\ast}) - \exp(i t \Dirac_{\infty})\right) \eta}{\mathscr{J}_\infty} \text{.}
  \end{align*}

  Note that, by construction, $S_{\infty}^t\mathscr{J}_\infty^S = \mathscr{J}_\infty^S$, for all $t\in\R$. We then compute, for all $t \leq \frac{1}{\varepsilon}$, using \cite[Ch. 9, eq. (2.3), p. 497]{Kato}
  \begin{align*}
    0
    &\leq \norm{\left(\exp(i t \varrho_n \Dirac_n \varrho_n^{\ast}) - \exp(i t \Dirac_{\infty})\right) \eta }{\mathscr{J}_\infty^S} \\
    &\leq \int_0^t \norm{\varrho_n S_n^{t-s}\varrho_n^{\ast} \cdot \left(\varrho_n \Dirac_n \varrho_n^{\ast} - \Dirac_\infty \right) S_{\infty}^s \eta}{\mathscr{J}_\infty^S} \, ds \\
    &\leq \int_0^t \opnorm{\varrho_n \Dirac_n \varrho_n^{\ast} - \Dirac_{\infty}}{}{\mathscr{J}_\infty^S} \, ds \\
    &\leq t \opnorm{\varrho_n \Dirac_n \varrho_n^{\ast} - \Dirac_{\infty}}{}{\mathscr{J}_\infty^S} \\
    &\leq t \varepsilon^2 \leq \frac{\varepsilon^2}{\varepsilon} = \varepsilon \text{.}
  \end{align*}

  Therefore, together with Theorem (\ref{metrical-cv-thm}), we conclude that the magnitude of $\tau_n$ is at most $\varepsilon$, so, for all $n\geq N$, we have
  \begin{equation*}
    \spectralpropinquity{}((\A_n,\mathscr{J}_n,\Dirac_n),(\A_\infty,\mathscr{J}_\infty,\Dirac_\infty)) \leq \varepsilon \text.
  \end{equation*}
  This concludes our proof.
\end{proof}

\vfill

\end{document}